\documentclass[]{amsart}

\makeatletter
\@addtoreset{equation}{section}

\makeatother

\newtheorem{theorem}{Theorem}[section]
\newtheorem{proposition}[theorem]{Proposition}
\newtheorem{corollary}[theorem]{Corollary}
\newtheorem{lemma}[theorem]{Lemma}
\theoremstyle{definition}
\newtheorem{definition}[theorem]{Definition}

\theoremstyle{remark}
\newtheorem{remark}[theorem]{Remark}

\DeclareMathOperator{\sech}{sech}
\DeclareMathOperator{\cn}{cn}
\DeclareMathOperator{\dn}{dn}
\DeclareMathOperator{\sn}{sn}
\DeclareMathOperator{\am}{am}
\DeclareMathOperator{\sgn}{sign}

\usepackage{braket} % \Set{} is available
\usepackage{amssymb} % \varkappa is available 
\usepackage{color}
\usepackage[pdftex]{graphicx}
\usepackage{mathrsfs} % \mathscr{} is available 
\usepackage{url}

\newcommand{\lm}{\lambda}
\newcommand{\R}{\mathbf{R}}
\newcommand{\Z}{\mathbf{Z}}

\newcommand{\N}{\mathbf{N}}
\newcommand{\vk}{\kappa}
\newcommand{\vp}{\varphi}

\newcommand{\va}{\alpha}
\newcommand{\vb}{\beta}
\newcommand{\vc}{\gamma}

\newcommand{\ve}{\varepsilon}

\newcommand{\B}{\mathcal{B}}
\newcommand{\K}{\mathrm{K}}
\newcommand{\E}{\mathrm{E}}

\newcommand{\amcn}{\am_{1,p}}
\newcommand{\amdn}{\am_{2,p}}
\newcommand{\Ecn}{\mathrm{E}_{1,p}}
\newcommand{\Fcn}{\mathrm{F}_{1,p}}
\newcommand{\Kcn}{\mathrm{K}_{1,p}}
\newcommand{\Edn}{\mathrm{E}_{2,p}}
\newcommand{\Fdn}{\mathrm{F}_{2,p}}
\newcommand{\Kdn}{\mathrm{K}_{2,p}}
\begin{document}

\title{Complete classification of planar $p$-elasticae}
%\title[Planar $p$-elasticae]{Classification theorems and explicit formulae for planar $p$-elasticae}
\author{Tatsuya Miura}
\address[T.~Miura]{Department of Mathematics, Tokyo Institute of Technology, 2-12-1 Ookayama, Meguro-ku, Tokyo 152-8551, Japan}
\email{miura@math.titech.ac.jp}

\author{Kensuke Yoshizawa}
\address[K.~Yoshizawa]{Institute of Mathematics for Industry, Kyushu University, 744 Motooka, Nishi-ku, Fukuoka 819-0395, Japan; 
present address: Faculty of Education, Nagasaki University, 1-14 Bunkyo-machi, Nagasaki, 852-8521, Japan}
\email{k-yoshizaw@nagasaki-u.ac.jp%k-yoshizawa@imi.kyushu-u.ac.jp
}


\date{\today}

\begin{abstract}
 Euler's elastica is defined by a critical point of the total squared curvature under the fixed length constraint, and its $L^p$-counterpart is called $p$-elastica.
 In this paper we completely classify all $p$-elasticae in the plane and obtain their explicit formulae as well as optimal regularity.
 To this end we introduce new types of $p$-elliptic functions which streamline the whole argument and result.
 As an application we also classify all closed planar $p$-elasticae.
\end{abstract}

\maketitle

\setcounter{tocdepth}{1}
\tableofcontents

\section{Introduction}

A central topic in geometric analysis is concerned with variational problems involving curvature energies.
In particular the bending energy for curves, defined by the $L^2$-norm of the curvature, is first investigated by Euler in the 18th century and by now a standard mathematical model of thin elastic objects; see e.g.\ \cite{Tru83,Lev} for the history, \cite{Love,LLbook,APbook} for theoretical and physical contexts, and \cite{Sac08_2,Sin,Miura20} for relatively recent mathematical treatises.

During the last decade or two, there has been a burst of attention to the $L^p$-counterpart of the bending energy, so-called \emph{$p$-bending energy} (also known as $p$-elastic energy).
For $p\in(1,\infty)$ and for immersed curves $\gamma$ in $\R^2$, the $p$-bending energy is defined by
\[
\mathcal{B}_p[\gamma] := \int_{\gamma}|k|^p\, ds,\]
where $s$ denotes the arclength parameter, and $k$ denotes the signed curvature.
The study of the $p$-bending energy branches into several directions; the study of critical points \cite{AGM,nabe14,LP20,SW20,LP22}, ambient approaches \cite{BDP93, BM04,BM07, MN13, AM17,Poz20}, gradient flows \cite{NP20, OPW20, BHV, BVH, Poz22, OW23} and others \cite{FKN18,DMOY}.
This energy also emerges in several contexts; for example, applications to image processing \cite{MM98, AM03, DFLM}, the study of linear Weingarten surfaces of revolution \cite{LP20}, mathematical models of soap films with vertical potentials or dynamics of plasma particles \cite{LP22}, and characterizations of translating flows by powers of curvature \cite{Pam24}.
See also relevant studies for the case $p\in(0,1)$ \cite{GPT23, MP23_JNS, MP23_JMAA}, and for the case $p=\infty$ \cite{Mos22, GM23}; the total absolute curvature $p=1$ also requires variationally special treatment.

The most fundamental problem of the $p$-bending energy is an extension of Euler's elastica problem, namely to understand the critical points of the energy
\[
\mathcal{B}_p[\gamma] + \lambda \mathcal{L}[\gamma],
\]
where $\mathcal{L}$ denotes the length and $\lambda\in\R$ is a given constant, often playing the role of a multiplier.
We call such a critical point \emph{$p$-elastica}, since the classical case $p=2$ corresponds to Euler's \emph{elastica}.

\begin{definition}[$p$-Elastica]\label{def:p-elastica}
Let $p\in(1,\infty)$ and $\gamma\in W^{2,p}(0,1;\R^2)$ be an immersed curve.
Then $\gamma$ is called \emph{$p$-elastica} if there is $\lambda\in\R$ such that, for any $\eta\in C^\infty_{\rm c}(0,1;\R^2)$,
\begin{align*}%\label{eq:1st-derivative-linear}
\frac{d}{d\varepsilon}\Big( \mathcal{B}_p[\gamma+\varepsilon\eta] + \lambda\mathcal{L}[\gamma+\varepsilon\eta] \Big)\Big|_{\varepsilon=0}=0.
\end{align*}
\end{definition}

It is well known that for $p=2$ all planar elasticae are smooth (by bootstrap) and classified into some classes (essentially due to Euler), and have explicit parametrizations in terms of Jacobian elliptic integrals and functions (at least since Saalsch\"utz's study \cite{Saa1880}, see also \cite{Love,DHMV,MR23}).
%\cite{Love} (see also \cite{DHMV,MR23}).

In contrast, for $p\neq2$, not only the full classification but also the optimal regularity of planar $p$-elasticae has been left open up to now.
One difficulty is that the leading term in the Euler--Lagrange equation for curvature is singular ($p<2$) or degenerate ($p>2$) like the $p$-Laplacian (cf.\ \cite{ATU_17, LL_17, ATU_18, DZZ_20}); in fact, the curvature distributionally satisfies the equation
\begin{align}\label{Teq:sw20-1.2}
%p(p-1)|k|^{p-2} \partial_s^2 k +p(p-1)(p-2) |k|^{p-4}k (\partial_s k)^2 
p(|k|^{p-2}k)_{ss}+(p-1)|k|^pk - \lambda k =0.
\end{align}
Thus $k$ may not be of class $C^2$ and equation \eqref{Teq:sw20-1.2} may not be understood in the classical sense.
Another point is that, although several kinds of ``$p$-elliptic functions'' have already been introduced (e.g.\ in \cite{Takeuchi12,Takeuchi_RIMS}), those functions would not be directly applicable to our problem since they would not solve the Euler--Lagrange equations for $p$-elasticae.

Among some known results on planar $p$-elasticae, we mention a remarkable result of Watanabe \cite{nabe14} that gives a family of examples of explicit solutions.
In particular, in the degenerate case $p>2$, the family includes a new class of solutions whose curvatures have non-discrete zero sets, called \emph{flat-core} solutions; more precisely, such solutions are obtained by concatenating certain loops and segments with some arbitrariness due to non-uniqueness induced by degeneracy, cf.\ Figure \ref{fig:classifyflatcore} below.
However, the solutions obtained in \cite{nabe14} are not exhaustive, and also their representations still involve rather complicated integrals.
Some (partial) classification results for related problems have also been obtained in \cite{LP20,SW20,LP22}.

The purpose of this paper is to completely solve the $p$-elastica problem in the plane.
More precisely, (i) we classify all planar $p$-elasticae, and also (ii) obtain their optimal regularity, (iii) starting from the natural energy space of Sobolev class $W^{2,p}$.
In addition, (iv) we introduce new $p$-elliptic functions and apply them to give concise explicit formulae.
All the points (i)--(iv) are new to the authors' knowledge.
Our $p$-elliptic functions not only streamline the whole proof but also rationalize how our results extend the classical ones, in particular giving the first complete extension of the Saalsch\"utz-type formulae \cite{Saa1880} from $p=2$ to $p\in(1,\infty)$.

We first state our main classification theorems in terms of $p$-elliptic integrals and functions, although all the precise definitions are postponed to Definitions \ref{def:K_p}, \ref{def:E_p}, \ref{def:cndn}, \ref{def:sech}, and \ref{def:tan_p}.
For clarity we distinguish the cases $p\leq2$ and $p>2$.
In the case of $p\leq 2$, our classification is fairly parallel to the classical one, cf.\ Figures \ref{fig:classifywave}, \ref{fig:classifyborder}, \ref{fig:classifyorbitlike}.
In particular there are no flat-core solutions.

\begin{theorem}[Explicit formulae for planar $p$-elasticae: $p\leq2$]\label{thm:Formulae_1}
  Let $p\in(1,2]$ and $\gamma$ be a $p$-elastica in $\R^2$.
  Then up to similarity (i.e., translation, rotation, reflection, and dilation) and reparametrization, the curve $\gamma$ is represented by $\gamma(s)=\gamma_*(s+s_0)$ with some $s_0\in\R$, where $\gamma_*:\R\to\R^2$ is one of the following five arclength parametrizations:
   \begin{itemize}
     \item \textup{(Case I --- Linear $p$-elastica)}
     $\gamma_\ell(s)=(s,0)$, where $k_\ell\equiv0$.
     \item \textup{(Case II --- Wavelike $p$-elastica)}
     For some $q \in (0,1)$,
     \begin{equation}\label{eq:EP2}
       \gamma_w(s,q) =
       \begin{pmatrix}
       2 \E_{1,p}(\am_{1,p}(s,q),q )-  s  \\
       -q\frac{p}{p-1}|\cn_p(s,q)|^{p-2}\cn_p(s,q)
       \end{pmatrix}.
     \end{equation}
     In this case, $\theta_w(s)=2\arcsin(q\sn_p(s,q))$ and $k_w(s) = 2q\cn_p(s,q).$
     \item \textup{(Case III --- Borderline $p$-elastica)}
     \begin{equation}\label{eq:EP3-1}
       \gamma_b(s) =
       \begin{pmatrix}
       2 \tanh_p{s} - s  \\
       - \frac{p}{p-1}(\sech_p{s})^{p-1}
       \end{pmatrix}.
     \end{equation}
     In this case, $\theta_b(s)=2\am_{1,p}(s,1)=2\am_{2,p}(s,1)$ and $k_b(s) = 2\sech_p{s}.$
     \item \textup{(Case IV --- Orbitlike $p$-elastica)}
     For some $q \in (0,1)$,
         \begin{equation}\label{eq:EP4}
       \gamma_o(s,q) = \frac{1}{q^2}
       \begin{pmatrix}
       2 \E_{2,\frac{p}{p-1}}(\am_{2,p}(s,q),q)  + (q^2-2)s \\
       - \frac{p}{p-1}\dn_p(s,q)^{p-1}
       \end{pmatrix}.
     \end{equation}
     In this case, $\theta_o(s)=2\am_{2,p}(s,q)$ and $k_o(s) = 2 \dn_p(s,q).$
       \item \textup{(Case V --- Circular $p$-elastica)}
     $\gamma_c(s)=(\cos{s},\sin{s})$, where $k_c\equiv1$.
   \end{itemize}
   Here $\theta_*$ denotes the tangential angle $\partial_s\gamma_*=(\cos\theta_*,\sin\theta_*)$, and $k_*$ the (counterclockwise) signed curvature $k_*=\partial_s\theta_*$.
\end{theorem}

On the other hand, if $p>2$, the degeneracy yields flat-core solutions in the case corresponding to the borderline $p$-elastica for $p\leq2$.
For describing this family, we introduce the following notation on a simple concatenation of curves:
For $\gamma_j:[a_j,b_j]\to\R^2$ with $L_j:=b_j-a_j\geq0$, we define $\gamma_1\oplus\gamma_2:[0,L_1+L_2]\to\R^2$ by
\begin{align*}
  (\gamma_1\oplus\gamma_2)(s) :=
  \begin{cases}
    \gamma_1(s+a_1), \quad & s \in[0, L_1], \\
    \gamma_2(s+a_2-L_1) +\gamma_1(b_1)-\gamma_2(a_2),  & s \in[L_1,L_1+L_2],
  \end{cases}
\end{align*}
 and inductively define $\gamma_1\oplus\dots\oplus\gamma_N := (\gamma_1\oplus\dots\oplus\gamma_{N-1})\oplus\gamma_N$.
 We also write
\begin{align*}
     \bigoplus_{j=1}^N\gamma_j:=\gamma_1\oplus\dots\oplus\gamma_N.
\end{align*}

\begin{theorem}[Explicit formulae for planar $p$-elasticae: $p>2$]\label{thm:Formulae_2}
  Let $p\in(2,\infty)$ and $\gamma$ be a $p$-elastica in $\R^2$.
  Then up to similarity and reparametrization, the curve $\gamma$ is represented by $\gamma(s)=\gamma_*(s+s_0)$ with some $s_0\in\R$, where either $\gamma_*:\R\to\R^2$ is one of the four arclength parametrizations in Cases I, II, IV, and V of Theorem \ref{thm:Formulae_1}, or $\gamma_*=\gamma_f:[0,L]\to\R^2$ is the following arclength parametrization:
  \begin{itemize}
    \item \textup{(Case III' --- Flat-core $p$-elastica)}
    For some integer $N\geq1$, signs $\sigma_1,\dots,\sigma_N\in\{+,-\}$, and nonnegative numbers $L_1,\dots,L_N\geq0$,
    \begin{equation}\label{eq:EP3-2}
      \gamma_f = \bigoplus_{j=1}^N (\gamma_\ell^{L_j}\oplus \gamma_b^{\sigma_j}),
    \end{equation}
    where $\gamma_b^\pm:[-\K_p(1),\K_p(1)]\to\R^2$ and $\gamma_\ell^{L_j}:[0,L_j]\to\R^2$ are defined by
    \begin{align}\label{eq:pm-border}
      \gamma_b^{\pm}(s) =
      \begin{pmatrix}
      2 \tanh_p{s} - s  \\
      \mp \frac{p}{p-1}(\sech_p{s})^{p-1}
      \end{pmatrix},
      \quad
      \gamma_\ell^{L_j}(s) =
      \begin{pmatrix}
      -s  \\
      0
      \end{pmatrix}.
    \end{align}
    %(Thus the length of $\gamma_f$ is $L=2NK_p(1)+\sum_{k=1}^N L_k$.)
    The curves $\gamma_b^{\pm}(s)$ have $\theta_b^\pm(s)=\pm2\am_{1,p}(s,1)=\pm2\am_{2,p}(s,1)$ and $k_b^\pm(s) = \pm2\sech_p{s}$ for $s\in [-\K_p(1),\K_p(1)]$.
  \end{itemize}
\end{theorem}

%%%%%%%%%%%%%%%%%%%%%%%%%%%%%% Maple figure 
\begin{figure}[htbp]
  \begin{center}
      \includegraphics[scale=0.4]{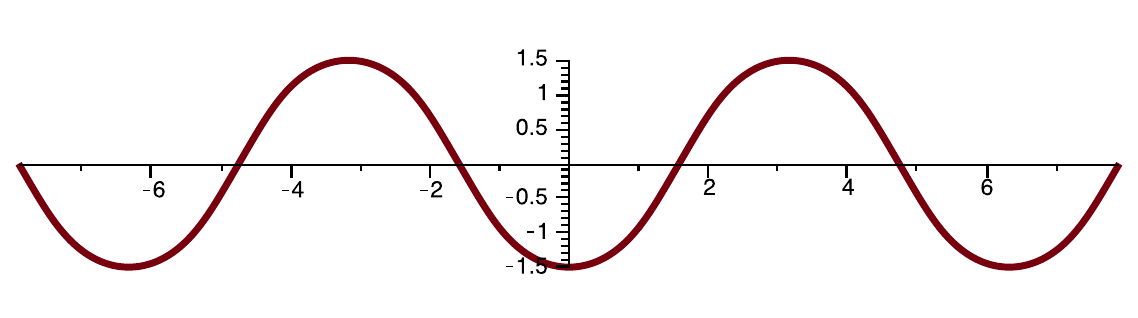}\\
      \includegraphics[scale=0.4]{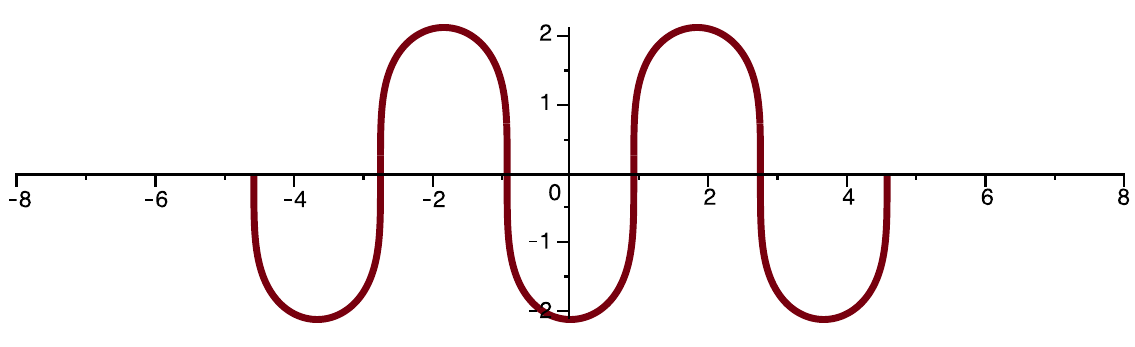}\\
      \includegraphics[scale=0.4]{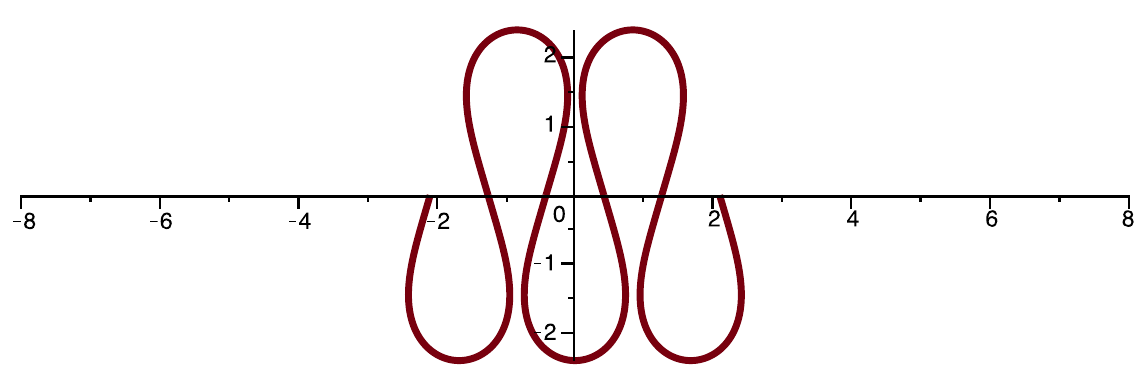}\\
      \includegraphics[scale=0.4]{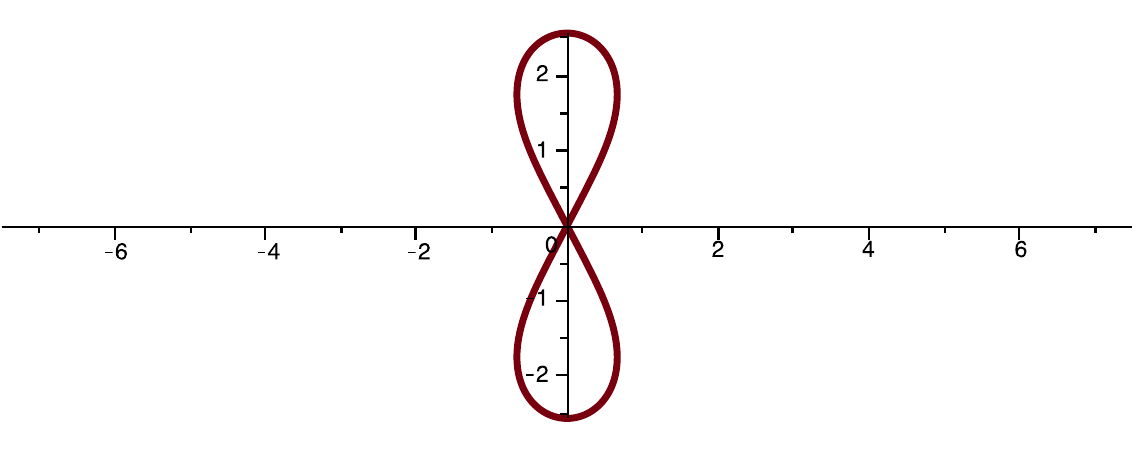}\\
      \includegraphics[scale=0.4]{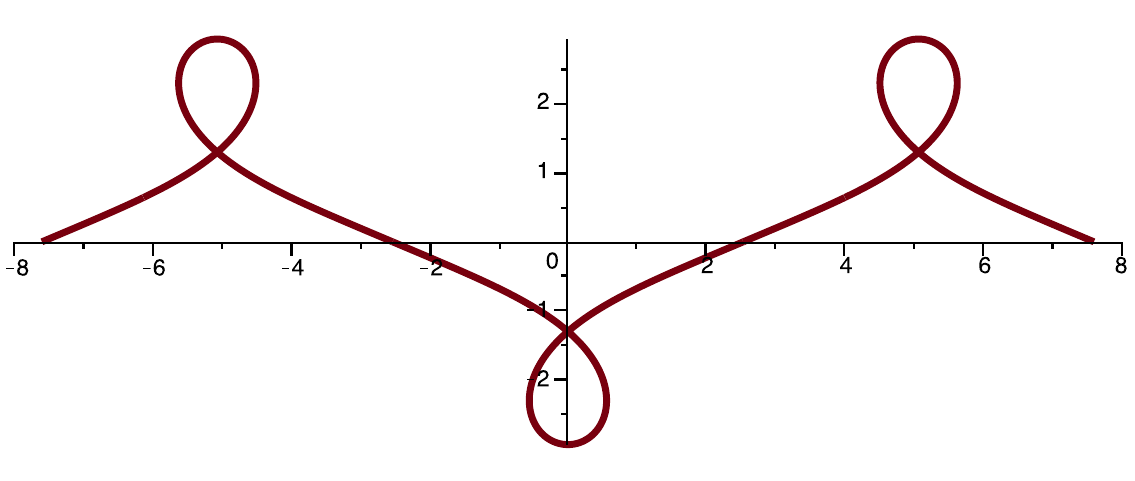}
   \caption{Wavelike $\frac{3}{2}$-elasticae. The modulus $q\in(0,1)$ increases from top to bottom; $q=0.5$, $q=0.7071$ ($\approx\frac{1}{\sqrt{2}}$), $q=0.8$, $q=0.8552$ ($\approx q^*(\frac{3}{2})$), and $q=0.98$, cf.\ Remark \ref{rem:freee} for $q=\frac{1}{\sqrt{2}}$ and Definition \ref{def:N-fold} for $q^*(p)$.}
   \label{fig:classifywave}
  \end{center}
\end{figure}
 
\begin{figure}[htbp]
 \begin{center}
     \includegraphics[scale=0.4]{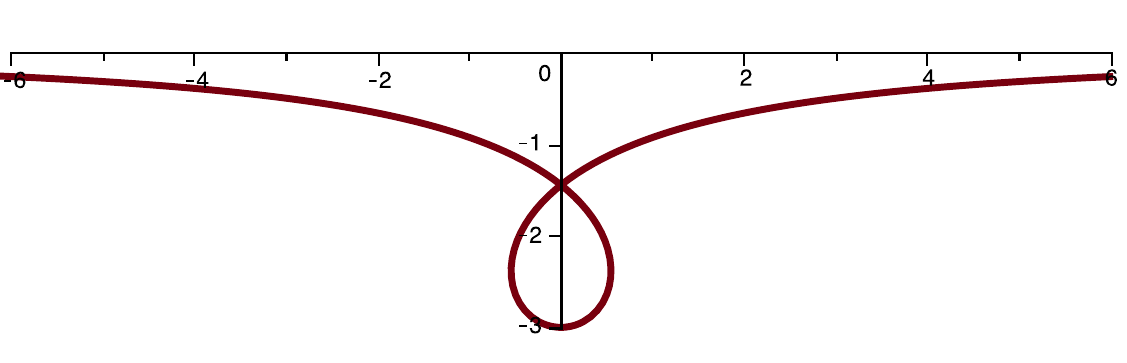}
     \caption{Borderline $\frac{3}{2}$-elastica.}
     \label{fig:classifyborder}
 \end{center}
\end{figure}

\begin{figure}[htbp]
 \begin{center}
      \includegraphics[scale=0.4]{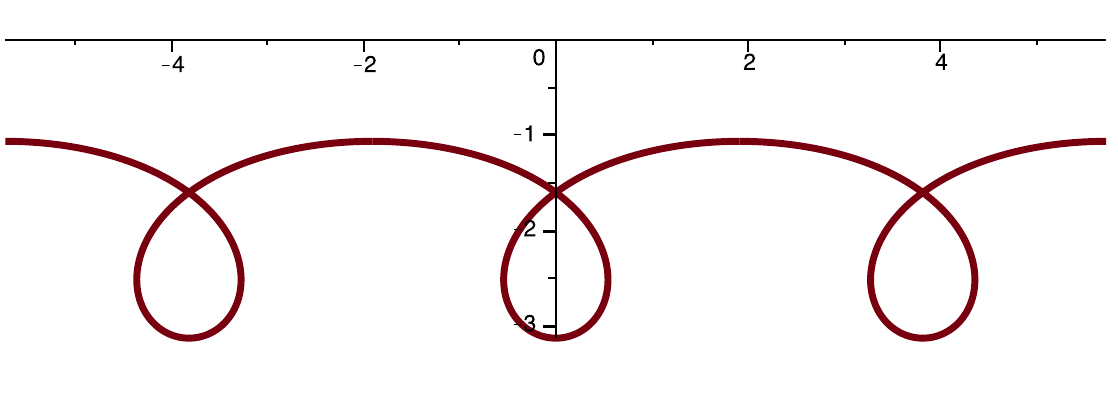}\\
      \includegraphics[scale=0.4]{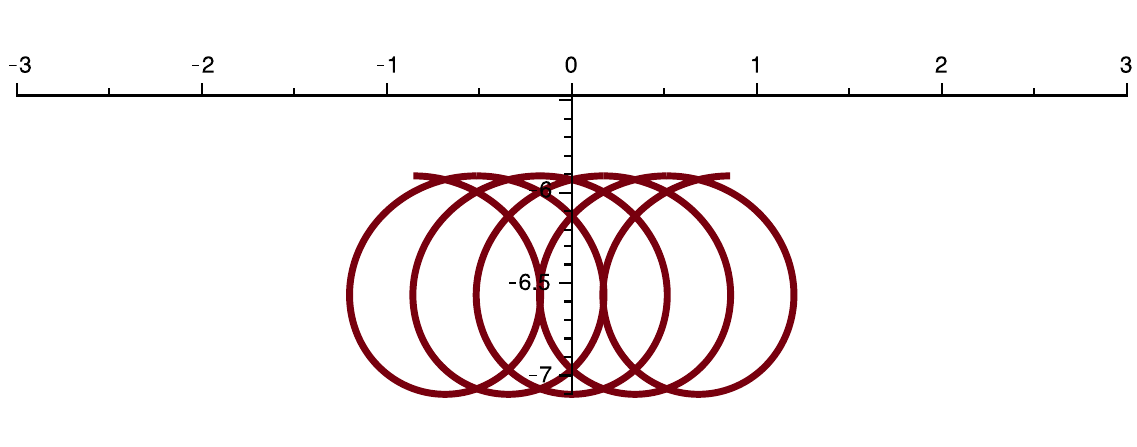}
  \caption{Orbitlike $\frac{3}{2}$-elasticae with modulus 
  $q=0.98$ (top) and $0.65$ (bottom).}
  \label{fig:classifyorbitlike}
 \end{center}
\end{figure}

\begin{figure}[htbp]
  \begin{center}
      \includegraphics[scale=0.25]{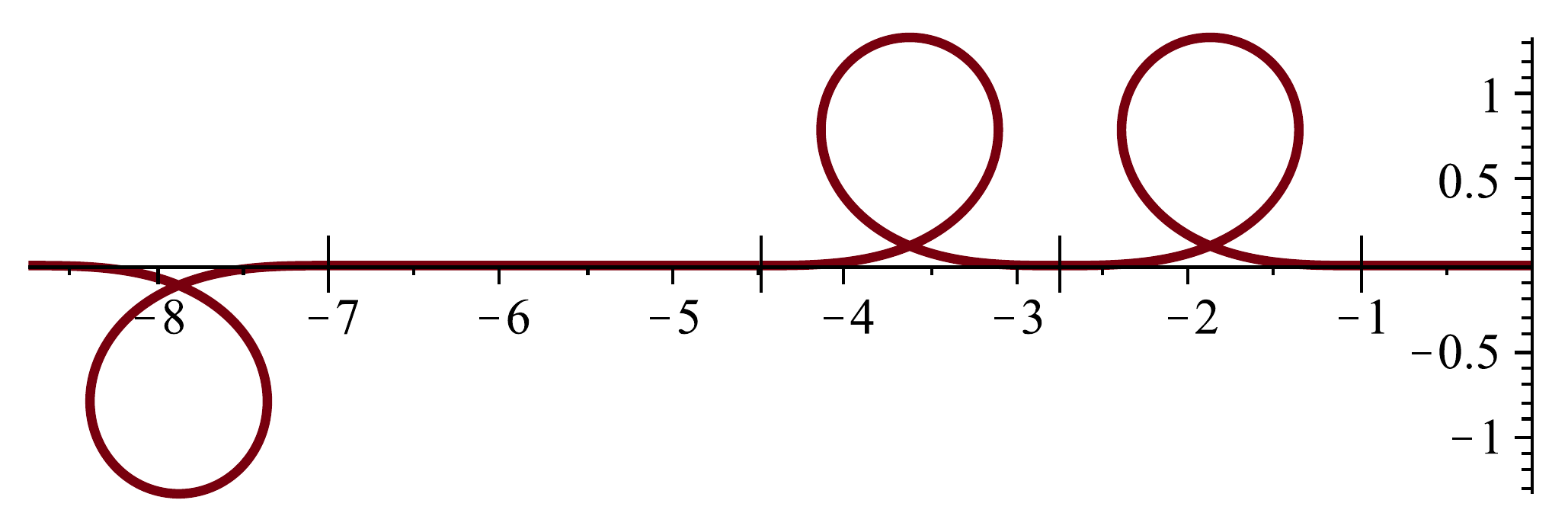}
    \caption{An example of a flat-core $4$-elastica.}
    \label{fig:classifyflatcore}
  \end{center}
\end{figure}
%%%%%%%%%%%%%%%%%%%%%%%%%%%%%% Maple figure 

\begin{remark}
  The profile curves $\gamma_\ell,\gamma_w,\gamma_b,\gamma_o,\gamma_c$ are defined on the whole $\R$, while the exceptional case $\gamma_f$ with $p>2$ has finite length $L=2N\K_p(1)+\sum_{j=1}^N L_j$ depending not only on $p$ but also on the choice of $N$ and $L_j$'s above.
  This is just due to our convention; one may suitably extend $\gamma_f$ to $\R$ if needed.
\end{remark}

\begin{remark}
  All the parametrizations of $\gamma_\ell,\gamma_w,\gamma_b,\gamma_o,\gamma_c$ and also $\gamma_b^\pm$ are chosen so that at $s=0$ the position lies in the $y$-axis and the tangential direction is rightward.
  Due to this convention, the curve $\gamma_f$ chases its profile ``from right to left''; for example, in Figure \ref{fig:classifyflatcore} our parametrization starts from the right-hand endpoint (the origin).
\end{remark}

\begin{remark}[Dependence on $\lambda$ and free $p$-elasticae]\label{rem:freee}
  We also have some relations between the above classifications and the multiplier $\lambda\in\R$, although this fact is omitted in the above statements just for simplicity.
  The trivial linear case always occurs regardless of the value of $\lambda$, while the other nontrivial cases do depend on $\lambda$.
  A particular restriction occurs for the \emph{free $p$-elasticae} corresponding to the case $\lambda=0$ (see \cite{LS_JDG, Lin93, Miura21} for $p=2$).
  In fact, any free $p$-elastica belongs to the wavelike case with modulus $q=\frac{1}{\sqrt{2}}$ (or the trivial linear case), and this nontrivial free $p$-elastica can be represented by (a part of) the graph of an antiperiodic function, with vertical slopes at the zeroes, cf.\ Figure \ref{fig:classifywave} (second).
  Indeed, in this case the angle $\theta_w(s)=2\arcsin(\frac{1}{\sqrt{2}}\sn_p(s,\frac{1}{\sqrt{2}}))$ oscillates between $\frac{\pi}{2}$ and $-\frac{\pi}{2}$.
  Such a geometric uniqueness is known to be useful in the study of obstacle problems \cite{Miura21, DMOY}.
  In addition, the case $\lambda<0$ corresponds to wavelike elasticae with $q<\frac{1}{\sqrt{2}}$, where the profile curve $\gamma_w$ is entirely the graph of a smooth function.
  The case $\lambda>0$ consists of all the other cases.
  For more details, see Theorem \ref{thm:k-classify}, which gives a fully detailed classification involving $\lambda$, thus extending \cite[Proposition 3.3]{Lin96} from $p=2$ to $p\in(1,\infty)$.
\end{remark}

Now we turn to the regularity of planar $p$-elasticae.
We begin with a soft but general regularity statement, which ensures that any $p$-elastica is of class $W^{3,1}$: %$C^2\hookrightarrow W^{3,1}$:

\begin{theorem}[General regularity]\label{thm:Soft_regularity}
  Let $p\in(1,\infty)$, $L>0$, and $\gamma:[0,L]\to\R^2$ be an arclength parametrized $p$-elastica.
  %Then $\gamma$ has continuous signed curvature $k\in C([0,L])$.
  Then $\gamma \in W^{3,1}(0,L;\R^2)$.
  In particular, $\gamma \in C^{2}([0,L];\R^2)$.
  In addition, $\gamma$ is analytic on $[0,L]\setminus \partial Z$, where $Z:=\Set{ s\in [0,L] | k(s)=0 }$
  and $\partial Z$ denotes the relative topological boundary in $[0,L]$.
\end{theorem}

In fact, our explicit formulae imply that, except in the flat-core case, the set $Z$ is always discrete and hence $Z=\partial Z$.

Hence for investigating the optimal (loss of) regularity it is sufficient to look at $p$-elasticae whose curvatures have nontrivial zero sets, i.e., $\partial Z\neq\emptyset$.
In the case of $p\leq2$ the loss of regularity occurs only in the wavelike case:

\begin{theorem}[Optimal regularity: $p\leq2$]\label{thm:Optimal_regularity_1}
  Let $p\in(1,2]$, $L>0$, and $\gamma:[0,L]\to\R^2$ be an arclength parametrized $p$-elastica with $Z=\partial Z\neq\emptyset$.
  If $\gamma$ is either linear, borderline, orbitlike, or circular, then
  $\gamma$ is analytic on $[0,L]$.
  In addition, if $\gamma$ is wavelike, then the following statements hold:
  Let $m_p:=\lceil \tfrac{1}{p-1} \rceil\geq1$ and $r_p:= (m_p-\tfrac{1}{p-1} )^{-1}>1$, where $\lceil \cdot \rceil$ stands for the ceiling function.
  \begin{itemize}
    \item[(i)] If $\frac{1}{p-1}$ is an odd integer, then $\gamma$ is analytic on $[0,L]$.
    \item[(ii)] If $\frac{1}{p-1}$ is an even integer, then
    $\gamma\in W^{m_p+2,\infty}(0,L;\R^2).$
    If in addition $Z\cap (0,L)\neq\emptyset$, then
    $\gamma\not\in W^{m_p+3,1}(0,L;\R^2).$
    \item[(iii)] If $\frac{1}{p-1}$ is not an integer, then $\gamma\in W^{m_p+2,r}(0,L;\R^2)$ for any $r\in[1,r_p)$, and $\gamma\not\in W^{m_p+2,r_p}(0,L;\R^2)$.
  \end{itemize}
\end{theorem}

For $p>2$, the flat-core $p$-elasticae may also lose their regularity, but in general they are more regular than wavelike ones:

\begin{theorem}[Optimal regularity: $p>2$]\label{thm:Optimal_regularity_2}
  Let $p\in(2,\infty)$, $L>0$, and $\gamma:[0,L]\to\R^2$ be an arclength parametrized $p$-elastica with $\partial Z\neq\emptyset$.
  If $\gamma$ is either linear, orbitlike, or circular, then $\gamma$ is analytic on $[0,L]$.
  If $\gamma$ is wavelike, then $\gamma\in W^{3,r}(0,L;\R^2)$ for any $r\in[1,\frac{p-1}{p-2})$, and $\gamma\not\in W^{3,\frac{p-1}{p-2}}(0,L;\R^2)$.
  In addition, if $\gamma$ is flat-core, then the following statements hold:
  Let $M_p:=\lceil \tfrac{2}{p-2} \rceil\geq1$ and $R_p:= (M_p- \tfrac{2}{p-2} )^{-1}>1$.
  \begin{itemize}
    \item[(i)] If $\frac{2}{p-2}$ is not an integer, then $\gamma\in W^{M_p+2,r}(0,L;\R^2)$ for any $r\in[1,R_p)$, and $\gamma\not\in W^{M_p+2,R_p}(0,L;\R^2).$
    \item[(ii)] If $\frac{2}{p-2}$ is an integer, then
    $\gamma\in W^{M_p+2,\infty}(0,L;\R^2).$
    If in addition $Z$ is not discrete (or equivalently contains an interval), then
    $\gamma\not\in W^{M_p+3,1}(0,L;\R^2).$
  \end{itemize}
\end{theorem}

\begin{remark}
  The above regularity of wavelike $p$-elasticae with $p>2$ may be regarded as a natural extension of the previous case $p\leq2$, since if $p>2$, then we always have $m_p=\lceil \frac{1}{p-1} \rceil=1$ and $r_p=(m_p- \frac{1}{p-1})^{-1}=\frac{p-1}{p-2}$.
  Also the $C^\infty$ regularity of borderline $p$-elasticae with $p\leq2$ (which are in fact analytic) may be related with the flat-core case $p>2$ in view of $\lim_{p\to 2+}M_p=\infty$.
\end{remark}

\begin{remark}
  The optimal regularity described above can roughly be summarized as follows:
  In the wavelike case the curvature $k$ has the same regularity as $x\mapsto \textrm{sign}(x)|x|^{\frac{1}{p-1}}$ around $x=0$, while in the flat-core case it behaves like $x\mapsto (x^+)^{\frac{2}{p-2}}$.
\end{remark}

In particular, the above regularity results directly determine the exponents $p\in(1,\infty)$ such that every $p$-elastica is analytic, since otherwise there is a $p$-elastica not of class $C^{2+\lceil \frac{1}{p-1}\rceil}$.

\begin{corollary}\label{cor:regularity-special-case}
    Let $p\in(1,\infty)$.
    Then every $p$-elastica is analytic if and only if $p=\frac{2m}{2m-1}$ for some integer $m\geq1$, that is, $p\in\{\frac{2}{1},\frac{4}{3},\frac{6}{5},\frac{8}{7},\dots\}$.
\end{corollary}

Now we discuss several outcomes of our classification and regularity results.

Thanks to the classification in Theorems \ref{thm:Formulae_1} and \ref{thm:Formulae_2}, we are now able to solve boundary value problems purely in terms of $p$-elliptic integrals and functions.
%For example, in Theorem \ref{thm:classify-closed}, we classify all the possible closed planar $p$-elasticae; we prove that they are a circle, a figure-eight $p$-elastica, and their multiple coverings, extending the classical case $p=2$ (see e.g.\ \cite[Theorem 0.1]{LS_85}).
As a particular example we classify all the possible closed planar $p$-elasticae into circles and figure-eight $p$-elasticae, cf.\ Definition \ref{def:N-fold} and Figure \ref{fig:p-figure8}:

%%%%%%%%%%%%%%%%%%%%%%%%%%%%%% Maple figure ver.1 
\begin{center}
    \begin{figure}[htbp]
  %\begin{minipage}{0.45\textwidth}
      \includegraphics[scale=0.2]{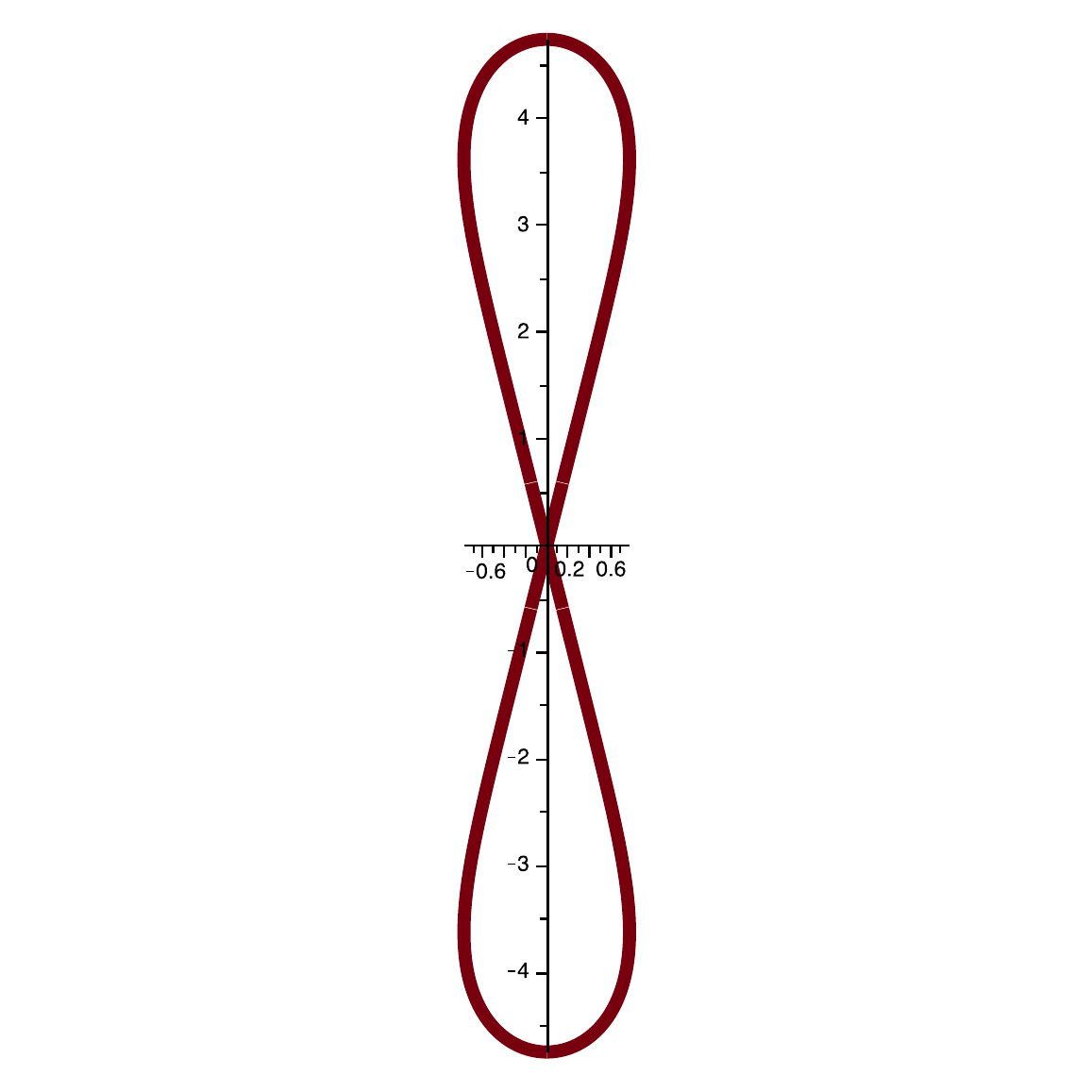}
  %\end{minipage}
  %\begin{minipage}{0.45\textwidth}
      \includegraphics[scale=0.2]{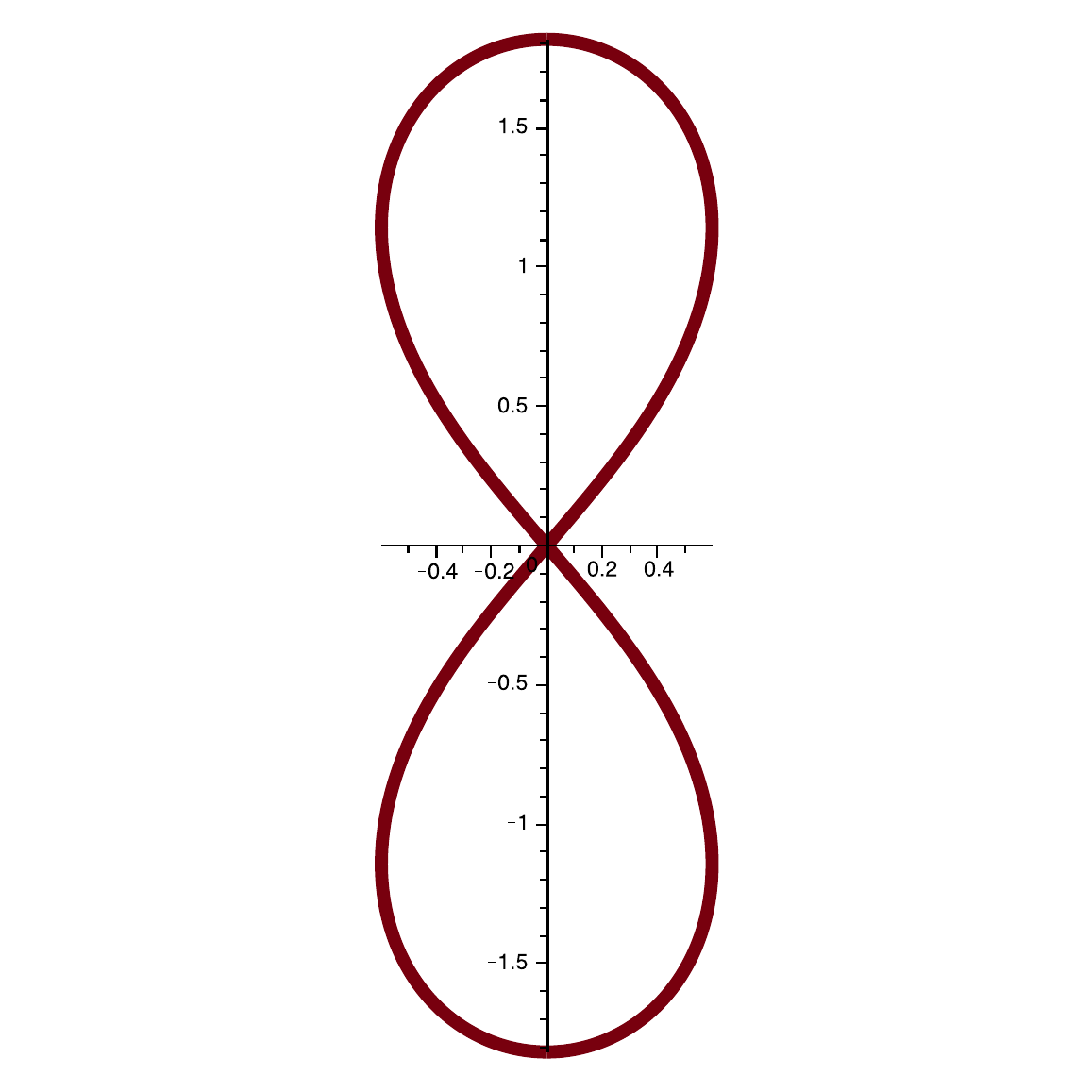}
  %\end{minipage}
    %\begin{minipage}{0.45\textwidth}
      \includegraphics[scale=0.2]{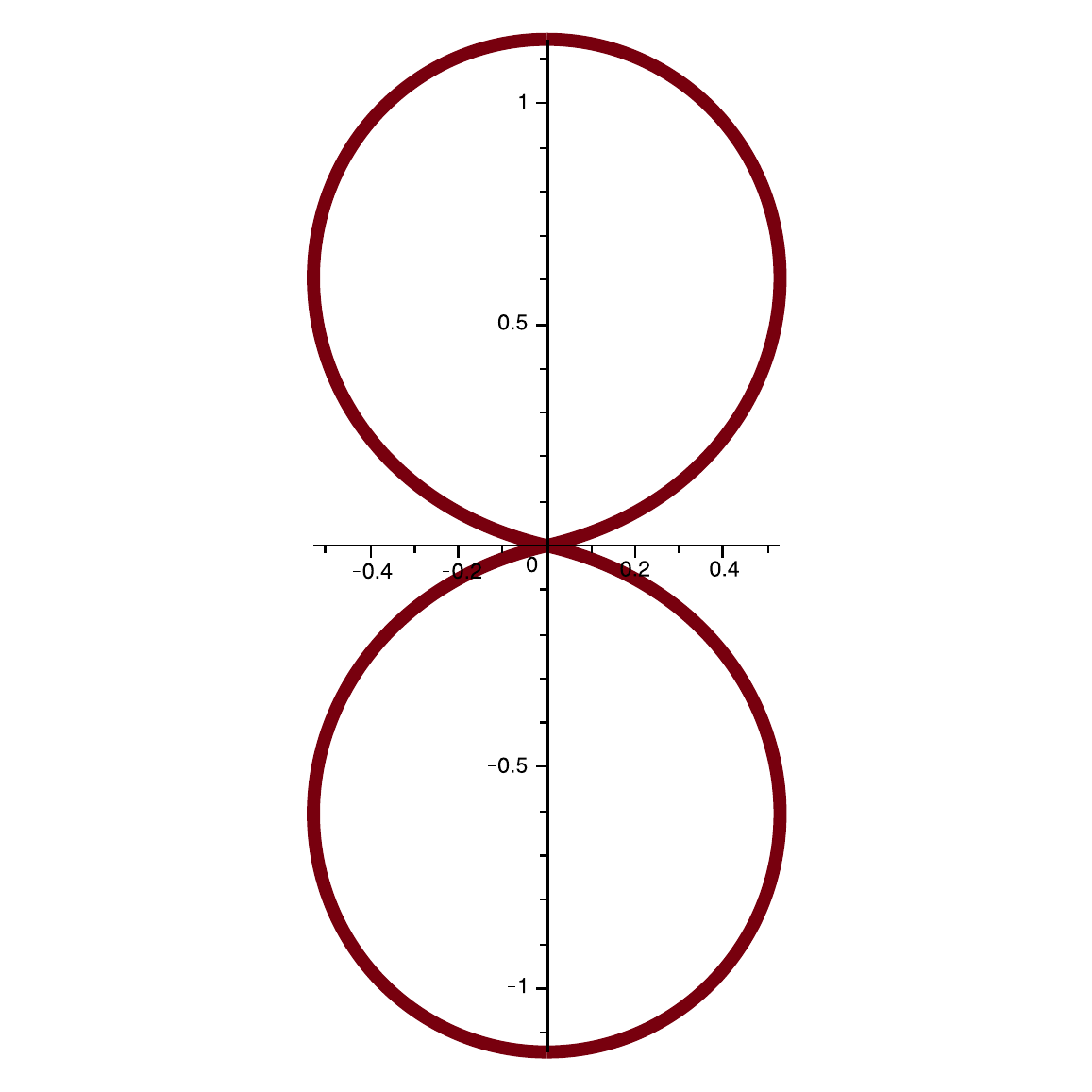}
  %\end{minipage}
  \caption{Figure-eight $p$-elasticae with $p=\frac{6}{5}$, $p=2$, $p=10$ (from left to right).}
  \label{fig:p-figure8}
  \end{figure}
\end{center}
%%%%%%%%%%%%%%%%%%%%%%%%%%%%%% Maple figure ver.1 

\begin{theorem}[Closed planar $p$-elastica: Theorem \ref{thm:classify-closed}]
Let $p\in(1,\infty)$ and $\gamma$ be a closed planar $p$-elastica.
Then $\gamma$ is either a circle or a figure-eight $p$-elastica, possibly multiply covered.
\end{theorem}

This result directly extends the classical case $p=2$ (see e.g.\ \cite[Theorem 0.1]{LS_85}).
In particular all flat-core solutions are ruled out.
The figure-eight $p$-elastica tends to a double-line as $p\to1$ and to a double-circle as $p\to\infty$ (see also \cite[Theorem 1.5]{MYarXiv2209}, and \cite{Mos22} for $p=\infty$.)
In our subsequent paper \cite{MYarXiv2209} our formulae are also applied to classify all $p$-elasticae under the so-called pinned boundary condition.
% In addition, Figure~\ref{fig:p-figure8} implies the geometric property that a figure-eight $p$-elastica moves to a double-line as $p$ is close to $1$ and it tends to be two circles as $p$ diverges, and this will be also discussed more precisely in  \cite{MYarXiv2209}.
% For the case $p=\infty$, it is in \cite{Mos22} shown that two adjacent circles appear as the so-called $\infty$-elastica.

The regularity results are also useful in various contexts.
In our third paper \cite{MYarXiv2301} the general $C^2$-regularity plays a crucial role to study stability of $p$-elasticae.
The optimal regularity results would be informative for studying gradient flows, of second order \cite{NP20, OPW20} and of fourth order \cite{BHV, BVH, OW23} (called $p$-elastic flow).
In the classical case $p=2$ many results about smooth convergence to elasticae as time goes to infinity are already known, in particular for elastic flows of fourth order; see the survey \cite{MPP21} and references therein (and also \cite{Wen, OW23} for second-order flows).
In contrast, for $p\neq2$, long-time behavior is less understood, and the only known results seem to be recent sub-convergence results \cite{NP20,BHV}; the former \cite{NP20} is about (at least) $W^{2,p}$-weak convergence of a second-order gradient flow for $\Theta$-networks; the latter \cite{BHV} is about $W^{2,p}$-strong convergence of a fourth-order $p$-elastic flow for closed curves with $p\geq 2$.
Our results already ensure that the topology of general convergence cannot be better than the ``$W^{m_p,r_p}$-type regularity'' in the wavelike case of Theorems \ref{thm:Optimal_regularity_1} and \ref{thm:Optimal_regularity_2}.
We conjecture that this $W^{m_p,r_p}$-type regularity is indeed the optimal convergence topology for $p$-elastic flows.

Finally we discuss the structure of this paper while mentioning the main stream and key ideas of our proof.

In Section \ref{chap:ELeq} we transform the condition for a curve to be a critical point into a weak form of the Euler--Lagrange equation in terms of curvature, while establishing the $L^\infty$-regularity of the curvature so that the weak form is well defined.
In Section \ref{sect:p-Jacobi} we introduce $p$-elliptic integrals and functions, and then verify that they indeed solve the above Euler--Lagrange equation for curvature (up to the choice of suitable coefficients).
In \cite{nabe14} Watanabe used two types of $p$-elliptic integrals $\E_{1,p}$ and $\E_{2,p}$.
Based on this idea, we introduce two types of amplitude functions $\amcn$ and $\amdn$, and independently use them to introduce the new $p$-elliptic functions $\cn_p$ and $\dn_p$.
In Section \ref{sect:mainthm} we give a full classification of all solutions in terms of an initial value problem for curvature (Theorem \ref{thm:k-classify}) by ensuring that no other solutions exist than the solutions obtained above by $p$-elliptic functions.
Our proof here is inspired by %Shioji--Watanabe's transformation \cite{SW20}.
the transformation used by Shioji--Watanabe \cite{SW20}.
More precisely, if we let $w:=|k|^{p-2}k$, then equation \eqref{Teq:sw20-1.2} is (again formally) transformed into the semilinear equation
\[
p w_{ss} + (p-1)|w|^{\frac{2}{p-1}}w -\lambda |w|^{\frac{2-p}{p-1}}w = 0,
\]
and hence we may expect more regularity for $w$.
In fact, we prove that although $k$ may not be of class $C^2$, the following property holds true in general (Lemma \ref{lem:0812}):
\[
\text{If $k$ is the curvature of a $p$-elastica, then $w:=|k|^{p-2}k \in C^2$.}
\]
The $C^2$-regularity of $w$ was merely assumed in \cite{SW20}, but in the present paper it is proven as a new regularity result.
This regularity property allows us to reduce the problem into classification of classical solutions to the above semilinear equation, which will be accomplished by careful but rather explicit computations thanks to our formulae.
Finally, Section \ref{sect:closed} is devoted to the classification of closed planar $p$-elasticae.
The arguments heavily rely on our explicit formulae again.

\subsection*{Acknowledgments}
The first author is supported by JSPS KAKENHI Grant Numbers 18H03670, 20K14341, and 21H00990, and by Grant for Basic Science Research Projects from The Sumitomo Foundation.
The second author is supported by JSPS KAKENHI Grant Numbers 19J20749 and 22K20339.

%%%%%%%%%%%%%%%%%%%%%%%%%%%%%%%%%%%%%%
%%%%%%%%%%%%%%%%%%%%%%%%%%%%%%%%%%%%%%
%%%%%%%%%%%%%%%%%%%%%%%%%%%%%%%%%%%%%%
\section{The Euler--Lagrange equations}\label{chap:ELeq}

In this section we first derive a weak form of the Euler--Lagrange equation for $p$-elasticae in terms of curves in the natural energy space.
Then we obtain an equivalent weak form of the equation in terms of the signed curvature by slightly improving the regularity, thus reducing the above fourth-order ODE system to a single second-order ODE.

To begin with, we prepare some notations. 
For $p\in (1, \infty)$ we define the set of $W^{2,p}$-immersed curves as follows: 
\[
W^{2, p}_{\rm imm}(0, 1; \R^2):= \set{ \gamma \in W^{2, p}(0,1; \R^2) | \gamma'(t) \neq 0  \text{ for all } t \in[0,1] }.
\]
For $\gamma \in W^{2, p}_{\rm imm}(0, 1; \R^2)$ we often let $\tilde{\gamma}$ denote the arclength parametrization of $\gamma$.
Hereafter, we use both the original parameter $t \in[0,1]$ and the arclength parameter $s\in[0, \mathcal{L}[\gamma]]$.
In this paper we are concerned with curves under the fixed length condition, and throughout this section $L>0$ denotes the length of curves.

Now we discuss the Euler--Lagrange equation.
By the known computation of the first variation of $\mathcal{B}_p$, cf.\ Appendix~\ref{sect:First-variation-Bp}, %Theorem~\ref{thm:arclengthLM} in the appendix, 
we deduce that $\gamma \in W^{2, p}_{\rm imm}(0, 1; \R^2)$ is a $p$-elastica with $\lambda \in \R$ in the sense of Definition \ref{def:p-elastica} if and only if the arclength parametrization $\tilde{\gamma}\in W^{2, p}(0, L; \R^2)$ of $\vc$ satisfies 
\begin{align}\label{eq:0525-1}
\int_0^L\Big( (1-2p) |\tilde{\vc}''|^p (\tilde{\vc}', \eta') +p|\tilde{\vc}''|^{p-2}(\tilde{\vc}'', \eta'') +\lm (\tilde{\vc}', \eta')\Big) ds = 0
\end{align}
for all $\eta \in C^{\infty}_{\rm c}(0, L;\R^2)$.
Here $(\cdot,\cdot)$ denotes the Euclidean inner product.

%%%%%%%%%%%%%%%%%%%%%%%%%%%%%%%%%%%%%%

Next, we translate the above Euler--Lagrange equation in terms of the signed curvature $k$:

%%%%%%%%%%%%%%%%%%%%%%%%%%%%%%%%%%%%%%
\begin{proposition} \label{prop:p-ELk} 
If an arclength parametrized curve $\tilde{\vc}\in W^{2,p}(0,L; \R^2)$ satisfies \eqref{eq:0525-1} for some $\lm \in \R$, then the signed curvature $k$ of $\tilde{\vc}$ belongs to $L^{\infty}(0,L)$.
Moreover, the function $k$ satisfies 
\begin{align}\label{eq:EL}\tag{EL}
\int_0^L \Big( p |k|^{p-2}k \vp'' +(p-1) |k|^pk\vp -\lm k \vp \Big) ds = 0
% \quad \text{for all} \quad \vp \in C^{\infty}_{\rm c}(0,L).
\end{align}
for all $\vp \in C^{\infty}_{\rm c}(0,L)$.
% Conversely, when the curvature of $\vc$ satisfies \eqref{eq:EL} for some $\lm \in \R$, then $\vc$ satisfies \eqref{eq:0525-1} with its constant $\lm \in \R$. 
\end{proposition}
%%%%%%%%%%%%%%%%%%%%%%%%%%%%%%%%%%%%%%
\begin{proof}
We first prove that $k \in L^{\infty}(0,L).$
By density we may assume that $\tilde{\vc}\in W^{2,p}(0,L; \R^2)$ satisfies \eqref{eq:0525-1} for any $\eta \in W^{2,p}_0(0,L; \R^2)$.
Here we adopt ideas from %\cite[Proposition 3.2]{DD18} and 
\cite[Theorem 3.9]{DDG}.
Fix any $\phi \in C^{\infty}_{\rm c}(0,L)$, and set
\[
\xi(s):= \int_0^s \!\! \int_0^t \phi(r)drdt + \alpha s^2 + \beta s^3, \quad s\in[0,L], \]
with
\[ \alpha:= \frac{1}{L} \int_0^L \phi(r)\,dr - \frac{3}{L^2} \int_0^L\! \int_0^r \phi(t)\,dtdr, \quad 
\beta:= -\frac{\alpha}{L}-\frac{1}{L^3}\int_0^L\! \int_0^r \phi(t)\,dtdr.
\]
Then $\xi \in W^{2,p}_0(0,L)$ and $|\va|, |\vb|, \| \xi \|_{C^1} \leq C \|\phi\|_{L^1}$ 
with a constant $C$ depending only on $L$.
Inserting $\eta(s)=(\xi(s),0)$ into \eqref{eq:0525-1}, we see that for $\tilde{\vc}(s)=(\tilde{\vc}_1(s), \tilde{\vc}_2(s))$
\begin{align*}
\Big| \int_0^L p |\tilde{\vc}''|^{p-2} \tilde{\vc}_1'' \,\phi \,ds  \Big| \leq C' \| \phi\|_{L^1(0,L)},
\end{align*}
with a constant $C'$ depending only on $L$ and the $W^{2,p}$-norm of $\tilde{\gamma}$.
By the arbitrariness of $\phi$ we get $ |\tilde{\vc}''|^{p-2} \tilde{\vc}_1'' \in L^{\infty}(0,L)$.
By the same argument we also obtain $ |\tilde{\vc}''|^{p-2} \tilde{\vc}_2'' \in L^{\infty}(0,L)$ and hence $|k|=|\tilde{\gamma}''|\in L^\infty(0,L)$, so that $k\in L^\infty(0,L)$.

We now prove that $k$ satisfies \eqref{eq:EL}.
% Denote the unit normal vector of $\tilde{\vc}(s)$ by $\mathbf{n}(s)$. 
Fix any $\varphi \in C^{\infty}_{\rm c}(0,L)$.
Let $\mathbf{t}$ and $\mathbf{n}$ be the unit tangent vector and the unit normal vector of $\tilde{\gamma}$, respectively, defined by $\mathbf{t}(s):=\tilde{\gamma}'(s)$ and $\mathbf{n}(s):=Q_{\pi/2}\tilde{\vc}'(s)$, where $Q_{\theta}$ stands for the counterclockwise rotation matrix through angle $\theta\in \R$.
Recall that the Frenet--Serret formula yields $\mathbf{t}'(s)=k(s)\mathbf{n}(s)$ and $\mathbf{n}'(s)=-k(s)\mathbf{t}(s)$, and in particular
\[\mathbf{n}(s)= \mathbf{n}(0)-\int_0^s k(\sigma) \mathbf{t}(\sigma)\,d\sigma, \quad s\in[0,L].\]
Take any sequence $\{ k_j \}_{j\in \N} \subset C^1(0,L)$ such that
$k_j\to k$ in $L^p(0,L)$ as $j\to \infty$.
For each $j\in \N$, set
\begin{align*}
\eta_j (s) &:= \varphi(s)\mathbf{n}_j (s), \quad \text{where} \quad
\mathbf{n}_j (s) := \mathbf{n}(0)-\int_0^s k_j(\sigma) \mathbf{t}(\sigma)\,d\sigma.
\end{align*} 
Since $k_j\to k $ in $L^p(0,L)$, it follows that $\mathbf{n}_j \to \mathbf{n}$ in $L^{\infty}(0,L)$.
We also notice that $\eta_j(0)=\eta_j(L)=(0,0)$ and, using again the Frenet--Serret formula,
\begin{align*}
\eta_j'(s)&=\vp'(s)\mathbf{n}_j(s)-\vp(s)k_j(s)\mathbf{t}(s), \\
\eta_j''(s)&=\vp''(s)\mathbf{n}_j(s)-2\vp'(s)k_j(s)\mathbf{t}(s) -\vp(s)k_j'(s)\mathbf{t}(s) - \vp(s)k_j(s)k(s)\mathbf{n}(s).
\end{align*}
Therefore, $\eta_j\in W^{2,p}_0(0,L)$ for each $j\in\N$.
%Note that
% \begin{align*}
%  (1-2p) |\tilde{\vc}''|^p (\tilde{\vc}', \eta_j') &=(1-2p)|k|^p \Big(  (\tilde{\gamma}',\varphi' \mathbf{n}_j) - k_j\vp \Big), \\ 
% p|\tilde{\vc}''|^{p-2}(\tilde{\vc}'', \eta_j'') &= p |k|^{p-2}\Big( (\tilde{\gamma}'', \varphi''\mathbf{n}_j)- |k|^{2} k_j\varphi \Big),
% \end{align*}
%
Substituting $\eta=\eta_j$ into \eqref{eq:0525-1}, 
we deduce from $|\tilde{\vc}'|\equiv|\mathbf{n}|\equiv1$ and  $\tilde{\gamma}''=k\mathbf{n}$ that
\begin{align*}
\int_0^L \bigg( (1-2p)|k|^p \Big( \vp'(\tilde{\gamma}', \mathbf{n}_j) - k_j\vp \Big) 
+p|k|^{p-2} \Big( \vp''(\tilde{\gamma}'', \mathbf{n}_j) - |k|^{2} k_j\varphi \Big)& \\
+\lambda \Big( \vp'(\tilde{\gamma}', \mathbf{n}_j)-k_j\varphi\Big)
\bigg) ds &= 0.
\end{align*}
By using $k_j\to k$ in $L^p(0,L)$, $\mathbf{n}_j\to \mathbf{n}$ in $L^{\infty}(0,L)$, and $k \in L^{\infty}(0,L)$, we can obtain \eqref{eq:EL} as the limit of the above equality.
\end{proof}

%%%%%%%%%%%%%%%%%%%%%%%%%%%%%%%%%%%%%%
%%%%%%%%%%%%%%%%%%%%%%%%%%%%%%%%%%%%%%
%%%%%%%%%%%%%%%%%%%%%%%%%%%%%%%%%%%%%%

\section{$p$-Elliptic functions}\label{sect:p-Jacobi}

From now on we investigate various properties of $p$-elasticae through the analysis of \eqref{eq:EL}.
In the classical case of $p=2$ it is well known that the curvature of all critical points is completely classified in terms of the Jacobian elliptic functions $\cn$ and $\dn$.
In order to extend this fact to the general power, in this section, we introduce suitable generalizations of the classical elliptic functions, which naturally appear as solutions to \eqref{eq:EL}.
One of the important points is that the elliptic cosine $\cn$ and the delta amplitude $\dn$ are extended in ``different'' ways, by introducing two kinds of generalization of the elliptic integral.

\subsection{Definitions}

We first introduce some generalized incomplete elliptic integrals:
%%%%%%%%%%%%%%%%%%%%%%%%%%%%%%%%%%%%%%
\begin{definition}[$p$-Elliptic integrals of the first kind] \label{def:K_p}
Let $p\in(1,\infty)$.
We define the incomplete $p$-elliptic integrals of the first kind $\mathrm{F}_{1,p}(x,q)$ and $\mathrm{F}_{2,p}(x,q)$ of modulus $q\in[0,1)$, where $x\in\R$, by 
\begin{align*}
&\mathrm{F}_{1,p}(x, q):=\int_0^{x} \frac{|\cos\theta|^{1-\frac{2}{p}} }{ \sqrt[]{1-q^2\sin^2\theta} }\,d\theta, \quad
 \mathrm{F}_{2,p}(x,q):=\int_0^{x} \frac{1}{ \sqrt[p]{1-q^2\sin^2\theta} }\,d\theta,
\end{align*}
and also the corresponding complete $p$-elliptic integrals $\K_{1,p}(q)$ and $\K_{2,p}(q)$ by 
\begin{align*}
    \K_{1,p}(q):=\mathrm{F}_{1,p}(\pi/2, q), \quad \K_{2,p}(q):=\mathrm{F}_{2,p}(\pi/2, q).
\end{align*}
For $q=1$, we define 
\[
\mathrm{F}_{1,p}(x,1)=\mathrm{F}_{2,p}(x,1):=\displaystyle \int_0^{x} \frac{d\theta}{|\cos \theta|^{\frac{2}{p} } }, \quad \text{where}\ 
\begin{cases}
x\in(-\frac{\pi}{2},\frac{\pi}{2}) \quad &\text{if} \ \ 1< p \leq  2,  \\
x\in\R &\text{if} \ \ p>2,
\end{cases}
\]
and
\begin{align} \notag%\label{eq:K_p}
\K_{1,p}(1)=\K_{2,p}(1)=\K_{p} (1) := 
\begin{cases}
\infty \quad &\text{if} \ \ 1< p \leq  2,  \\
\displaystyle \int_0^{\frac{\pi}{2}} \frac{d\theta}{(\cos \theta)^{\frac{2}{p} } } < \infty &\text{if} \ \ p>2.
%=\int_0^{\frac{\pi}{2}} \frac{d\theta}{(\sin \theta)^{\frac{2}{p} } }
\end{cases}
\end{align}
\end{definition}
%%%%%%%%%%%%%%%%%%%%%%%%%%%%%%%%%%%%%%

%%%%%%%%%%%%%%%%%%%%%%%%%%%%%%%%%%%%%%
\begin{definition}[$p$-Elliptic integrals of the second kind]\label{def:E_p}
Let $p\in (1,\infty)$.
We define the incomplete $p$-elliptic integrals of the second kind $\E_{1,p}(x,q)$ and $\E_{2,p}(x,q)$ 
of modulus $q\in[0,1]$, where $x\in\R$, by 
\[
\E_{1,p}(x,q):=\int_0^{x} \sqrt{1-q^2 \sin^2 \theta}\, |\cos \theta|^{1-\frac{2}{p}}\,d\theta, 
\quad
\E_{2,p}(x,q):=\int_0^{x} \sqrt[p]{1-q^2 \sin^2 \theta} \,d\theta,
\]
and also the corresponding complete $p$-elliptic integrals $\E_{1,p}(q)$ and $\E_{2,p}(q)$ by 
\begin{align*}
    \E_{1,p}(q):=\E_{1,p}(\pi/2, q), \quad \E_{2,p}(q):=\E_{2,p}(\pi/2, q).
\end{align*}
\end{definition}
%%%%%%%%%%%%%%%%%%%%%%%%%%%%%%%%%%%%%%

\begin{remark}
If $p=2$, then the functions $\mathrm{F}_{1,2}$ and $\mathrm{F}_{2,2}$ (resp.\ $\E_{1,2}$ and $\E_{2,2}$) coincide with the classical elliptic integral of the first kind $\mathrm{F}$ (resp.\ of the second kind $\E$).
In addition, our complete $p$-elliptic integrals coincide with Watanabe's generalization in \cite{nabe14}, 
which can be regarded as special cases of Takeuchi's generalization  \cite{Takeuchi_RIMS, Takeuchi16}.
In general, since $1-\frac{2}{p} >-1$, both $\Fcn(x,q)$ and $\Ecn(x,q)$ are well defined for each $x\in\R$ and $q\in[0,1)$.
Also, by definition and periodicity, for any $x\in \R$, $q\in[0,1)$, and $n \in \Z$ we have
\begin{align}\label{eq:period_Ecn}
\begin{split}
\Ecn(x+ n\pi , q) &=  \Ecn(x , q) + 2n\Ecn(q), \ \
\Fcn(x+ n\pi , q) =  \Fcn(x , q) + 2n\Kcn(q), \\
\Edn(x+ n\pi , q) &=  \Edn(x , q) + 2n\Edn(q), \ \ 
\Fdn(x+ n\pi , q) =  \Fdn(x , q) + 2n\Kdn(q).
\end{split}
\end{align}
\end{remark}

%%%%%%%%%%%%%%%%%%%%%%%%%%%%%%%%%%%%%%
%\begin{remark} 
%(1) For each $x\in\R$ and $q\in[0,1)$, from $1-\frac{2}{p} >-1$ we infer that $\Fcn(x,q)$ and $\Ecn(x,q)$ are well defined.\\
%(2) When $p=2$, we notice that $\K_{1,2}(q)=\K_{2,2}(q)=\K(q)$ and $\E_{1,2}(q)=\E_{2,2}(q)=\E(q)$, where $\K(q)$ (resp.\ $\E(q)$) is (classical) elliptic integrals of the first (resp.\ second) kind.
%\end{remark}
%%%%%%%%%%%%%%%%%%%%%%%%%%%%%%%%%%%%%%

Next we generalize the Jacobian elliptic functions:

%%%%%%%%%%%%%%%%%%%%%%%%%%%%%%%%%%%%%%
\begin{definition}[$p$-Elliptic functions] \label{def:cndn}
Let $p\in(1,\infty)$ and $q \in[0, 1]$.
We define $\amcn(x,q)$ by the inverse function of $\mathrm{F}_{1,p}(x, q)$, so that
\begin{align*} %\label{eq:am1}
x= \int_0^{\amcn(x,q)} \frac{|\cos\theta|^{1-\frac{2}{p}} }{ \sqrt{1-q^2\sin^2\theta} }\,d\theta \qquad \text{for}\ x\in\R.
\end{align*}
We define $\sn_p(x,q)$, \emph{$p$-elliptic sine function} with modulus $q$, by 
\begin{align}\label{eq:sn_p}
\sn_p (x, q):= \sin \amcn(x,q), \quad x\in \R,
\end{align}
and define $\cn_p(x,q)$, \emph{$p$-elliptic cosine function} with modulus $q$, by 
\begin{align}\label{eq:cn_p}
\cn_p (x, q):= |\cos \amcn(x,q)|^{\frac{2}{p}-1} \cos \amcn(x,q), \quad x\in \R.
\end{align}
In addition, we also define $\amdn(x,q)$ by the inverse function of $\mathrm{F}_{2,p}(x, q)$,
\begin{align}\label{eq:am2}
x= \int_0^{\amdn(x,q)} \frac{1}{ \sqrt[p]{1-q^2\sin^2\theta} }\,d\theta \qquad \text{for}\ x\in\R,
\end{align}
and define $\dn_p(x,q)$, \emph{$p$-delta amplitude function} with modulus $q$, by 
\begin{align}\label{eq:dn_p}
\dn_p (x, q):= \sqrt[p]{1-q^2 \sin^2 \big(\amdn(x, q)\big)}, \quad x\in \R.
\end{align}
\end{definition}
%%%%%%%%%%%%%%%%%%%%%%%%%%%%%%%%%%%%%%

\begin{remark}
The amplitude functions $\amcn(\cdot,q)$ and $\amdn(\cdot,q)$ are strictly increasing functions from the domain $\R$ to the range $\R$ unless $q=1$ and $p\leq2$, while if $q=1$ and $p\leq2$ then their ranges are the bounded interval $(-\frac{\pi}{2},\frac{\pi}{2})$.
\end{remark}

\begin{remark}
The above functions for $p=2$, namely $\sn_2$, $\cn_2$, and $\dn_2$, coincide with the classical Jacobian elliptic functions.
It follows from \eqref{eq:sn_p} and \eqref{eq:cn_p} that $\sn_p$ and $\cn_p$ satisfy a generalized trigonometric identity of the form
\begin{align}\label{eq:0916-1}
    |\sn_p(x, q)|^2 + |\cn_p (x, q)| ^p = 1.
\end{align} 
On the other hand, by definition we have $\amcn(x, q) \not\equiv \amdn(x, q)$ unless $p=2$, and hence $\cn_p$ and $\dn_p$ are defined somewhat independently.
In particular, we do not find any simple relation between $\dn_p$ and $\cn_p$ (or $\sn_p$); for example, if $p\neq2$, then
$
\dn_p^2(x, q) - q^2 |\cn_p(x, q)|^p \not\equiv 1-q^2.
$
\end{remark}

\begin{remark} 
Our definitions of $p$-elliptic functions are different from Takeuchi's generalization in \cite{Takeuchi12} (although in particular his $\sn_{p2}$ and $\cn_{p2}$ also satisfy \eqref{eq:0916-1}, cf.\ \cite[page 28]{Takeuchi12}).
The main difference already appears in the periods of those functions, i.e., the corresponding complete elliptic integrals of the first kind.
In fact, on one hand, we have the following alternative forms of our $\Kcn(q)$ and $\Kdn(q)$:
\begin{equation*}
    \Kcn(q)=\int_0^1\frac{dz}{\sqrt[p]{1-z^2}\sqrt{1-q^2z^2}}, \quad \Kdn(q)=\int_0^1\frac{dz}{\sqrt{1-z^2}\sqrt[p]{1-q^2z^2}}.
\end{equation*}
On the other hand, Takeuchi's definition takes the form of
\[
\K_{pr}(q)=\int_0^1\frac{dz}{\sqrt[p]{1-z^r}\sqrt[p]{1-q^rz^r}},
\]
so that they are different even if $r=2$.
%More careful investigations verify that $\sn_p(\cdot,q)\not\equiv\sn_{p'r}(\cdot,q)$, $\cn_p(\cdot,q)\not\equiv\cn_{p'r}(\cdot,q)$, and $\dn_p(\cdot,q)\not\equiv\dn_{p'r}(\cdot,q)$ for any $p,p',r\in(1,\infty)$ and $q\in(0,1]$.
\end{remark}

Now we also define the $p$-hyperbolic secant function $\sech_p$.
It is well known that the classical Jacobian elliptic functions satisfy
\[\cn(x,1)=\dn(x,1) = \sech x \quad \text{for any} \ \  x\in \R, \]
and this hyperbolic secant appears in a parametrization of the borderline elastica.
From this point of view, we next introduce a generalization of the hyperbolic secant function.
By definition, if $1<p \leq 2$, then $\cn_p(x,1)=\dn_p(x,1)$ in $(-\K_p(1), \K_p(1))=\R$.
However, if $p>2$, then $\cn_p(x,1)$ and $\dn_p(x,1)$ do not agree in the whole $\R$ but only partially, e.g.\ on the bounded interval $ (-\K_p(1), \K_p(1) )$; more precisely, $\cn_p(\cdot,1)$ changes the sign and $2\K_p(1)$-antiperiodic but $\dn_p(\cdot,1)$ is nonnegative and $2\K_p(1)$-periodic.
Here we choose the definition of $\sech_p$ to be the zero extension of those functions from $ (-\K_p(1), \K_p(1) )$ to $\R$, cf.\ Figure~\ref{fig:sech}.
This definition will be useful for describing flat-core $p$-elasticae.

%%%%%%%%%%%%%%%%%%%%%%%%%%%%%%%%%%%%%%
\begin{definition}[$p$-Hyperbolic secant function] \label{def:sech}
Let $p\in(1,\infty)$.
We define 
\begin{align}\label{eq:sech_p} 
 \sech_p x:= 
\begin{cases}
 \cn_p(x,1)=\dn_p(x,1), \quad & x\in (-\K_p(1), \K_p(1)), \\
 0, &x\in \R \setminus (-\K_p(1), \K_p(1)).
\end{cases}
\end{align}
When $1<p \leq 2$, we regard $(-\K_p(1), \K_p(1))$ as $\R$.
\end{definition}
%%%%%%%%%%%%%%%%%%%%%%%%%%%%%%%%%%%%%%

%%%%%%%%%%%%%%%%%%%%%%%%%%%%%% Maple figure ver.1 
\begin{center}
  \begin{figure}[htbp]
  %\begin{minipage}{0.45\textwidth}
      \includegraphics[scale=0.4]{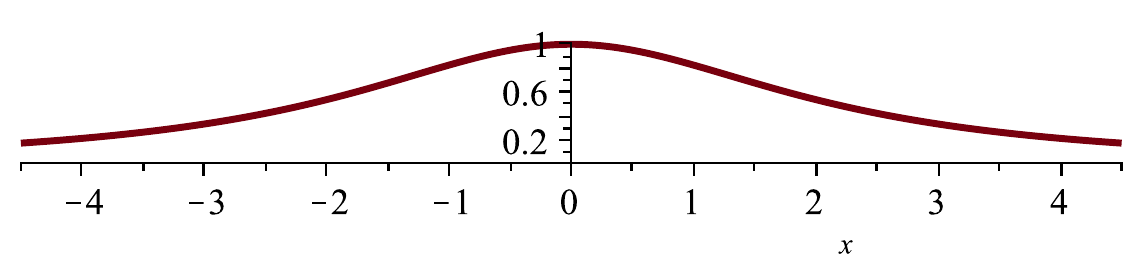}\\
  %\end{minipage}
  %\begin{minipage}{0.45\textwidth}
      \includegraphics[scale=0.2]{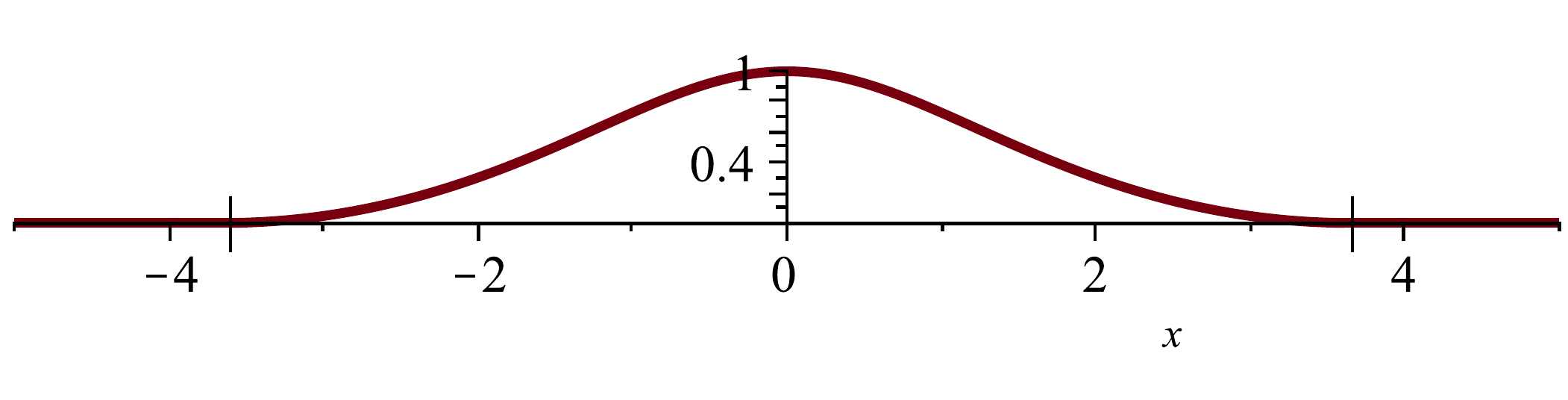}\\
  %\end{minipage}
    %\begin{minipage}{0.45\textwidth}
      \includegraphics[scale=0.4]{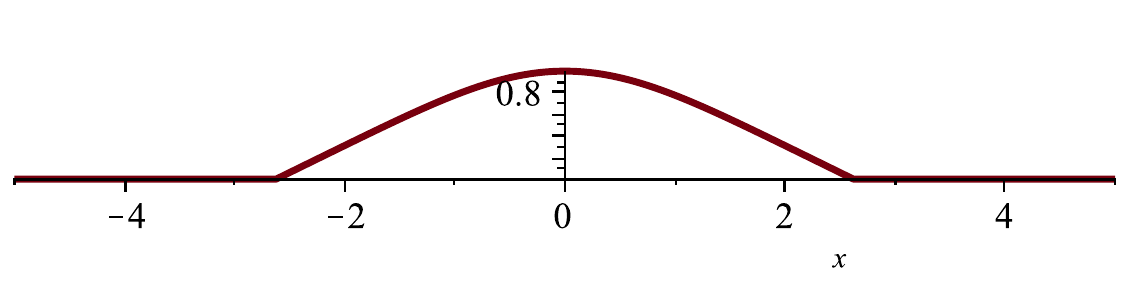}
  %\end{minipage}
      %\begin{minipage}{0.45\textwidth}
      \includegraphics[scale=0.4]{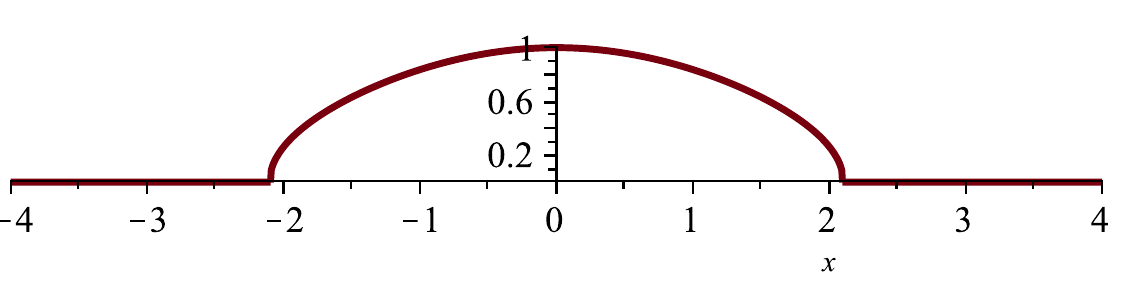}
  %\end{minipage}
  \caption{The graphs of $\sech_p$ with $p=\frac{6}{5}$, $p=3$, $p=4$, and $p=6$ (from top to bottom).
  The positive part is bounded if and only if $p>2$.}
  \label{fig:sech}
  \end{figure}
\end{center}
%%%%%%%%%%%%%%%%%%%%%%%%%%%%%% Maple figure ver.1 

In addition, extending the classical relation $\tanh{x}=\int_0^x(\sech{t})^2dt$, we introduce a new $p$-hyperbolic tangent function:

%%%%%%%%%%%%%%%%%%%%%%%%%%%%%%%%%%%%%%
\begin{definition}[$p$-Hyperbolic tangent function] \label{def:tan_p}
Let $p\in(1,\infty)$.
We define
\[\tanh_p x:=\int_0^x (\sech_p t)^p dt, \qquad x\in\R. \]
\end{definition}
%%%%%%%%%%%%%%%%%%%%%%%%%%%%%%%%%%%%%%

Similar to the classical Jacobian elliptic functions, $\cn_p$, $\dn_p$, and $\sech_p$ satisfy the following fundamental properties such as periodicity:

%%%%%%%%%%%%%%%%%%%%%%%%%%%%%%%%%%%%%%
\begin{proposition}\label{prop:property-cndn}
Let $\sn_p$, $\cn_p$, $\dn_p$, and $\sech_p$ be given by \eqref{eq:sn_p}, \eqref{eq:cn_p}, \eqref{eq:dn_p}, and \eqref{eq:sech_p}, respectively. 
Then the following statements hold:
\begin{itemize}
\item [(i)] For $q\in[0,1)$, $\sn_p(\cdot,q)$ is an odd $2\K_{1,p}(q)$-antiperiodic function on $\R$ and, in $[-\K_{1,p}(q), \K_{1,p}(q)]$, strictly increasing from $-1$ to $1$.
\item [(ii)] For $q\in[0,1)$, $\cn_p(\cdot, q)$ is an even $2\K_{1,p}(q)$-antiperiodic function on $\R$ and, in $[0, 2\K_{1,p}(q)]$, strictly decreasing from $1$ to $-1$. 
\item [(iii)] For $q\in[0,1)$, $\dn_p(\cdot, q)$ is an even, positive, $2\K_{2,p}(q)$-periodic function on $\R$ and, in $[0, \K_{2,p}(q)]$, strictly decreasing from $1$ to $\sqrt[p]{1-q^2 }$. 
\item [(iv-1)] If $1< p \leq 2$, then $\sech_p $ is an even positive function on $\R$, and strictly decreasing in $[0, \infty)$.
Moreover, $\sech_p 0=1$ and $\sech_p x \to 0$ as $x \to \infty$.
\item [(iv-2)] %If $p> 2$, $\sech_p $ is an even function, in $[0, \K_{p} (1)]$, is strictly decreasing from $1$ to $0$.
If $p> 2$, then $\sech_p $ is an even nonnegative function on $\R$, and strictly decreasing in $[0, \K_{p} (1))$. 
Moreover, $\sech_p 0=1$ and 
$\sech_p x \to 0$
as $x \uparrow \K_{p} (1)$.
In particular, $\sech_p $ is continuous on $\R$.
\end{itemize}
\end{proposition}
%%%%%%%%%%%%%%%%%%%%%%%%%%%%%%%%%%%%%%
\begin{proof}
Since all the properties are elementary, we only demonstrate the proof of statement (ii) for $\cn_p$.
First, the evenness of $\cn_p$ follows since $x\mapsto \cos x$ is even and $x \mapsto \amcn(x, q)$ is odd.
Let us check the periodicity of $\cn_p$.
Since $\amcn$ is the inverse of $\Fcn$, and since $\Fcn$ is one-to-one and has periodicity \eqref{eq:period_Ecn}, we have
\begin{align} \notag%\label{eq:0925-7}
\amcn(x + 2n \K_{1,p}(q), q) = \amcn(x, q) + n\pi, \quad \text{for any} \quad n\in\N, \  x\in \R, 
\end{align}
and in particular, $\amcn(x + 2 \K_{1,p}(q), q)=\amcn(x, q) + \pi$.
%Then we obtain 
%\begin{align*}
%\cn_p (x+2 \K_{1,p}(q),q) 
%&=|\cos \amcn(x+2 \K_{1,p}(q),q)|^{\frac{2}{p}-1} \cos \amcn(x+2 \K_{1,p}(q),q) \\
%&=|\cos (\amcn(x,q) +\pi)|^{\frac{2}{p}-1} \cos \amcn(\amcn(x,q) +\pi) \\
%&=- \cn_p (x,q), 
%\end{align*}
Combining this with the definition of $\cn_p$, we obtain $\cn_p (x+2 \K_{1,p}(q),q) =- \cn_p (x,q)$, 
which means that $\cn_p$ is $2\K_{1,p}(q)$-antiperiodic.
%Furthermore, it also follows that
%\[
%\cn_p (0,q) =|\cos 0|^{\frac{2}{p}-1} \cos 0 =1, \quad
%\cn_p (2 \K_{1,p}(q),q) =|\cos \pi|^{\frac{2}{p}-1} \cos \pi =-1.
%\]
Finally, the desired monotonicity of $\cn_p(\cdot,q)$ on the interval $[0, 2\K_{1,p}(q)]$ follows by the fact that $\cn_p$ is the composition of the increasing bijection
$\amcn:[0, 2\K_{1,p}(q)]\to[0,\pi]$, decreasing $\cos:[0,\pi]\to[-1,1]$, and increasing $x\mapsto |x|^{\frac{2}{p}-1}x$ from $[-1,1]$ to $[-1,1]$.
%The periodicity of $\dn_p$ can be obtained analogously to $\cn_p$. The evenness and the monotonicity of $\dn_p$ and $\sech_p$ follow immediately from definition.
\end{proof}

\subsection{Regularity of $p$-elliptic functions}

Recall that the classical Jacobian elliptic functions $\cn_2$, $\dn_2$, and $\sech_2$ are analytic.
Below we shall show that our $p$-elliptic functions may not be smooth, depending on $p$.

First we indicate that $\dn_p$ is always analytic (unless $q=1$).

%%%%%%%%%%%%%%%%%%%%%%%%%%%%%%%%%%%%%%
\begin{proposition}[Regularity of $\dn_p$]\label{thm:regularity_dn}
Let $p\in(1,\infty)$ and $q\in[0,1)$.
Then
$\dn_p(\cdot, q)$ is analytic on $\R$.
\end{proposition}
%%%%%%%%%%%%%%%%%%%%%%%%%%%%%%%%%%%%%%
\begin{proof}
%As $q\in[0,1)$ the function $x \mapsto (1-q^2\sin^2x)^{-\frac{1}{p}}$ belongs to $C^{\infty}(\R)$, and hence its antiderivative $\Fdn(\cdot,q)$ is also smooth.
%Since the derivative $(1-q^2\sin^2x)^{-\frac{1}{p}}$ of $\Fdn(\cdot, q)$ is strictly positive, the inverse $\amdn(\cdot, q)$ of $\Fdn(\cdot, q)$ is also smooth.
%Thus the function $\dn_p(\cdot, q)$ being the composition of $\amdn(\cdot, q)$ and the smooth function $x \mapsto (1-q^2\sin^2x)^{\frac{1}{p}}$ is also of class $C^\infty$.
Since $q\in[0,1)$, the function $x \mapsto (1-q^2\sin^2x)^{-\frac{1}{p}}$ is analytic on $\R$, and hence its antiderivative $\Fdn(\cdot,q)$ is also analytic.
Since the derivative $(1-q^2\sin^2x)^{-\frac{1}{p}}$ of $\Fdn(\cdot, q)$ is strictly positive, the inverse $\amdn(\cdot, q)$ of $\Fdn(\cdot, q)$ is also analytic.
% Thus the function $\dn_p(\cdot, q)$ being the composition of $\amdn(\cdot, q)$ and the analytic function $x \mapsto (1-q^2\sin^2x)^{\frac{1}{p}}$ is also analytic.
Thus $\dn_p(\cdot, q)$, a composition of analytic functions, is also analytic.
\end{proof}

More delicate is the regularity of $\cn_p$ and $\sech_p$ since they are defined through a composition of functions which may not be smooth, including
\begin{align*}%\label{eq:1114-1}
\Psi(x):= |x|^{\frac{2}{p}-1}x.
\end{align*}
% Indeed, $\cn_p$ and $\sech_p$ also have less regularity at points where $\cos\amcn$ vanishes,
% in view of 
% \[
% \cn_p(x,q)= \Psi\big(\cos\amcn(x,q) \big).
% \]
Let $Z_{p,q}$ be the set of zero points of $\cn_p(\cdot, q)$, namely
\[Z_{p,q}:=\Set{ x\in \R | \cn_p(x, q) =0  } 
=\Set{ (2n+1)\K_{1,p}(q) | n\in \Z }.\]

%%%%%%%%%%%%%%%%%%%%%%%%%%%%%%%%%%%%%%
\begin{proposition}[Regularity of $\cn_p$] \label{thm:regularity_cn}
Let $p\in(1,\infty)$ and $q\in[0,1)$. 
Then $cn_p(\cdot, q)$ is analytic on $\R\setminus Z_{p,q}$.
%\begin{align} \label{eq:1205-1}
%\cn_p(\cdot, q) \text{ is analytic on } \R\setminus Z_{p,q}.
%\end{align}
Furthermore, the following statements hold:
\begin{itemize}
\item[(i)] If $\frac{1}{p-1}$ is an odd integer, then $\cn_p(\cdot, q)$ is analytic on $\R$.
%$\cn_p(\cdot, q) \in C^{\infty}(\R)$.
\item[(ii)] If $\frac{1}{p-1}$ is an even integer, then 
$\cn_p(\cdot, q) \in W^{m_p, \infty}(\R)$ but $\cn_p(\cdot, q) \notin  W^{m_p+1,1}_{\rm loc}(\R)$.
\item[(iii)] If $\frac{1}{p-1}$ is not an integer, then 
$\cn_p(\cdot, q) \in W^{m_p, r}_{\rm loc}(\R)$ for any $r\in[1,r_p)$, but $\cn_p(\cdot, q) \notin  W^{m_p,r_p}_{\rm loc}(\R)$. 
\end{itemize}
Here $m_p:=\lceil \frac{1}{p-1} \rceil\geq 1$ and $r_p:=(m_p-\frac{1}{p-1})^{-1} \in(1, \infty)$ as in Theorem \ref{thm:Optimal_regularity_1}.
\end{proposition}

We postpone the detailed arguments for Proposition~\ref{thm:regularity_cn} to Appendix~\ref{sect:proof-regularity-cn-sech} since the rigorous proof involves explicit but delicate and lengthy computations.
Roughly speaking, in cases (ii) and (iii) we explicitly compute the derivatives by using our explicit definitions, while in case (i) we rather appeal to the analytic structure of ODE satisfied by $\cn_p$.

A similar argument also works for $\sech_p$, again postponed to Appendix~\ref{sect:proof-regularity-cn-sech}.

%%%%%%%%%%%%%%%%%%%%%%%%%%%%%%%%%%%%%%
\begin{proposition}[Regularity of $\sech_p$] \label{thm:regularity_sech}
If $p\in(1,2]$, then $\sech_p$ is analytic on $\R$.
If $p\in(2,\infty)$, then $\sech_p$ is analytic on $\R\setminus \{ \pm \K_p(1)\}$, and in addition the following statements hold:
\begin{itemize}
\item[(i)] If $\frac{2}{p-2}$ is an integer, then $\sech_p \in W^{M_p, \infty} (\R)$ but $\sech_p  \notin W^{M_p+1, 1}_{\rm loc}(\R )$.
\item[(ii)] If $\frac{2}{p-2}$ is not an integer, then $\sech_p \in W^{M_p, r}_{\rm loc}(\R)$ for any $r\in [1,R_p)$, but $\sech_p \notin W^{M_p, R_p}_{\rm loc}(\R)$.
\end{itemize}
Here $M_p:=\lceil \frac{2}{p-2} \rceil\geq 1$ and $R_p:=(M_p-\frac{2}{p-2})^{-1} \in(1,\infty)$ as in Theorem \ref{thm:Optimal_regularity_2}.
\end{proposition}
%%%%%%%%%%%%%%%%%%%%%%%%%%%%%%%%%%%%%%

\subsection{Differential equation}

Now we prove that $\cn_p$ and $\dn_p$ actually appear as solutions to \eqref{eq:EL}.
Note carefully that $\cn_p$ is not twice differentiable if $p>2$, cf.\ Proposition~\ref{thm:regularity_cn}.
However, it is worth pointing out that the composite of $x \mapsto |x|^{p-2}x$ and $\cn_p$ is of class $C^2$ and hence smooth enough for considering \eqref{eq:EL} in the classical sense.
 
%%%%%%%%%%%%%%%%%%%%%%%%%%%%%%%%%%%%%%
\begin{lemma} \label{lem:classical-sense}
Let $p\in(1, \infty)$ and $q\in[0,1]$.
Then
\[
|\cn_p(\cdot, q)|^{p-2}\cn_p(\cdot, q) \in C^2(\R), 
\quad |\dn_p(\cdot, q)|^{p-2}\dn_p(\cdot, q) \in C^2(\R). %C^{\infty}(\R).
\]
\end{lemma}
%%%%%%%%%%%%%%%%%%%%%%%%%%%%%%%%%%%%%%
\begin{proof}
Let $q\in[0,1]$.
%Noting \eqref{eq:1204-4}, we have
We compute
\begin{align}\label{transformed_cn_1bibun}
\begin{split}
& \hspace{-30pt}\frac{d}{dx} \Big( | \cn_p(x,q) |^{p-2}\cn_p(x,q) \Big) \\
% &= \frac{d}{dx} \Big(  | \cos \amcn(x,q) |^{1-\frac{2}{p}} \cos \amcn(x,q)\Big) \\
% &=  \big(2-\tfrac{2}{p} \big) | \cos \amcn(x,q)|^{1-\frac{2}{p}}  \big( -\sin \amcn(x,q) \big) \amcn'(x, q) \\
&= - \big(2-\tfrac{2}{p} \big) \sin \amcn(x,q) \sqrt{1-q^2\sin^2 \amcn (x,q)}.
\end{split}
\end{align}
This implies that $ | \cn_p(\cdot ,q) |^{p-2} \cn_p(\cdot ,q) \in C^1(\R)$.
We further compute 
\begin{align} \label{eq:1203-88}
\begin{split}
&\frac{d^2}{dx^2} \Big( | \cn_p(x,q) |^{p-2}\cn_p(x,q) \Big) \\
=& - \big(2-\tfrac{2}{p} \big)  |\cos \amcn (x,q)|^{\frac{2}{p}-1} 
\cos \amcn(x,q)  \Big(1- 2q^2\sin^2 \amcn (x,q)\Big).
\end{split}
\end{align}
%Since the map $t \mapsto  |\cos t|^{\frac{2}{p}-1} \cos t (1- 2q^2\sin^2 t)$ is continuous, we also have
This is continuous, %since $x\mapsto |x|^{\frac{2}{p}-1}x$ is continuous, 
so $ | \cn_p(\cdot ,q) |^{p-2} \cn_p(\cdot ,q) \in C^2(\R)$.

Let us turn to $\dn_p(\cdot, q)$.
%We may assume that $q<1$ since $\cn_p(\cdot,1) = \dn_p(\cdot,1)$.
%Then $\dn_p(\cdot, q)$ is positive and belongs to $C^{\infty}(\R)$, so clearly $|\dn_p(\cdot, q)|^{p-2}\dn_p(\cdot, q) \in C^{\infty}(\R)$.
If $q\in[0,1)$, then $\dn_p(\cdot, q)$ is positive and analytic on $\R$, so clearly $|\dn_p(\cdot, q)|^{p-2}\dn_p(\cdot, q)$ is smooth.
In addition, for $q\in[0,1]$, noting that $\dn_p(x,q)=(\amdn(x,q))'$, we compute 
\begin{align}
&\frac{d}{dx} \Big( | \dn_p(x,q) |^{p-2}\dn_p(x,q) \Big) = -q^2(2-\tfrac{2}{p})\cos\amdn(x,q)\sin\amdn(x,q), \notag \\
&\frac{d^2}{dx^2} \Big( | \dn_p(x,q) |^{p-2}\dn_p(x,q) \Big) \label{eq:2nd-diff-dn_p}\\
&\qquad= -q^2(2-\tfrac{2}{p})\big( 1-2\sin^2\amdn(x,q) \big) \big(1-q^2\sin^2\amdn(x,q) \big)^{\frac{1}{p}}, \notag
\end{align}
which in particular implies that $ | \dn_p(\cdot ,1) |^{p-2} \dn_p(\cdot ,1) \in C^2(\R)$.
\end{proof}

Now we verify that, although $\cn_p$ and $\dn_p$ were defined independently, both $\cn_p$ and $\dn_p$ satisfy the same differential equation up to suitable choices of coefficients.

%%%%%%%%%%%%%%%%%%%%%%%%%%%%%%%%%%%%%%
\begin{proposition} \label{prop:cndn}
Let $A, \beta\in \R$, $\alpha \geq0$, and $q\in[0,1]$. 
Given $L>0$, the following statements hold.
\begin{itemize}
\item[(i)] If $A^2=4\alpha^2q^2$ and $\lambda=2(p-1)|A|^{p-2}\alpha^2(2q^2-1)$, then
\[ k(s)=A\cn_p(\alpha s +\beta, q), \quad s\in [0,L],
\] 
is a solution of \eqref{eq:EL} with the above $\lambda$ in the %classical 
strong sense.
\item[(ii)] If %$q\neq1$, 
$A^2=4\alpha^2$ and $\lambda=2(p-1)|A|^{p-2}\alpha^2(2-q^2)$, then 
\[
k(s)=A\dn_p(\alpha s +\beta, q), \quad s\in [0,L],
\]
is a solution of \eqref{eq:EL} with the above $\lambda$ in the %classical
strong sense.
\end{itemize}
Here the strong sense means that $p(|k|^{p-2}k )''+ (p-1)|k|^p k -\lambda k=0$ holds pointwise.
\end{proposition}
%%%%%%%%%%%%%%%%%%%%%%%%%%%%%%%%%%%%%%
\begin{proof}

%Note first that in case (i), Lemma~\ref{lem:classical-sense} implies that $|k|^{p-2}k\in C^2([0,L])$, while in case (ii) it is clear that $|k|^{p-2}k\in C^2([0,L])$. %$|k|^{p-2}k\in C^\infty([0,L])$.
By Lemma~\ref{lem:classical-sense} we have $|k|^{p-2}k\in C^2([0,L])$ in the both cases.
%Hence in any case, in order to ensure \eqref{eq:EL}, it is sufficient to check that the strong form $p\big(|k|^{p-2}k \big)''+ (p-1)|k|^p k -\lambda k$ holds pointwise. This
Then the strong form $p(|k|^{p-2}k )''+ (p-1)|k|^p k -\lambda k=0$ follows just by direct computations using \eqref{eq:1203-88} and \eqref{eq:2nd-diff-dn_p}, so we only briefly demonstrate case (i).
If $k(s)=A\cn_p(\alpha s +\beta, q)$, then we compute
\begin{align*}
& p\big(|k|^{p-2}k \big)''+ (p-1)|k|^p k -\lambda k \\
=\,& A\cn_p\Big( \big( (p-1)|A|^{p-2}(A^2-2\alpha^2)-\lambda \big) + (p-1)A^{p-2} \big(4\alpha^2 q^2 -A^2 \big)\sn_p^2 \Big),
\end{align*}
which vanishes if $A^2=4\alpha^2q^2$ and $\lambda=2(p-1)|A|^{p-2}\alpha^2(2q^2-1)$.
\end{proof}

In fact, we will prove that the above solutions almost exhaust all the possibilities.
However, if $p>2$, then the degeneracy of the equation yields more possibilities.
We finally discuss this peculiar fact by showing that a continuous family of ``nonunique'' solutions emerge in the flat-core case (corresponding to $\sech_p$ with $p>2$).

To this end we first precisely recall the corresponding solution in the previous theorem.
Let $\lambda>0$.
Then the solution obtained in Proposition~\ref{prop:cndn} with $q=1$ (thus in both cases (i) and (ii)) and $\beta=0$ is represented by
\begin{align*} 
k(s)=2A_{p, \lambda} \sech_p \big( A_{p, \lambda}\, s \big)\quad \text{or}\quad 
k(s)=-2A_{p, \lambda} \sech_p \big( A_{p, \lambda}\, s \big)
\end{align*}
at least if $s\in(-\K_p(1), \K_p(1))$, where
\begin{align} \label{eq:1105-1}
A_{p, \lambda} :=\frac{1}{2}\left( \frac{2\lambda}{p-1}\right)^{\frac{1}{p}}.
\end{align}
%Namely, the above functions are solutions of \eqref{eq:EL}.
Note that the above $k$ satisfies 
\begin{align*}
\begin{cases}
k(s) \neq 0 \quad \text{for any} \ \ s\in \R \quad &\text{if} \quad 1<p\leq2, \\
\Set{ k(s)\neq0 } = (-T_{p,\lambda}, T_{p,\lambda}) \subset \R &\text{if} \quad p>2,
\end{cases}
\end{align*}
where $T_{p,\lambda}$ is the constant defined with $p>2$ by
\begin{align}\notag%\label{eq:T_plambda}
T_{p,\lambda}:=\frac{\K_p(1)}{A_{p, \lambda}}.
\end{align}
%Note that $k(s)=\pm2A_{p, \lambda} \sech_p ( A_{p, \lambda}\, (s-\beta) )$ is also a solution of \eqref{eq:EL} for any $\beta\in\R$ since the choice of $\beta$ in Proposition~\ref{thm:cndn} is arbitrary.
If $p>2$, then the above curvature function may be regarded as a (bounded) ``single mountain'', which corresponds to a ``single loop'' in terms of the corresponding curve.
Proposition~\ref{prop:cndn} with $q=1$ already verifies that both the antiperiodic and periodic extension of the mountain solve \eqref{eq:EL} indeed.

We now observe that in fact any sum of ``disjoint mountains'', possibly with flat parts in the middle, is again a solution.
We first define such disjoint mountains rigorously.

%%%%%%%%%%%%%%%%%%%%%%%%%%%%%%%%%%%%%%
\begin{definition} \label{def:flat-core}
Let $p>2$.
We say that a (curvature) function $k:[0,L] \to \R$ is of \emph{flat-core type} if there exist $\lambda>0$, $N \in \N$, $\{\sigma_i\}_{i=1}^N \subset\{-1, 1\}$, and $\{s_i\}_{i=1}^N\subset(-T_{p,\lambda}, L+T_{p,\lambda})$ such that $s_{i+1} \geq s_i+2T_{p,\lambda}$ ($i=1, \ldots, N-1$) and 
\begin{align}\label{eq:1118-1}
 k(s)=\sum_{i=1}^N \sigma_i  2A_{p, \lambda} \sech_p \big( A_{p, \lambda}\, (s-s_i) \big). 
\end{align}
\end{definition}
%%%%%%%%%%%%%%%%%%%%%%%%%%%%%%%%%%%%%%

%%%%%%%%%%%%%%%%%%%%%%%%%%%%%%%%%%%%%%
\begin{remark}
Notice that $k \not\equiv 0$ in $[0,L]$ since $-T_{p,\lambda}< s_1\leq s_N< L+T_{p,\lambda}$.
In addition, the relation $s_{i+1} \geq s_i+2T_{p,\lambda}$ ensures that the nonzero parts
$$J_i:= \Set{ s\in\R | \sigma_i  2A_{p, \lambda} \sech_p \big( A_{p, \lambda}\, (s-s_i) \big) \neq0 }$$
are mutually disjoint for $i=1,\dots,N$.
The constant $A_{p, \lambda}$ plays the role of a scaling factor, that is, the shape of the graph of ``one mountain'' coincides with that of $2\sech_p$ up to similarity.
We have introduced the notion of flat-core type in order to characterize flat-core $p$-elasticae.
\end{remark}
%%%%%%%%%%%%%%%%%%%%%%%%%%%%%%%%%%%%%%

Let us show that function \eqref{eq:1118-1} is indeed a solution of \eqref{eq:EL}:

%%%%%%%%%%%%%%%%%%%%%%%%%%%%%%%%%%%%%%
\begin{proposition} \label{prop:sech-EL}
Let $p>2$ and $k:[0,L]\to\R$ be of flat-core type.
Then $k$ is a solution of \eqref{eq:EL}.
\end{proposition}
%%%%%%%%%%%%%%%%%%%%%%%%%%%%%%%%%%%%%%
\begin{proof}
It is sufficient to prove that for any given $k:[0,L]\to\R$ of the form \eqref{eq:1118-1} we first canonically extend the domain of $k$ to $\R$ (but the support is still bounded) and argue that equation \eqref{eq:EL} holds for any given $\varphi \in C^{\infty}_{\rm c}(\R)$.
By only looking at the nonzero part of $k$ and accordingly decomposing the domain of the left-hand side of \eqref{eq:EL} into $N$ disjoint intervals, it suffices to check that for each $i=1,\dots,N$,
\begin{align}\label{eq:1221-1}
\int_{s_i-T_{p, \lambda}}^{s_i+T_{p,\lambda}} \Big( p \big(  |k|^{p-2}k \big) \varphi'' +(p-1) |k|^pk\varphi  -\lm k\varphi \Big)\,ds = 0.
\end{align}
Proposition~\ref{prop:cndn} with $q=1$ implies that $2A_{p, \lambda}\sech_p A_{p,\lambda}(\cdot -s_i)$ satisfies \eqref{eq:EL} on $(s_i-T_{p, \lambda},s_i+T_{p, \lambda})$ in the strong sense.
This fact with integration by parts twice for \eqref{eq:1221-1} implies that it is now sufficient to show
\begin{align}\label{eq:1221-2}
\Big[ p|k(s)|^{p-2}k(s) \varphi'(s)
-p\big(|k(s)|^{p-2}k(s) \big)'\varphi(s)
\Big]_{s=s_i-T_{p, \lambda}}^{s=s_i+T_{p, \lambda}} =0.
\end{align}
Since $|k(s_i\pm T_{p, \lambda})|^{p-2}k(s_i\pm T_{p, \lambda})=\sigma_i \sech_p (\pm\K_p(1) )^{p-1} =0$, the first term in \eqref{eq:1221-2} vanishes.
For the second term, inserting $q=1$ into \eqref{transformed_cn_1bibun}, we have
\begin{align} \notag%\label{eq:1202-2}
\frac{d}{dx}\big(|\sech_p x|^{p-2}\sech_p x \big) =-(2-\tfrac{2}{p})\sin\amcn(x,1)|\cos\amcn(x,1)|,
\end{align}
from which we deduce that  
\begin{align*}
&\frac{d}{ds}\big(|k(s)|^{p-2}k(s) \big)\Big|_{s=s_i\pm T_{p, \lambda}} \\
&= -\sigma_i(2-\tfrac{2}{p}) 2^{p-1}A_{p, \lambda}^p \sin\amcn(\pm\K_p(1),1)|\cos\amcn(\pm\K_p(1),1)| =0.
\end{align*}
Hence \eqref{eq:1221-2} holds true.
The proof is now complete.
\end{proof}

%%%%%%%%%%%%%%%%%%%%%%%%%%%%%%%%%%%%%%
%%%%%%%%%%%%%%%%%%%%%%%%%%%%%%%%%%%%%%
%%%%%%%%%%%%%%%%%%%%%%%%%%%%%%%%%%%%%%
\section{Classification, representation, and regularity of $p$-elasticae}\label{sect:mainthm}

By Proposition~\ref{prop:p-ELk}, any $p$-elastica needs to have curvature satisfying \eqref{eq:EL}.
In Propositions~\ref{prop:cndn} and \ref{prop:sech-EL},
we have exhibited examples of solutions to \eqref{eq:EL} with our $p$-elliptic functions. 
In this section, conversely, we shall show that \eqref{eq:EL} has no other solutions, and accordingly we prove Theorems~\ref{thm:Formulae_1}, \ref{thm:Formulae_2}, \ref{thm:Optimal_regularity_1}, and \ref{thm:Optimal_regularity_2}.

To this end, we consider the Cauchy problem of \eqref{eq:EL}
\begin{align} \label{eq:CEL} \tag{CEL}
\begin{cases}
p\big( |k|^{p-2}k \big)'' + (p-1) |k|^pk - \lambda k=0, \\
|k(0)|^{p-2}k(0)=w_0, \quad \frac{d}{ds}\big|_{s=0}\big(|k(s)|^{p-2}k(s)\big)=\dot{w}_0,
\end{cases}
\end{align}
where $w_0, \dot{w}_0\in \R$.
As in \eqref{eq:EL}, the first equation is first understood in a weak sense but, by Lemma~\ref{lem:0812} below, it will turn out that any solution $k \in L^\infty(0,L)$ to \eqref{eq:EL} (cf.\ Proposition \ref{prop:p-ELk}) satisfies $|k|^{p-2}k \in C^2([0,L])$.
Therefore, problem \eqref{eq:CEL} makes sense (see also Definition~\ref{def:CEL}).
The aim of this section is to prove the following complete classification in terms of the initial values $w_0, \dot{w}_0\in \R$:

%%%%%%%%%%%%%%%%%%%%%%%%%%%%%%%%%%%%%%
\begin{theorem}%[Classification of $p$-elastica] 
\label{thm:k-classify}
Let $k:[0,L]\to\R$ be a solution of \eqref{eq:CEL} with some $\lambda \in \R$ in the sense of Definition \ref{def:CEL}.
Then $k$ is given by either of the following cases:
\begin{itemize}
\item[(I)] $k(s) \equiv 0$. 
\item[(II)] $k(s)=A\cn_p(\alpha s +\beta, q)$ for some constants $A, \beta\in \R$, $\alpha \geq0$, and $q\in(0,1)$ such that 
$A^2=4\alpha^2q^2$, $\lambda=2(p-1)|A|^{p-2}\alpha^2(2q^2-1)$, 
$A\cn_p(\beta, q)=|w_0|^{\frac{2-p}{p-1}}w_0$, and  $|A|^{p-2}A\alpha(p-1)|\cn_p(\beta, q)|^{p-2} \cn'_p(\beta, q)=|\dot{w}_0|^{\frac{2-p}{p-1}}\dot{w}_0$.
% (or $(|k(s)|)'|_{s=0}=\dot{w}_0$) 
\item[(III)] {\rm ($1<p \leq 2$)} $k(s)=2\sgn (w_0)A_{p, \lambda}\sech_p(A_{p, \lambda}s +\beta)$ for $A_{p, \lambda}$ given by \eqref{eq:1105-1} and some $\beta\in\R$ such that
$2A_{p, \lambda}\sech_p\beta=|w_0|$.
\item[(III')] {\rm ($p> 2$)}  $k$ is of flat-core type, cf.\ Definition \ref{def:flat-core}.
\item[(IV)] $k(s)=A\dn_p(\alpha s +\beta, q)$ for some constants $A, \beta\in \R$, $\alpha \geq0$, and $q\in(0,1)$ such that 
$A^2=4\alpha^2$, $\lambda=2(p-1)|A|^{p-2}\alpha^2(2-q^2)$, 
$A\dn_p(\beta, q)=|w_0|^{\frac{2-p}{p-1}}w_0$, and  $|A|^{p-2}A\alpha(p-1)\dn_p(\beta, q)^{p-2} \dn'_p(\beta, q)=|\dot{w}_0|^{\frac{2-p}{p-1}}\dot{w}_0$.
%(or $(|k(s)|)'|_{s=0} =\dot{w}_0$), 
\item[(V)] $k(s) \equiv k_0$ for some constant $k_0 \neq 0$.
\end{itemize}
Furthermore, let 
$D_0=D_0(w_0, \dot{w}_0)$ be given by 
\begin{align}\label{eq:def-D_0}
D_0:=p^2 \dot{w}_0^2 + F(w_0),\ \text{where}\ 
F(x):= (p-1)^2 |x|^{\frac{2p}{p-1}} - 2\lambda (p-1) |x|^{\frac{p}{p-1}}.
\end{align}
Then each of the above cases (I)--(V) is attained if and only if the corresponding one of the following mutually exclusive conditions (i)--(v) holds:
\begin{itemize}
\item[(i)] $(1<p\leq 2)$ $w_0 = \dot{w}_0 =0$; $(p>2)$ $k \equiv0$.
\item[(ii)]  $D_0>0$.
\item[(iii)] $(1<p\leq 2)$ $D_0=0$ and $(w_0, \dot{w}_0) \neq (0, 0)$.
\item[(iii')] $(p> 2)$ $D_0=0$ and $k \not\equiv 0$.
\item[(iv)] $\min F<D_0<0$.
\item[(v)]  $D_0=\min F<0$.
\end{itemize}

\end{theorem}
%%%%%%%%%%%%%%%%%%%%%%%%%%%%%%%%%%%%%%

\begin{remark}
The last conditions are useful for distinguishing the cases in terms of the initial data $w_0$ and $\dot{w}_0$.
Indeed, if $p\leq 2$, then all the conditions are described only by the initial data (and the given power and multiplier).
On the other hand, if $p>2$, then both Cases (i) and (iii') allow the choice of $w_0=\dot{w}_0=0$, so we need to use the function $k$ itself to make the conditions mutually exclusive.
\end{remark}

The rest of this section is devoted to the proof of the above classification theorem.
A key technique is to apply the transformation used by Shioji--Watanabe \cite{SW20}, which regards $|k|^{p-2}k$ as an unknown function, in order to transform our possibly degenerate or singular ODE to a semilinear ODE.
The semilinearized equation can always be understood in the classical sense, so that standard uniqueness theory is applicable to the corresponding initial value problem in many cases; in some case we have no general uniqueness, but can prove that the uniqueness-breaking only occurs when $k=0$ and hence the possible configurations have certain rigidity.

\subsection{Semilinearized Euler--Lagrange equation}

Let us introduce the transformation used in \cite{SW20}.
Set 
\[ \Phi (x) := |x|^{\frac{2-p}{p-1}}x, \quad x\in \R. \]
The function $\Phi$ is monotonically increasing and its inverse is $\Phi^{-1}(x)=|x|^{p-2}x$.
When we set $w(s):=|k(s)|^{p-2}k(s)$, then  
\begin{align*} %\label{eq:semilinearize}
    k(s)=|w(s)|^{\frac{2-p}{p-1}}w(s) = \Phi(w(s))
\end{align*}
and \eqref{eq:EL} is transformed to  
\begin{align} \label{eq:0510-1}
\int_0^L \! \Big(p  w \vp'' + (p-1)|w|^{\frac{2}{p-1}}w \vp - \lm |w|^{\frac{2-p}{p-1}}w \vp \Big) ds = 0 \ \  \text{for any} \  \varphi \in C^{\infty}_{\rm c}(0,L).
\end{align}
Note that $w=|k|^{p-2}k \in L^\infty(0,L)$ if $k \in L^\infty(0,L)$.
In this form any weak solution is always a classical solution thanks to the semilinear structure.

%%%%%%%%%%%%%%%%%%%%%%%%%%%%%%%%%%%%%%
\begin{lemma} \label{lem:0812}
If $w\in L^\infty(0,L)$ satisfies \eqref{eq:0510-1} for any $\varphi \in C^{\infty}_{\rm c}(0,L)$, then $w\in C^2([0,L])$ and $w$ also satisfies 
\begin{align}  \label{eq:0610-1}
p w''(s) + (p-1)|w(s)|^{\frac{2}{p-1}}w(s) - \lm |w(s)|^{\frac{2-p}{p-1}}w(s) =0 \quad \text{in} \ [0,L].
\end{align}
\end{lemma}
%%%%%%%%%%%%%%%%%%%%%%%%%%%%%%%%%%%%%%
\begin{proof}
%First, in order to improve the regularity of $w$, we employ ideas used in \cite[Proposition 3.2]{DD18} and \cite[Theorem 3.9]{DDG}.
Fix $\eta \in C^{\infty}_{\mathrm c}(0,L)$ arbitrarily and again as in \cite{DDG} 
set
\[
\varphi(s) := \int_0^s \eta(t)\,dt + \Big( \int_0^L \eta(s)\,ds \Big) \bigg(-\frac{3}{L^2}s^2 + \frac{2}{L^3}s^3 \bigg), \quad s\in [0,L].
\]
Then $\varphi \in W^{2,2}_0(0,L)$ and $\|\varphi \|_{C^0} \leq C \| \eta \|_{L^1}$, 
$\|\varphi' \|_{L^1} \leq C \| \eta \|_{L^1}$ for some constant $C>0$ depending only on $L$. %\textcolor{magenta}{(cf.\ \cite[Proposition 3.2]{DD18} and \cite[Theorem 3.9]{DDG})}
Applying the above $\varphi$ to \eqref{eq:0510-1}, we have 
\begin{align*}
\left| p\int_0^L w \eta' \,ds \right|
&\leq p\| \eta \|_{L^1}\| w \|_{L^\infty} \Big( \int_0^L \Big| -\frac{6}{L^2}+\frac{12}{L^3}s \Big| \, ds \Big)  \\
&\qquad +(p-1)\| w \|_{L^{\infty}}^{\frac{p+1}{p-1}} \int_0^L | \varphi| \, ds
+ \lambda  \| w \|_{L^{\infty}}^{\frac{1}{p-1}} \int_0^L | \varphi| \, ds  
\leq C' \| \eta \|_{L^1},
\end{align*}
where $C'$ is a positive constant independent of $\varphi$.
Thus we obtain $w\in W^{1, \infty}(0,L)$ and it follows from \eqref{eq:0510-1} that for any $\varphi \in C^{\infty}_{\rm c}(0,L)$
\begin{align*}
\left| p\int_0^L  w' \varphi' \,ds \right|
&\leq (p-1)\| w \|_{L^{\infty}}^{\frac{p+1}{p-1}} \int_0^L | \varphi| \, ds
+ \lambda  \| w \|_{L^{\infty}}^{\frac{1}{p-1}} \int_0^L | \varphi| \, ds  
\leq C'' \| \varphi \|_{L^1},
\end{align*}
where $C'':=(p-1)\| w \|_{L^{\infty}}^{\frac{p+1}{p-1}}+ \lambda  \| w \|_{L^{\infty}}^{\frac{1}{p-1}}$.
Therefore, $w\in W^{2, \infty}(0,L)$, and by \eqref{eq:0510-1},
\[
w''(s) = -\frac{(p-1)}{p}|w(s)|^{\frac{2}{p-1}}w(s) + \frac{\lm}{p} |w(s)|^{\frac{2-p}{p-1}}w(s)  \quad \text{a.e. } s\in(0,L).
\]
Since the function in the right-hand side belongs to $C([0,L])$, so does $w''$, 
and hence we obtain $w \in C^2([0,L])$.
Hence, $w$ satisfies \eqref{eq:0610-1} in the classical sense.
\end{proof}

%%%%%%%%%%%%%%%%%%%%%%%%%%%%%%%%%%%%%%
\begin{definition} \label{def:CEL}
We say that $k$ is a solution of \eqref{eq:CEL} if $k\in L^{\infty}(0,L)$ and if the function $w:=|k|^{p-2}k \in L^{\infty}(0,L)$ satisfies \eqref{eq:0510-1} for some $\lambda\in \R$ (and hence $w\in C^2([0,L])$ by Lemma~\ref{lem:0812}), $w(0)=w_0$, and  $w'(0)=w_0$.
\end{definition}
%%%%%%%%%%%%%%%%%%%%%%%%%%%%%%%%%%%%%%

In what follows we investigate properties of $C^2$-solutions to \eqref{eq:0510-1}.
Multiplying $2p w'(s)$ to \eqref{eq:0610-1} and integrating it on $(0,s)$, we obtain the conservation law
\begin{align} \label{eq:0610-4}
p^2 w'(s)^2 + F(w(s)) = p^2 w'(0)^2 + F(w(0)).
\end{align}
For later use we exhibit some elementary properties of $F$ without proof:
%%%%%%%%%%%%%%%%%%%%%%%%%%%%%%%%%%%%%%
\begin{lemma} \label{lem:0612-1}
Let $F$ be the function defined by \eqref{eq:def-D_0}.
Then $F$ is even and $F(x) \to \infty$ as $|x| \to \infty$.
In addition the following statements hold:
\begin{itemize}
\item[(i)] In the case $\lambda>0$, the function $F$ has exactly three local extrema at the points
$0,\pm(\frac{\lambda}{p-1})^{\frac{p-1}{p}}.$
More precisely, $F(0)=0$ is a unique local maximum and $F(\pm (\frac{\lambda}{p-1})^{\frac{p-1}{p}}) <0$ are the global minimum.
Furthermore, the equation $F(x)=0$ has exactly three solutions $x=0,\pm c$, where
\begin{align} \label{eq:0925-3}
c :=\Big(\frac{2\lambda}{p-1} \Big)^{\frac{p-1}{p}}=\big(2 A_{p, \lambda} \big)^{p-1}>0.
\end{align}
\item[(ii)] In the case $\lambda \leq 0$, the function $F$ is convex and hence has only one extremum as the global minimum $F(0)=0$.
\end{itemize}
\end{lemma}
%%%%%%%%%%%%%%%%%%%%%%%%%%%%%%%%%%%%%%

%%%%%%%%%%%%%%%%%%%%%%%%%%%%%%%%%%%%%%
\begin{lemma} \label{lem:double-ido}
Let $\lambda>0$ and $D\in\R$ satisfy $\min F < D<0$.
Then the equation $F(x)-D=0$ has exactly four solutions $x\in \R$. 
Moreover, if $\xi>0$ denotes the largest solution, then all the four solutions are given by 
\begin{align*} %\label{eq:211102-7}
-\xi < - \left( \frac{2\lambda}{p-1}-\xi^{\frac{p}{p-1}} \right)^{\frac{p-1}{p}} < \left( \frac{2\lambda}{p-1}-\xi^{\frac{p}{p-1}} \right)^{\frac{p-1}{p}}<\xi.
\end{align*}
\end{lemma}

\subsection{Classification} \label{subsect:classify}

In this subsection we classify all solutions to \eqref{eq:0610-1} in terms of given initial values $w_0, \dot{w}_0 \in \R$ in the following
\begin{align}
\begin{cases} \label{eq:0610-10}
p w''(s) + (p-1)|w(s)|^{\frac{2}{p-1}}w(s) - \lm |w(s)|^{\frac{2-p}{p-1}}w(s) =0 \quad \text{in} \ (0,L) \\
w(0) = w_0, \quad w'(0)=\dot{w}_0. 
\end{cases}
\end{align}
Recall that the energy level of the Hamiltonian is denoted by $D_0:=  p^2 \dot{w}_0^2 + F(w_0)$.
First, we consider which condition for $p\in(1,\infty)$ and $D_0$ ensures the uniqueness of the problem \eqref{eq:0610-10}.
The uniqueness is nontrivial since equation \eqref{eq:0610-1} can be written by using the derivative $\dot{F}$ of $F$ in form of
$2p^2w''(s) + \dot{F}(w(s))=0,$
and 
\begin{align}\label{eq:211101-1}
\dot{F}(x)= 2p(p-1) |x|^{\frac{2}{p-1}}x -2\lambda |x|^{\frac{2-p}{p-1}}x
\ \ \text{is locally Lipschitz on } \R\setminus\{0\},
\end{align}
but may not be locally Lipschitz at the origin.

%%%%%%%%%%%%%%%%%%%%%%%%%%%%%%%%%%%%%%
\begin{lemma} \label{lem:unique_ini}
Suppose either 
\begin{align} \label{eq:0924-3}
1< p\leq 2\quad \text{or}\quad  D_0 \neq 0.
\end{align}
Then \eqref{eq:0610-10} possesses a unique solution in $C^2([0,L])$.
Moreover, the solution can be uniquely extended to a $C^2$-solution in $\R$.
\end{lemma}
%%%%%%%%%%%%%%%%%%%%%%%%%%%%%%%%%%%%%%
\begin{proof}
%*****[DRAFT OF NEW PROOF]****
The proof proceeds with some standard ODE arguments.
We first indicate that any solution $w$ to \eqref{eq:0610-10} extends to $\R$ thanks to a priori bounds up to first order, since \eqref{eq:0610-4} implies $p^2 |w'(s)|^2 \leq D_0 - \min F$, while \eqref{eq:0610-4} implies $F(w(s)) \leq D_0$ and $F(x)\to\infty$ as $x\to\pm\infty$ (see Lemma~\ref{lem:double-ido}).
The only delicate point is therefore uniqueness, which we discuss below.

In the case that $1<p\leq 2$, the function $\dot{F}$ is clearly locally Lipschitz on $\R$, and hence uniqueness follows by the Picard--Lindel\"of theorem.

Next we turn to the case that $p>2$ and $D_0<0$.
In this case, by \eqref{eq:0610-4}, $F$ takes a negative value, so Lemma~\ref{lem:0612-1} implies that $\lambda>0$ and Lemma~\ref{lem:double-ido} implies the a priori bound $0<\zeta \leq |w| \leq \xi$ for some $\zeta\in(0,\xi)$.
This implies that any solution $w$ is away from zero and hence \eqref{eq:211101-1} implies uniqueness.

Finally we assume that $p>2$ and $D_0>0$.
By \eqref{eq:211101-1}, uniqueness holds except at points $s_0$ where $w(s_0)=0$, so it is sufficient to show that even at such a point $s_0$ uniqueness still holds.
By $D_0>0$ we have $w'(s_0)\neq0$; by symmetry we may assume that $w'(s_0)>0$.
Then around $s_0$, any solution $w$ also solves the equation $w'(s) = \frac{1}{p}\sqrt{D_0 -F(w(s))}$, whose right-hand side is locally Lipschitz around $w=0$ since $D_0>0$.
Hence uniqueness holds at $s_0$.
The proof is complete.
\end{proof}

%%%%%%%%%%%%%%%%%%%%%%%%%%%%%%%%%%%%%%
\begin{lemma} \label{lem:flat-core}
Suppose that \eqref{eq:0924-3} does not hold, i.e., 
\[
p > 2\quad \text{and}\quad  D_0 = 0.
\]
If $\lambda\leq0$, then $w\equiv0$ is a unique solution of \eqref{eq:0610-10}. 
If $\lambda>0$, then $w\in C^2([0,L])$ is a solution of \eqref{eq:0610-10} if and only if 
\begin{align}\label{eq:211101-3}
w(s)= |k(s)|^{p-2}k(s) \quad  \text{for some }  k \text{ of flat-core type}.
\end{align}
\end{lemma}
%%%%%%%%%%%%%%%%%%%%%%%%%%%%%%%%%%%%%%
\begin{proof}
Global existence follows in the same way as in the proof of Lemma \ref{lem:unique_ini}.
By $D_0=0$ and \eqref{eq:0610-4}, any solution $w$ satisfies
\begin{equation}\label{eq:230527-5}
    p^2w'^2+F(w)\equiv0.
\end{equation}
Hence, if $\lambda\leq0$, then by Lemma \ref{lem:0612-1} it is clear that $w\equiv0$.

In what follows we assume $\lambda>0$.
By Proposition \ref{prop:sech-EL} and Lemma~\ref{lem:0812}, any function $w:=|k|^{p-2}k$ with $k$ of flat-core type is a $C^2$-solution to \eqref{eq:0610-1}.
In this case it is easy to check $D_0=0$ by using \eqref{eq:0610-4}, since $w$ may be canonically extended to a global solution with compact support so that $w=w'=0$ holds at some point.

The remaining task is to show that, if $p>2$, $\lambda>0$, and $D_0=0$, then any solution $w\in C^2([0,L])$ to \eqref{eq:0610-10} is of the form \eqref{eq:211101-3}.
Let $I\subset[0,L]$ be any connected component of the nonzero set of $w$, and $a\in I$ be any point.
One easily notices that it is sufficient to show that, for some $\beta\in\R$, the function $w|_{I}$ agrees (on $I$) with $w_{\beta}$ defined by
\begin{equation} \notag%\label{eq:230527-1}
    \text{$w_{\beta}(s):=\sgn\big(w(a)\big)\Big(2A_{p,\lambda}\sech_p(A_{p,\lambda}s+\beta)\Big)^{p-1}$}.
\end{equation}
In particular, $I\subset(\frac{-T_{p,\lambda}-\beta}{A_{p,\lambda}},\frac{T_{p,\lambda}-\beta}{A_{p,\lambda}})$.
Since any $w_{\beta}$ also solves the equation in \eqref{eq:0610-10} which is uniquely solvable whenever $w\neq0$, it is now sufficient to verify that there is $\beta\in\R$ such that $w_\beta(a)=w(a)$ and $w_\beta'(a)=w'(a)$.
Notice that by \eqref{eq:230527-5} we have $F(w(a))\leq 0$, and hence, by Lemma \ref{lem:double-ido} and \eqref{eq:0925-3}, we have $|w(a)|\leq (2A_{p, \lambda})^{p-1}$.
Thanks to this bound we can find $\beta\in \R$ such that
\begin{equation}\label{eq:0529-2}
    w_\beta(a)=w(a), \quad \sgn(w_\beta'(a))=\sgn (w'(a)),
\end{equation}
since $\sech_p$ is an even function on $(-\K_p(1),\K_p(1))$ and decreasing from $1$ to $0$ on $[0,\K_p(1))$ as in Proposition \ref{prop:property-cndn}, and since $w_\beta$ has the same sign as $w(a)\neq0$.
For this $\beta$, since both $w$ and $w_\beta$ satisfy \eqref{eq:230527-5}, we also have $p^2w_\beta'(a)^2+F(w_\beta(a))=p^2w'(a)^2+F(w(a))$, which together with \eqref{eq:0529-2} implies $w_\beta'(a)=w'(a)$.
\end{proof}

We are now in a position to prove Theorem~\ref{thm:k-classify}.

\begin{proof}[Proof of Theorem~\ref{thm:k-classify}]
We prove that if we assume one of the conditions in (i)--(v), then for any solution $k$ in the sense of Definition \ref{def:CEL} the corresponding assertion in (I)--(V) holds true.

\smallskip

\textbf{Case (I).} 
Here we assume (i).
If $p>2$, the assertion in (I) is trivial, while if $p\leq2$, we have $k\equiv0$ by uniqueness in Lemma \ref{lem:unique_ini}.

\smallskip

\textbf{Case (II).} 
Here we assume (ii).
By uniqueness in Lemma \ref{lem:unique_ini} it is sufficient to prove that there is at least one solution to \eqref{eq:CEL} in form of (II).
Let $\mu>0$ denote the positive one of the two solutions of $F(x)-D_0=0$.
Using this $\mu$, we define $A_{\mu}>0$, $q_{\mu}\in(0,1)$, and $\alpha_{\mu}>0$ by
\begin{align*}
A_{\mu}:= \mu^{{\frac{1}{p-1}}}=|\mu|^{\frac{2-p}{p-1}}\mu, \quad q_{\mu}^2:=\left(2-\frac{2\lambda}{|A_{\mu}|^{p} (p-1)}\right)^{-1}, \quad 
\alpha_{\mu}:=\frac{1}{2q_{\mu}}A_{\mu}.
\end{align*}
Here we remark that $q_{\mu}^2\in(0,1)$ is well defined; indeed, if $\lambda\leq0$, then clearly $0<q_{\mu}^2 \leq 1/2$;
if $\lambda>0$, then we have
$\mu>c$, where $c$ is the positive root of $F$, cf.\ \eqref{eq:0925-3}, and hence
$
|A_{\mu}|^p = \mu^{\frac{p}{p-1}} > \frac{2\lambda}{p-1},
$
which implies that $q_\mu^2=(2-\frac{2\lambda}{|A_{\mu}|^{p} (p-1)})^{-1}\in(0,1)$.
Then the triplet $(A_{\mu}, q_{\mu}, \alpha_{\mu})$ satisfies the relations in Proposition~\ref{prop:cndn} so that 
$$k_{\mu, \beta}(s):=A_{\mu} \cn_p(\alpha_{\mu}s+\beta, q_{\mu})$$
is a solution of \eqref{eq:EL} for any $\beta \in \R$.

Finally we address the compatibility with the given initial data.
More precisely, we show that for the above $\mu$ there is some $\beta\in \R$ such that $w_{\beta}(0)=w_0$ and $w_{\beta}'(0)=\dot{w}_0$, where $w_{\beta}:=|k_{\mu,\beta}|^{p-2}k_{\mu,\beta}$.
By definition of $D_0$ and $\mu$, 
we have $
F(w_0) \leq p^2\dot{w}_0^2 + F(w_0) = D_0 = F(\mu)$, and this with
Lemma~\ref{lem:0612-1} implies $|w_0| \leq \mu=|A_{\mu}|^{p-2}A_{\mu}$.
Therefore, as in the last part of the proof of Lemma \ref{lem:flat-core}, we can appeal to the properties of $\cn_p$ in Proposition \ref{prop:property-cndn} to deduce that there is some $\beta\in\R$ such that
\begin{equation}\label{eq:0529-1}
    w_\beta (0) = w_0, \quad \sgn(w_\beta'(0)) = \sgn(\dot{w}_0).
\end{equation}
In addition, since $w_\beta$ solves the equation in \eqref{eq:0610-10} and hence \eqref{eq:0610-4} holds, and since $F(\mu)=D_0=p^2 \dot{w}_0^2 +F(w_0)$, we obtain
$$p^2w'_\beta(0)^2+F(w_\beta(0)) =p^2w'_\beta(-\tfrac{\beta}{\alpha_\mu})^2 +F(w_\beta(-\tfrac{\beta}{\alpha_\mu})) = F(\mu)= p^2 \dot{w}_0^2 +F(w_0),$$
which with \eqref{eq:0529-1} implies $w'_\beta(0)=\dot{w}_0$.

%the function $w(s)=|k_{\mu, \bar{\beta}}(s)|^{p-2}k_{\mu, \bar{\beta}}(s)$ is equal to $w_{\mu}(s+\bar{\beta})$, which is a solution of \eqref{eq:0610-1} with $w(0)=w_0$ and $w'(0)=\dot{w}_0$. Consequently, $k_{\mu, \bar{\beta}}(s)$ is a solution of \eqref{eq:CEL}.

\smallskip

\textbf{Case (III).}
Here we assume (iii).
Note that in this case $\lambda >0$. 
Indeed, if $\lambda\leq0$, then the strictly convex function $D_0$ takes the minimum $0$ only at $(w_0,\dot{w}_0)=(0,0)$.

Now for $\beta\in\R$ we define 
$
k_{\beta}(s):= 2\sgn(w_0)A_{p, \lambda} \sech_p(A_{p, \lambda} s+\beta)
$
with $A_{p, \lambda}$ in \eqref{eq:1105-1}.
Then any $k_\beta$ satisfies the equation in 
\eqref{eq:CEL} by Proposition~\ref{prop:cndn}.
In addition, we can prove that there is some $\beta$ such that $w_{\beta}:=|k_\beta|^{p-2}k_\beta$ satisfies $w_\beta(0)=w_0$ and $w_\beta'(0)=\dot{w}_0$, completely in the same way as in the proof of Lemma \ref{lem:flat-core} (by interpreting $T_{p,\lambda}=\K_p(1)=\infty$ and replacing $(w(a),w'(a))$ with $(w_0,\dot{w}_0)$).

\smallskip

\textbf{Case (III').}
Here we assume (iii'), i.e., $k$ is a nonzero solution to \eqref{eq:CEL} with $D_0=0$.
Such a solution exists only if $p>2$ and $\lambda>0$ by Lemma \ref{lem:unique_ini} and Lemma \ref{lem:flat-core}.
Hence, again by Lemma \ref{lem:flat-core}, any $k$ is of flat-core type.

\smallskip

\textbf{Case (IV).}
Here we assume (iv).
As in Case (III), we may assume $\lambda>0$, and it is sufficient to find at least one solution by uniqueness.
By Lemma~\ref{lem:double-ido}, the equation $F(x)=D_0$ has four solutions.
Let $\xi>0$ denote the maximal one.
Set
\begin{align} \notag%\label{eq:211103-1}
A_{\xi}:= |\xi|^{\frac{2-p}{p-1}}\xi, \quad 
\alpha_\xi:=\frac{A_\xi}{2}, \quad 
q_\xi^2:=2-\frac{2\lambda}{|A_{\xi}|^{p} (p-1)}.
\end{align}
Here we have $q_\xi^2\in(0,1)$, since we can show that $\frac{\lambda}{p-1} < |A_{\xi}|^{p} = \xi^{\frac{p}{p-1}} < \frac{2\lambda}{p-1}$.
Indeed, by Lemma~\ref{lem:0612-1}, $F$ takes its minimum at $(\frac{\lambda}{p-1})^{\frac{p-1}{p}}$, and $F(x)$ is increasing for $x>(\frac{\lambda}{p-1})^{\frac{p-1}{p}}$ while $F$ takes zero at $x=(\frac{2\lambda}{p-1})^{\frac{p-1}{p}}$.
This together with $\min F < D_0 < 0$ implies that  
$(\frac{\lambda}{p-1})^{\frac{p-1}{p}} < \xi < (\frac{2\lambda}{p-1})^{\frac{p-1}{p}}$, so that $q_\xi^2\in(0,1)$.
Then the triplet $(A_\xi,q_\xi,\alpha_\xi)$ satisfies the relations in Proposition~\ref{prop:cndn} so that 
$$k_{\xi,\sigma,\beta}(s):=\sigma A_\xi\dn_p(\alpha_\xi s+\beta,q_\xi)$$ 
is a solution of \eqref{eq:EL} for any $\sigma\in\{1,-1\}$ and $\beta\in\R$.

Finally, for the above $\xi$, we find suitable $\sigma$ and $\beta$ compatible with the initial data, i.e., $w_{\sigma,\beta}(0)=w_0$ and $w'_{\sigma,\beta}(0)=\dot{w}_0$ with $w_{\sigma,\beta}:=|k_{\xi,\sigma,\beta}|^{p-2}k_{\xi,\sigma,\beta}$.
Since $w_0\neq0$ by $D_0<0$, we can take $\sigma:=\sgn(w_0)$.
In order to find $\beta$, we first note that by Lemma~\ref{lem:double-ido}, the (unique) positive solution of $F(x)=D_0$ other than $\xi$ is given by 
$\zeta:=\big(\frac{2\lambda}{p-1} - \xi^{\frac{p}{p-1}} \big)^{\frac{p-1}{p}}.$
From $F(w_0) \leq p^2 \dot{w}_0 + F(w_0) =D_0$ 
we infer that $ \zeta \leq |w_0| \leq \xi$.
This bound together with the properties of $\dn_p$ in Proposition \ref{prop:property-cndn} implies, as in Case (II), that there is some $\beta$ such that $w_{\sigma,\beta}(0)=w_0$ and $\sgn(w'_{\sigma,\beta}(0))=\sgn(\dot{w}_0)$ with $\sigma:=\sgn(w_0)$.
In addition, again as in Case (II), by the conservation law \eqref{eq:0610-4} and the fact that $F(\xi)=D_0=p^2\dot{w}_0^2+F(w_0)$, we also deduce that $w'_{\sigma,\beta}(0)=\dot{w}_0$.

\smallskip

\textbf{Case (V).} 
Here we assume (v).
In this case we also have uniqueness, so only need to find one solution.
As in Cases (III) and (IV), we have $\lambda>0$.
In addition, since $p^2\dot{w}_0^2 +F(w_0)=D_0=\min F$, we have $\dot{w}_0=0$ and $F(w_0)=\min F$.
By Lemma \ref{lem:0612-1}, $|w_0|=(\frac{\lambda}{p-1})^{\frac{p-1}{p}}$.
Hence $k(s) := \sgn(w_0)(\frac{\lambda}{p-1})^{\frac{1}{p}}$
is a solution of \eqref{eq:CEL} by direct computation.
\end{proof}

%%%%%%%%%%%%%%%%%%%%%%%%%%%%%%%%%%%%%%
\begin{remark}
Of course, the classification of classical elasticae with respect to the Hamiltonian energy level is well known both in terms of curvature as well as of tangential angle (in analogy with the simple pendulum equation).
In particular, a detailed classification involving initial values of curvature as in Theorem \ref{thm:k-classify} has been already obtained in \cite[Proposition 3.3]{Lin96} for $p=2$.
\end{remark}
%%%%%%%%%%%%%%%%%%%%%%%%%%%%%%%%%%%%%%
\begin{remark}
Ardentov--Lokutsievskiy--Sachkov deal with another (anisotropic) type of ``$p$-elastica'' in the context of optimal control \cite[Section 5.10]{ALS21}.
This is different from our (isotropic) $p$-elastica, but interestingly they have the common feature that non-uniqueness induced by degeneracy occurs only in the special energy level of Hamiltonian corresponding to the separatrix in the phase portrait, i.e., the ``borderline'' case.
%Note that our result here is described in terms of curvature, while Ardentov--Lokutsievskiy--Sachkov use an anisotropic tangential angle.
\end{remark}

\subsection{Explicit parametrization}

We now apply Theorem \ref{thm:k-classify} to the proof of Theorem~\ref{thm:Formulae_1} and Theorem \ref{thm:Formulae_2}.

\begin{proof}[Proof of Theorem~\ref{thm:Formulae_1} and Theorem \ref{thm:Formulae_2}]
Let $\gamma$ be a $p$-elastica and $k$ be the signed curvature of $\gamma$.
Up to reparametrization we may assume that $\gamma$ is parametrized by the arclength.
By Proposition~\ref{prop:p-ELk}, $k$ satisfies the Euler--Lagrange equation  \eqref{eq:EL}, and 
it follows from Lemma~\ref{lem:0812} that $|k|^{p-2}k \in C^2([0,L])$.
Then setting 
$
w_0:= |k(0)|^{p-2}k(0)$ and $\dot{w}_0:=\frac{d}{ds}|_{s=0}(|k(s)|^{p-2}k(s) )$
we observe that $k$ is characterized as a solution of \eqref{eq:CEL}.
%Since classification of solutions of \eqref{eq:CEL} has already been obtained in Theorem~\ref{thm:k-classify},
By Theorem \ref{thm:k-classify}, this
$k$ is given by one of Cases I, II, III, III', IV, and V.
Cases I and V are trivial, so we consider the other cases.
In Cases II and IV we will deal with a general $p\in(1, \infty)$ (regardless whether $p\leq2$ or $p>2$).

\smallskip

\textbf{Case II} (\textsl{Wavelike $p$-elastica})\textbf{.}
In this case, the signed curvature $k:[0,L]\to\R$ of $\gamma$ is given by 
$k(s)= \pm2\alpha q \cn_p(\alpha s +\beta, q)$ for some $\alpha>0$, $\beta\in \R$, and $q\in[0,1)$.
Therefore, up to similarity and reparametrization, the curve $\gamma$ is given by a part of a planar curve with curvature
\[
k_{w}(s):= 2q \cn_p(s,q), \quad s\in \R.
\]
It now suffices to show that if we define $\gamma_w$ and $\theta_w$ by
\[
\gamma_w(s) 
:= \begin{pmatrix}
 X_w(s) \\
 Y_w(s)
 \end{pmatrix} +v_w
:= \int_0^s 
 \begin{pmatrix}
 \cos \theta_w(\sigma) \\
 \sin \theta_w(\sigma)
 \end{pmatrix}
d\sigma +v_w, 
\quad 
\theta_w(s):=\int_0^s k_w(\sigma) \,d\sigma,
\]
where $v_w\in \R^2$ will be chosen later, then they are indeed given by the formulae in Case II of Theorem \ref{thm:Formulae_1}.
First, noting that $\frac{da}{dx}=| \cos a(x)|^{\frac{2}{p}-1}\sqrt{1-q^2\sin^2 a(x)}$, by the change of variables $\phi=\amcn(\sigma,q)$
we see that
\begin{align*}
\theta_w(s)&=\int_0^s 2q \cn_p(\sigma, q)\,d\sigma 
= \int_0^{\amcn(s,q)}\frac{2q\cos \phi}{\sqrt{1-q^2\sin^2 \phi}} \,d\phi \\
&=\left[2\arcsin\big(q\sin \phi \big)\right]_{\phi=0}^{\phi=\amcn(s,q)}=2\arcsin\big(q\sn_p (s,q) \big).
\end{align*}
Therefore, it also follows that
\begin{align*}
X_w(s)&=\int_0^s \Big( 1-2\sin^2\tfrac{\theta_w(\sigma)}{2} \Big) \,d\sigma 
=\int_0^s \Big( 1-2q^2\sn_p^2(\sigma, q) \Big) \,d\sigma \\
&=\int_0^{\amcn(s,q)}\frac{1-2q^2\sin^2 \phi}{\sqrt{1-q^2\sin^2 \phi}} |\cos\phi|^{1-\frac{2}{p}} \,d\phi \\
&=2\E_{1,p}\big( \amcn(s,q), q \big) - \mathrm{F}_{1,p}\big(\amcn(s,q),q \big)
=2\E_{1,p}\big( \amcn(s,q), q \big) -s.
\end{align*}
Noting that $\theta_w(s)/2 = \arcsin (q\sn_p (s,q)) \in (-\pi/2, \pi/2)$, we have
\begin{align*}
Y_w(s)&=\int_0^s \sin \theta_w(\sigma) \,d\sigma 
=\int_0^s 2\sin \tfrac{\theta_w(\sigma)}{2}\cos \tfrac{\theta_w(\sigma)}{2} \,d\sigma \\
%&=\int_0^s 2\sin \frac{\theta_w(\sigma)}{2} \sqrt{1-\sin^2 \tfrac{\theta_w(\sigma)}{2}} \,d\sigma 
&=\int_0^s 2q\sn_p(\sigma, q) \sqrt{1-q^2\sn_p^2(\sigma, q) } \,d\sigma 
=\int_0^{\amcn(s,q)} 2q\sin \phi |\cos\phi|^{1-\frac{2}{p}} \,d\phi \\
&=\left[-2q(2-\tfrac{2}{p})^{-1} |\cos\phi|^{1-\frac{2}{p}}\cos\phi \right]_{\phi=0}^{\phi=\amcn(s,q)} \\
%&=q\frac{p}{p-1}\Big(1- |\cos\amcn(s,q)|^{1-\frac{2}{p}}\cos\amcn(s,q) \Big) \\
&= q\frac{p}{p-1}\Big(1- |\cn_p(s,q)|^{p-2}\cn_p(s,q) \Big) .
\end{align*}
Hence we obtain \eqref{eq:EP2} by choosing $v_w=(0, -\frac{qp}{p-1})$.

\smallskip

\textbf{Case III} (\textsl{Borderline $p$-elastica: $1<p\leq2$})\textbf{.} 
In this case the curvature is of the form 
$k(s)=2\sgn (w_0)A_{p, \lambda}\sech_p(A_{p, \lambda}s +\beta)$.
Analogous to the previous case, up to similarity and reparametrization, it suffices to consider
\begin{align*}
    k_{b}(s)&:= 2 \sech_p s, \ \  
\theta_b(s):=\int_0^s k_b(\sigma) \,d\sigma, \\
\gamma_b(s) 
&:= \begin{pmatrix}
 X_b(s) \\
 Y_b(s)
 \end{pmatrix} + v_b
:= \int_0^s 
 \begin{pmatrix}
 \cos \theta_b(\sigma) \\
 \sin \theta_b(\sigma)
 \end{pmatrix}
d\sigma + v_b,
\end{align*}
where $v_b\in \R^2$ will be chosen later, and to verify that they coincide with those in Case III of Theorem \ref{thm:Formulae_1}.
By \eqref{eq:am2} and \eqref{eq:dn_p},
\begin{align}\label{eq:1204-5}
\frac{d}{ds} \amdn(s,q) = \sqrt[p]{1-q^2\sin^2\amdn(s,q)} = \dn_p(s,q) \quad (q\in[0,1]).
\end{align}
By $\sech_p s= \dn_p(s,1)$ and $\amcn(s,1)=\amdn(s,1)$, and by \eqref{eq:1204-5} we deduce that
\begin{align*}%\label{eq:theta_b}
\begin{split}
\theta_b(s)&=\int_0^s 2\dn_p(\sigma,1) \,d\sigma
=2\Big[\amdn(\sigma, 1) \Big]_{\sigma=0}^{\sigma=s} =2\amdn(s, 1) 
=2\amcn(s, 1).
\end{split}
\end{align*}
By definition we have $\sech_p x= \cos^{\frac{2}{p}}\amcn(s, 1)$, and hence 
\begin{align*}%\label{eq:X_b}
%\begin{split}
X_b(s)&=\int_0^s\cos\theta_b(\sigma)\,d\sigma
=\int_0^s \cos(2\amcn(\sigma, 1))\,d\sigma \\
&=\int_0^s \Big( 2\cos^2\amcn(\sigma, 1) -1 \Big)d\sigma 
=2\int_0^s (\sech_p \sigma)^p d\sigma-s 
=2\tanh_p s-s,
%\end{split}
\end{align*}
where we recall that $\tanh_p s:=\int_0^s (\sech_p \sigma)^p d\sigma$.
Moreover, partly for later use, we prepare the following computation which is valid for any $q\in(0,1]$: By the change of variables $\phi=\amdn(s,q)$,
\begin{align}\label{eq:1204-10}
\begin{split}
\int_0^s\sin (2\amdn(\sigma,q))\,d\sigma &= \int_0^{\amdn(s,q)}\frac{\sin 2\phi}{\sqrt[p]{1-q^2\sin^2\phi}}\,d\phi \\ &=\int_0^{\amdn(s,q)}\frac{2\sin \phi}{\sqrt[p]{1-q^2\sin^2\phi}}\cos\phi\,d\phi \\
&=\Big[
-\frac{1}{q^2(1-\tfrac{1}{p})}\big(1-q^2\sin^2\phi\big)^{1-\frac{1}{p}}
\Big]_{\phi=0}^{\phi=\amdn(s,q)} \\
%&=\frac{1}{q^2}\frac{p}{p-1}\bigg(1-\big(1-q^2\sin^2\amdn(s,q)\big)^{1-\frac{1}{p}} \bigg)\\
&=\frac{1}{q^2}\frac{p}{p-1}\Big(1-\dn_p^{p-1}(s,q) \Big).
\end{split}
\end{align}
Taking $q=1$ in \eqref{eq:1204-10}, we obtain
\begin{align*}
%\begin{split}\label{eq:Y_b}
Y_b(s) = \int_0^s\sin\theta_b(\sigma)\,d\sigma = \int_0^s\sin (2\amcn(\sigma, 1))\,d\sigma 
=\frac{p}{p-1}\Big(1-(\sech_p s)^{p-1} \Big).
%\end{split}
\end{align*}
We thus obtain \eqref{eq:EP3-1} by choosing $v_b:=(0,-\frac{p}{p-1})$.

\smallskip

\textbf{Case III'} (\textsl{Flat-core $p$-elastica: $p>2$})\textbf{.}
Note first that by the same computation as in Case III we deduce that the curves $\gamma_b^\pm:[-\K_p(1),\K_p(1)]\to\R^2$ parametrized by \eqref{eq:pm-border} indeed have the tangential angle $\theta_b^\pm(s)=\pm 2\amcn(s,1)=\pm2\amdn(s,1)$ and the curvature $k_b^\pm(s)=\pm 2\sech_ps$ for $s\in[-\K_p(1), \K_p(1)]$.
In particular, $\theta_b^\pm(\K_p(1))=\pm\pi$ and $\theta_b^\pm(-\K_p(1))=\mp\pi$ hold so that
\begin{equation}\nonumber
    (\gamma_b^\pm)'(\K_p(1))=(\gamma_b^\pm)'(-\K_p(1))=
    \begin{pmatrix}
    -1 \\ 0
    \end{pmatrix}.
\end{equation}
Therefore, the curve $\gamma_f$ of the form \eqref{eq:EP3-2} is always well defined as an arclength parametrized $W^{2,p}$-curve, and in addition has curvature of the form
\begin{equation}\label{eq:0307-2}
    k_f(s) = \sum_{j=1}^N \sigma_j 2\sech_p(s-s_j), \quad s_j:=\big( \sum_{i=1}^j L_i \big) + (2j-1)\K_p(1),
\end{equation}
with the same $N\geq1$, $\sigma_j$'s, and $L_j$'s in \eqref{eq:EP3-2}.

Now we enter the proof.
In this case, the curvature $k$ of the $p$-elastica $\gamma$ under consideration is of flat-core type, cf.\ Definition \ref{def:flat-core}.
Up to rescaling of the curve $\gamma$ we may suppose that the curvature is given by 
\begin{equation}\label{eq:0307-1}
    k(s)=\sum_{j=1}^N\sigma_j 2\sech_p(s-s_j), \quad s\in [0, L],
\end{equation}
where $N\geq1$, $\{\sigma_j\}_{j=1}^N\subset\{+,-\}$, and $0\leq s_1<\dots<s_N\leq L$ such that $s_{j+1}\geq s_j+2\K_p(1)$ for $j=1,\dots,N-1$.
After the canonical extension of $k$ from $[0,L]$ to $[\min\{0,s_1-\K_p(1)\},\max\{L,s_N+\K_p(1)\}]$ and shifting the parameter (and redefining $L$), we may additionally suppose that $s_1-\K_p(1)=0$ and $s_N+\K_p(1)=L$.
Now, if we let
\begin{equation}\label{eq:0307-3}
    L_1:=0, \qquad L_j:=s_j-s_{j-1}-2\K_p(1) \quad(j=2,\dots,N),
\end{equation}
then by \eqref{eq:0307-2} we deduce that the signed curvature of the arclength parametrized curve $\gamma_f$ defined in \eqref{eq:EP3-2} with the $L_j$'s in \eqref{eq:0307-3} and the $\sigma_j$'s in \eqref{eq:0307-1} is exactly given by the function $k_f$ in \eqref{eq:0307-2}.
Since any given curvature defines a unique arclength parametrized planar curve up to rigid motion, the proof is now complete.

\smallskip

\textbf{Case IV} (\textsl{Orbitlike $p$-elastica})\textbf{.} 
Similar to Case II and Case III, we may only consider the curvature of the form
\[
k_{o}(s)= 2 \dn_p(s, q), \quad s\in \R
\]
and compute the corresponding curve $\gamma_o$ and the tangential angle $\theta_o$ defined by
\[
\gamma_o(s) 
:= \begin{pmatrix}
 X_o(s) \\
 Y_o(s)
 \end{pmatrix} + v_o
:= \int_0^s 
 \begin{pmatrix}
 \cos \theta_o(\sigma) \\
 \sin \theta_o(\sigma)
 \end{pmatrix}
d\sigma +v_o, 
\quad 
\theta_o(s):=\int_0^s k_o(\sigma) \,d\sigma,
\]
where $v_o$ will be chosen later.
It follows from \eqref{eq:1204-5} that 
\begin{align*} %\label{eq:1204-7}
\theta_o(s)=\int_0^s 2\dn_p(\sigma, q) \,d\sigma 
=2\amdn(s,q).
\end{align*}
With the change of variables $\phi=\amdn(\sigma,q)$ we compute
\begin{align*}
X_o(s)&=\int_0^s \cos \theta_o(\sigma) \,d\sigma 
=\int_0^{\amdn(s,q)}\frac{\cos 2\phi}{\sqrt[p]{1-q^2\sin^2\phi}}\,d\phi \\
&=\frac{1}{q^2}\int_0^{\amdn(s,q)}\frac{2(1-q^2\sin^2 \phi) +q^2-2}{\sqrt[p]{1-q^2\sin^2\phi}}\,d\phi \\
&=\frac{1}{q^2}
\int_0^{\amdn(s,q)} \bigg(
2(1-q^2\sin^2 \phi)^{\frac{p-1}{p}}
+(q^2-2) \frac{1}{\sqrt[p]{1-q^2\sin^2\phi}}
\bigg)d\phi \\
&=\frac{1}{q^2}\bigg( 2\E_{2, \frac{p}{p-1}}\big( \amdn(s,q), q\big)
+(q^2-2) \mathrm{F}_{2, p}\big( \amdn(s,q), q\big)
\bigg)\\
&=\frac{1}{q^2}\bigg( 2\E_{2, \frac{p}{p-1}}\big( \amdn(s,q), q\big)
+(q^2-2) s
\bigg).
\end{align*}
In addition, by using \eqref{eq:1204-10} with $q\in(0,1)$, we have
\begin{align*}
\begin{split}
Y_o(s)&=\int_0^s \sin \theta_o(\sigma) \,d\sigma = \int_0^s \sin(2\amdn(\sigma,q)) \,d\sigma =\frac{1}{q^2}\frac{p}{p-1}\Big(1-\dn_p^{p-1}(s,q) \Big).
\end{split}
\end{align*}
Thus we obtain \eqref{eq:EP4} with the choice of $v_o:=(0, -\frac{p}{q^2(p-1)})$.
\end{proof}

\subsection{Regularity}\label{regularity:p-ela}

Now the regularity results in Theorems \ref{thm:Soft_regularity}, \ref{thm:Optimal_regularity_1}, and \ref{thm:Optimal_regularity_2} directly follow by the already obtained results.

\begin{proof}[Proof of Theorems \ref{thm:Soft_regularity}, \ref{thm:Optimal_regularity_1}, and \ref{thm:Optimal_regularity_2}]
All but the optimality part directly follow from Theorems~\ref{thm:Formulae_1} and \ref{thm:Formulae_2} and %the regularity of $p$-elliptic functions in
Propositions~\ref{thm:regularity_dn}, \ref{thm:regularity_cn}, and \ref{thm:regularity_sech}.
%Propositions~\ref{thm:regularity_dn}, \ref{thm:regularity_cn}, and \ref{thm:regularity_sech} and Corollary~\ref{cor:analyticity-cn_p}.

Concerning the optimality, in the wavelike case with $\frac{1}{p-1}\not\in\N$ in both of Theorems \ref{thm:Optimal_regularity_1} and \ref{thm:Optimal_regularity_2}, the same loss of regularity occurs as $\cn_p$; in particular, even if the curvature $k$ takes zero only at an endpoint, it can be shown that $k\not\in W^{m_p,r_p}$ since the same proof for case (iii) of Proposition \ref{thm:regularity_cn} is valid even for a one-sided neighborhood such as $(x_0,x_0+\delta)$.

In case (ii) of Theorem \ref{thm:Optimal_regularity_1}, the additional assumption $Z \cap (0,L)\neq\emptyset$ ensures that the discontinuity of the $(m_p+3)$-rd derivative like $\cn_p$ indeed occurs in the interior of $[0,L]$.

In case (i) of Theorem \ref{thm:Optimal_regularity_2}, since $\partial Z\neq\emptyset$, the curvature is not identically zero but has a zero, the same loss of regularity occurs as $\sech_p$ (with $p>2$), more precisely as in case (ii) of Proposition \ref{thm:regularity_sech}.

Finally, in case (ii) of Theorem \ref{thm:Optimal_regularity_2}, the assumption that $Z$ is not discrete implies that the curvature (locally) has the same loss of regularity as $\sech_p$ (with $p>2$) around $\pm \K_p(1)$.
The proof is complete.
\end{proof}

\begin{proof}[Proof of Corollary~\ref{cor:regularity-special-case}]
This is a direct consequence of Theorems~\ref{thm:Optimal_regularity_1} and \ref{thm:Optimal_regularity_2}.
\end{proof}

%%%%%%%%%%%%%%%%%%%%%%%%%%%%%%%%%%%%%%
\begin{remark}
A similar regularity result to Theorems~\ref{thm:Optimal_regularity_1} and \ref{thm:Optimal_regularity_2} is also known for a different equation involving the $p$-Laplacian, arising from Sobolev--Poincar\'e type inequalities \cite[Proposition~2.1]{O84}.
\end{remark}
%%%%%%%%%%%%%%%%%%%%%%%%%%%%%%%%%%%%%%

%%%%%%%%%%%%%%%%%%%%%%%%%%%%%%%%%%%%%%
%%%%%%%%%%%%%%%%%%%%%%%%%%%%%%%%%%%%%%
%%%%%%%%%%%%%%%%%%%%%%%%%%%%%%%%%%%%%%
\section{Closed $p$-elasticae} \label{sect:closed} 
%%%%%%%%%%%%%%%%%%%%%%%%%%%%%%%%%%%%%%
%%%%%%%%%%%%%%%%%%%%%%%%%%%%%%%%%%%%%%
%%%%%%%%%%%%%%%%%%%%%%%%%%%%%%%%%%%%%%

In this section, applying our explicit formulae, we classify all closed planar $p$-elasticae; the only closed $p$-elasticae are a circle, a ``figure-eight'', and their multiple coverings.
Here we say that a planar curve $\gamma$ is \emph{closed} if the arclength parametrization $\tilde{\gamma}:[0,L]\to\R^2$ satisfies that $\tilde{\gamma}(0)=\tilde{\gamma}(L)$ and $\tilde{\gamma}'(0)=\tilde{\gamma}'(L)$.

In order to state our main theorem, we first define the figure-eight $p$-elasticae.
To this end, we recall a monotonicity of $q\mapsto {\E_{1,p}(q)}/{\K_{1,p}(q)}$ obtained in \cite{nabe14}.%; here we give a proof for the reader's convenience.

%%%%%%%%%%%%%%%%%%%%%%%%%%%%%%%%%%%%%%
\begin{lemma}[{\cite[Lemma 2]{nabe14}}]\label{lem:nabe-lem2}
Let $Q_p:[0,1]\to \R$ be the function given by
\begin{align*} %\label{eq:1116-4}
Q_p(q):= 2 \frac{\E_{1,p}(q)}{\K_{1,p}(q)}-1 , \quad q\in[0,1).
\end{align*}
Then $Q_p$ is strictly decreasing on $[0,1)$ and satisfies $Q_p(0)=1$ and
\begin{align}\notag%\label{eq:1117-1}
\lim_{q\uparrow 1}Q_p(q)=
\begin{cases}
-1 \quad & \text{if} \ \ 1<p \leq 2, \\
-\dfrac{1}{p-1} & \text{if} \ \ p>2.
\end{cases}
\end{align}
\end{lemma}

By Lemma~\ref{lem:nabe-lem2}, for each $p\in (1, \infty)$ there exists a unique solution $q\in(0,1)$ of 
\begin{align}\label{eq:2E-K}
2 \frac{\E_{1,p}(q)}{\K_{1,p}(q)}-1 =0.
\end{align}
Using this root, we now extend the notion of the classical figure-eight elastica.

%%%%%%%%%%%%%%%%%%%%%%%%%%%%%%%%%%%%%%
\begin{definition}[Figure-eight $p$-elastica]
For each $p\in(1, \infty)$, let $q^*=q^*(p)\in (0,1)$ be the unique solution of \eqref{eq:2E-K}.
We say that a planar curve $\gamma$ is a \emph{figure-eight $p$-elastica}
if up to similarity and reparametrization $\gamma$ is represented by $\gamma(s)=\gamma_w(s+s_0, q^*)$ with some $s_0\in \R$.
\end{definition}
%%%%%%%%%%%%%%%%%%%%%%%%%%%%%%%%%%%%%%

We also introduce the term ``$N$-fold'' as follows.

%%%%%%%%%%%%%%%%%%%%%%%%%%%%%%%%%%%%%%
\begin{definition}[$N$-fold $p$-elastica] \label{def:N-fold}
Let $p\in(1,\infty)$ and $N\in \N$.
\begin{itemize}
    \item[(i)] Let $q^* = q^*(p) \in(0,1)$ be the unique solution of \eqref{eq:2E-K}.
    We call $\gamma$ an \emph{$N$-fold closed figure-eight $p$-elastica} if, up to similarity and reparametrization, the curve $\gamma$ is represented by
    \[ [0, 4N\Kcn(q^*)] \ni s \mapsto \gamma_w(s+s_0, q^*) \quad\text{for some $s_0 \in \R$.} \]
    \item[(ii)] We call $\gamma$ an \emph{$N$-fold circular $p$-elastica}, or \emph{$N$-fold circle} as usual, if up to similarity and reparametrization, the curve $\gamma$ is represented by
    \[ [0, 2N\pi] \ni s \mapsto \gamma_c(s+s_0) \quad \text{for some $s_0 \in \R$.}  \]
\end{itemize}
\end{definition}
%%%%%%%%%%%%%%%%%%%%%%%%%%%%%%%%%%%%%%

%%%%%%%%%%%%%%%%%%%%%%%%%%%%%%%%%%%%%%
\begin{remark}\label{rem:condition_fold}

Here in order to characterize $N$-fold $p$-elasticae we employ the \emph{length} of the curve \emph{after} dilation, such as $4N\Kcn(q^*)$.
If $\gamma:[0,L]\to \R^2$ is an arclength parametrized $N$-fold closed figure-eight $p$-elastica, then using the scaling factor $\Lambda>0$ we can characterize the signed curvature $k$ of $\gamma$ as 
\begin{align} \label{Nfold_8}
        k(s)= 2\Lambda q^*(p) \cn_p \big( \Lambda (s+s_0), q^*(p) \big), \quad \Lambda=\frac{4N \Kcn( q^*(p) )}{L}, 
\end{align}
where $s$ denotes the arclength parameter of the original $\gamma$ (before dilation).
Note that \eqref{Nfold_8} coincides with the definition in \cite[Definition 2.3]{Miura_LiYau}.

\end{remark}
%%%%%%%%%%%%%%%%%%%%%%%%%%%%%%%%%%%%%%

By definition of $q^*$ and the periodicity of $\gamma_w$ it follows that any $N$-fold figure-eight $p$-elastica is indeed closed, the proof of which can be safely omitted.

%%%%%%%%%%%%%%%%%%%%%%%%%%%%%%%%%%%%%%
\begin{proposition}\label{prop:8_is_closed}
Let $q^* \in(0,1)$ be the unique solution of \eqref{eq:2E-K}.
Let $N\in \N$ and $s_0\in\R$ be arbitrary.
Let $\gamma:[0,4N\Kcn(q^*)]\to \R^2$ be defined by
\[\gamma(s):=\gamma_w(s+s_0, q^*), \quad s\in [0,4N\Kcn(q^*)]. \]
Then, for each integer $n\in [1 ,N]$,
\begin{align*}%\label{eq:8_is_closed}
\gamma(0)=\gamma(4n\Kcn(q^*)), \quad \gamma'(0)=\gamma'(4n\Kcn(q^*)).
\end{align*}
\end{proposition}
%%%%%%%%%%%%%%%%%%%%%%%%%%%%%%%%%%%%%%
% \begin{proof}
% Fix an integer $n\in [1 ,N]$ arbitrarily.
% Let $\theta:[0, 4N\Kcn(q^*)]\to\R$ denote the tangential vector of $\gamma$.
% Then in view of Theorems \ref{thm:Formulae_1} and \ref{thm:Formulae_2}, we see that $\theta(s)=2\arcsin(q^*\sn_p(s +s_0,q^*))$ for $s\in[0, 4N\Kcn(q^*)]$.
% Noting that $\sn_p(\cdot, q^*)$ is $4\Kcn(q^*)$-periodic (see Remark~\ref{rem:sn_p}), we have
% \begin{align*}
% &\theta(4n\Kcn(q^*))- \theta(0) \\
% =\,&2\arcsin(q^*\sn_p(4n\Kcn(q^*) +s_0,q^*)) - 2\arcsin(q^*\sn_p(s_0,q^*)) = 0,
% \end{align*}
% from which $\gamma'(0)=\gamma'(4n\Kcn(q^*))$ follows.

% Next we shall show that $\gamma(0)=\gamma(4n\Kcn(q^*))$. 
% Notice that $\gamma$ is explicitly given by \eqref{eq:EP2}. 
% It follows from \eqref{eq:0925-7} that
% \[ \amcn(4n\Kcn(q^*)+s_0, q^*) = \amcn(s_0, q^*) +2n\pi,\]
% and we infer from the above relation that 
% \begin{align*}
%     2\Ecn\big( \amcn(4n\Kcn(q^*)+s_0, q^*), q^* \big) 
% =\,&2\Ecn\big( \amcn(s_0, q^*) +2n\pi, q^* \big) \\
% =\,&2 \Ecn\big( \amcn(s_0, q^*), q^* \big) + 8n \Ecn(q^*),
% \end{align*}
% where in the last equality we used \eqref{eq:period_Ecn}.
% This together with \eqref{eq:EP2} yields
% \begin{align*}
% \gamma(4n\Kcn(q^*)) -\gamma(0) 
% =
% \begin{pmatrix}
%  8n \Ecn(q^*) - 4n \Kcn(q^*) \\
%  0
% \end{pmatrix}.
% \end{align*}
% Thus $\gamma(4n\Kcn(q^*)) -\gamma(0)=0$ follows since $q^*$ is a solution of \eqref{eq:2E-K}.
% \end{proof}

We are in a position to state the main theorem in this section. 

%%%%%%%%%%%%%%%%%%%%%%%%%%%%%%%%%%%%%%
\begin{theorem}[Classification of closed $p$-elasticae]\label{thm:classify-closed}
Let $1<p<\infty$.
Then $\gamma$ is a closed planar $p$-elastica if and only if the curve $\gamma$ is either
an $N$-fold circle or an $N$-fold figure-eight $p$-elastica for some $N \in \N$.
\end{theorem}
%%%%%%%%%%%%%%%%%%%%%%%%%%%%%%%%%%%%%%

In particular, all the linear, borderline, orbitlike $p$-elasticae are ruled out in Theorem \ref{thm:classify-closed} as in the case of $p=2$, and also the new flat-core $p$-elasticae can also be ruled out thanks to our formula.
In what follows we give some preparatory results.
The first one is concerning the flat-core case.

%%%%%%%%%%%%%%%%%%%%%%%%%%%%%%%%%%%%%%
\begin{lemma}\label{lem:border-endpoints}
Let $p>2$, and $\gamma_b^\pm:[-\K_p(1),\K_p(1)]\to\R^2$ be defined by \eqref{eq:pm-border}.
Then
\begin{align}\label{eq:border-endpoints}
\gamma_b^\pm(\K_p(1)) - \gamma_b^\pm(-\K_p(1)) =-\frac{2}{p-1}\K_{p}(1)
\begin{pmatrix}
1 \\
0
\end{pmatrix}.
\end{align}
\end{lemma}
%%%%%%%%%%%%%%%%%%%%%%%%%%%%%%%%%%%%%%
\begin{proof}
Since $\tanh_p$ is an odd function (cf.\ Definition~\ref{def:tan_p}) and since $\sech_p (\pm\K_p(1)) =0$, we deduce from \eqref{eq:pm-border} that 
\begin{align}
\label{eq:endpoints-loop}
\gamma_b^\pm(\K_p(1)) - \gamma_b^\pm(-\K_p(1)) =
\begin{pmatrix}
4 \tanh_p (\K_p(1)) -2\K_p(1) \\
0
\end{pmatrix}.
\end{align}
Letting $\xi=\amcn(\pm\K_p(1), 1)$ and using $\amcn(\K_p(1),1)=\pi/2$, we compute
\begin{align*}
\tanh_p (\K_p(1)) 
&= \int_0^{\K_p(1)} (\sech_p t)^p \,dt 
= \int_0^{\K_p(1)} \cos^2 \big( \amcn(t,1) \big) \,dt \\
&= \int_0^{\amcn(\K_p(1), 1)} \cos^2\xi \frac{d\xi}{\cos^{\frac{2}{p}}\xi }
= \E_{1,p}(1).
\end{align*}
%where we used definition of $\E_{1,p}(1)$ in the last equality.
Thus, combining \eqref{eq:endpoints-loop} with the fact that  
\begin{align*}
\gamma_b^\pm(\K_p(1)) - \gamma_b^\pm(-\K_p(1)) 
&= 4\E_{1,p}(1) -2 \K_{p}(1) 
= 2 Q_p(1) \K_{p}(1)
= -\frac{2}{p-1}\K_{p}(1),    
\end{align*}
we obtain \eqref{eq:border-endpoints}.
The proof is complete.
\end{proof}

In particular, Lemma~\ref{lem:border-endpoints} ensures that after one loop $\gamma_b^\pm$ the point moves leftward.
This fact combined with formula \eqref{eq:EP3-2} immediately implies the following characterization of the flat part: 

%%%%%%%%%%%%%%%%%%%%%%%%%%%%%%%%%%%%%%
\begin{corollary}[Flat part on the $x$-axis]\label{cor:parallel-xaxis}
Let $p>2$.
Let $\gamma_f:[0,L]\to\R^2$ be of the form \eqref{eq:EP3-2}, and
$
Z:=\Set{s\in[0,L] | k_f(s)=0 },
$
where $k_f$ denotes the signed curvature of $\gamma_f$.
Then
$
    \gamma_f(Z) \subset \Set{(x_1, x_2) \in \R^2 | x_2=0}. 
$
Moreover, the map $Z\to\gamma_f(Z)$ defined by $s\mapsto \gamma_f(s)$ is injective.
\end{corollary} 
%%%%%%%%%%%%%%%%%%%%%%%%%%%%%%%%%%%%%%

Now we turn to preparation for ruling out the orbitlike case.
%The following lemma was already given in \cite[Lemma 4]{nabe14}, but again we give a proof for completeness.
To this end, we recall the following monotonicity obtained in \cite{nabe14}:

%%%%%%%%%%%%%%%%%%%%%%%%%%%%%%%%%%%%%%
\begin{lemma}[{\cite[Lemma 4]{nabe14}}]\label{lem:nabe-lem4}
Let $p\in(1,\infty)$ and $X_{2,p}:(0,1)\to \R$ be the function given by
\begin{align} 
X_{2,p}(q):=&\, 
\frac{2}{q^2}\mathrm{E}_{2, \frac{p}{p-1}}(q) + \left(1-\frac{2}{q^2}\right) \Kdn(q) \label{eq:def-X_2,p}\\
=&\, 
\int_0^{\frac{\pi}{2}} \frac{\cos 2\theta}{\sqrt[p]{1-q^2\sin^2\theta}}\,d\theta, 
\quad q\in(0,1). \notag 
\end{align}
Then $X_{2,p}(0+)=0$ and $X_{2,p}$ is strictly decreasing on $[0,1)$. In particular, $X_{2,p}(q)<0$ for each $q\in (0,1)$.
\end{lemma}
%%%%%%%%%%%%%%%%%%%%%%%%%%%%%%%%%%%%%%

We are now ready to classify closed planar $p$-elasticae.

\begin{proof}[Proof of Theorem~\ref{thm:classify-closed}]

Thanks to Theorems \ref{thm:Formulae_1} and \ref{thm:Formulae_2}, up to similarity and reparametrization, it suffices to consider $\gamma(s)=\gamma_*(s+s_0)$ with some $s_0\in \R$,  
where $\gamma_*$ is $\gamma_{\ell}$, $\gamma_c$, $\gamma_w$, $\gamma_o$, $\gamma_b$, or $\gamma_f$.
%After reparametrization, we may also assume that $\gamma$ is originally parametrized by the arclength.
Write $\mathcal{L}[\gamma]=L$.
Since it is trivial that $\gamma_\ell$ cannot be closed, and $\gamma_c$ makes $N$-fold circles, we only argue for the other cases.

\smallskip

% \textbf{Case I} (\textsl{Linear $p$-elastica})\textbf{.}
% It is obvious that any linear $p$-elastica is not closed.

%\smallskip

\textbf{Case II} (\textsl{Wavelike $p$-elastica})\textbf{.} 
As in \eqref{eq:EP2}, let us consider 
    \begin{equation*}
      \gamma(s,q) =\gamma_w(s+s_0,q)=
      \begin{pmatrix}
      2 \Ecn(\am_{1,p}(s+s_0,q),q )-  (s+s_0)  \\
      -q\frac{p}{p-1}|\cn_p(s+s_0,q)|^{p-2}\cn_p(s+s_0,q)
      \end{pmatrix},
    \quad s\in [0, L]
    \end{equation*}
with some $q\in [0,1)$ and $s_0\in \R$.
The tangential angle $\theta$ of $\gamma$ and the signed curvature $k$ of $\gamma$ are given by
\begin{align*}
\theta(s)&=\theta_w(s+s_0)=2\arcsin(q\sn_p(s+s_0,q)), \\
k(s)&= k_w(s+s_0)= 2q\cn_p(s+s_0,q),
\end{align*}
where $\theta_w$ and $k_w$ are defined in Theorem~\ref{thm:Formulae_1}.
By the change of variables $s=s(x)=\Fcn(x, q)-s_0$, we have
    \begin{gather*} %\label{eq:wavelike-hat}
      \hat{\gamma} (x,q) :=\gamma(s(x), q)=
      \begin{pmatrix}
      2 \Ecn(x,q) - \Fcn(x,q)  \\
      -q\frac{p}{p-1}|\cos x|^{1-\frac{2}{p}}\cos x
      \end{pmatrix}, \\
\hat{\theta} (x):=\theta(s(x))= 2\arcsin (q\sin x ), \quad \hat{k} (x):=k(s(x))=2q|\cos x|^{\frac{2}{p}-1}\cos x. \notag
    \end{gather*}
Let $x_0$, $x_1$ be the parameters which correspond to endpoints of $\hat{\gamma}$, i.e., 
\[
x_0:=\Fcn^{-1}(s_0, q) =\amcn(s_0, q), \quad x_1:=\Fcn^{-1}(s_0+L, q) =\amcn(s_0+L, q).
\]
Then it is necessary for $\hat{\gamma} (\cdot,q)$ to be a closed curve that 
\begin{gather} 
     2 \Ecn(x_0,q) - \Fcn(x_0,q) = 2 \Ecn(x_1,q) - \Fcn(x_1,q) , \label{eq:1211-1} \\ 
    -q\frac{p}{p-1}|\cos x_0|^{1-\frac{2}{p}}\cos x_0 = -q\frac{p}{p-1}|\cos x_1|^{1-\frac{2}{p}}\cos x_1, \label{eq:1211-2} \\
    \hat{\theta} (x_0) - \hat{\theta} (x_1) \in 2\pi \Z. \label{eq:1211-3} 
\end{gather}
Here \eqref{eq:1211-1} and \eqref{eq:1211-2} come from $\hat{\gamma} (x_0,q) = \hat{\gamma} (x_1,q)$.

Let us show that \eqref{eq:1211-1}, \eqref{eq:1211-2}, and \eqref{eq:1211-3} are equivalent to 
\begin{align}\label{eq:1211-5}
q=q^*(p) \quad\text{and} \quad  
x_1 = x_0 +2N\pi  \ \ \text{for some} \ \ N \in \N.
\end{align}
Indeed, since $\hat{\theta} (\R) \subset (-\pi, \pi)$, we infer from \eqref{eq:1211-3} that $\hat{\theta} (x_0) = \hat{\theta} (x_1)$. 
This together with definition of $\hat{\theta} $ implies that $\sin x_0 = \sin x_1$.
Therefore, \eqref{eq:1211-2} and \eqref{eq:1211-3} hold if and only if 
$x_1 = x_0 + 2N \pi $ for some $N \in \N$.
Moreover, using \eqref{eq:period_Ecn} we get for $x_1 = x_0 + 2N\pi$
\begin{align*}
&\big( 2 \Ecn(x_1,q) - \Fcn(x_1,q) \big) - \big( 2 \Ecn(x_0,q) - \Fcn(x_0,q) \big) \\
=& 2N \big( 2 \Ecn(q) - \K_{1,p}(q) \big).
\end{align*}
This together with Lemma~\ref{lem:nabe-lem2} implies that \eqref{eq:1211-1} holds if and only if $q=q^*(p)$.
Therefore we obtain \eqref{eq:1211-5} as an equivalent condition.

Furthermore, it follows from \eqref{eq:1211-5} that 
\begin{align*} %\label{eq:1212-1}
    k (s) = 2q^*(p) \cn_p (s+s_0, q^*(p)).
\end{align*}
Combining \eqref{eq:period_Ecn} with \eqref{eq:1211-5}, we notice that the length $L$ of $\gamma $ satisfies 
\begin{align*}
\begin{split}
L&= \Fcn(x_1, q^*(p)) -  \Fcn(x_0, q^*(p)) \\
&= \Fcn(x_0+2N\pi, q^*(p)) -  \Fcn(x_0, q^*(p)) = 4N\Kcn(q^*(p)).
\end{split}
\end{align*}
Therefore, if $\gamma$ is a closed wavelike $p$-elastica, then $\gamma$ is an $N$-fold figure-eight $p$-elastica for some $N\in \N$.
Conversely, by Proposition~\ref{prop:8_is_closed}, the endpoints of $N$-fold figure-eight $p$-elasticae for all $N\in \N$ agree up to the first order.

\smallskip

\textbf{Case III} (\textsl{Borderline $p$-elastica: $1<p\leq2$})\textbf{.}
In this case, by Theorem~\ref{thm:Formulae_1}, the tangential angle function of $\gamma_b$ is given by $\theta_b(s)=2\amcn(s,1)=2\amdn(s,1)$, so that $\theta_b$ is strictly monotone and its range is $(-\pi, \pi)$.
Hence there are no points $s_1\neq s_2$ satisfying $\gamma'(s_1) = \gamma'(s_2)$.
Therefore, we conclude that any borderline $p$-elasticae cannot be closed.

\smallskip

\textbf{Case III'} (\textsl{Flat-core $p$-elastica: $p>2$})\textbf{.}
We prove that also this case does not contain any closed curve.
It is sufficient to show that any $\gamma_f:[0,L]\to\R^2$ of the form \eqref{eq:EP3-2} does not admit distinct points $s_1\neq s_2$ such that $\gamma_f(s_1)=\gamma_f(s_2)$ and $\gamma_f'(s_1)=\gamma_f'(s_2)$.
We argue by contradiction so suppose that such points exist.
Denote the connected components of $[0,L]\setminus Z$ by $J_j:=(a_j,a_j+2\K_p(1))$, $j=1,\dots,N$, where $Z$ is defined in Corollary \ref{cor:parallel-xaxis}; then any $\gamma_f|_{J_j}$ coincides with one of $\gamma_b^\pm|_{(-\K_p(1), \K_p(1))}$ up to translation and shifting the parameter.
Note that their tangential angle functions $\theta_b^\pm=\pm2\amcn(s,1)=\pm2\amdn(s,1)$ are strictly monotone and their ranges are $(-\pi,\pi)$ (in particular, the tangent directions cannot be leftward).
Therefore, by $\gamma_f'(s_1)=\gamma_f'(s_2)$ (and $s_1\neq s_2$), we deduce that there are only two possibilities;
\begin{itemize}
    \item[(i)] $s_1,s_2\in Z$, or
    \item[(ii)] $s_1\in J_{j_1}$ and $s_2\in J_{j_2}$ with $j_1\neq j_2$.
\end{itemize}
However neither of the two case is possible.
Indeed, concerning the former case (i), the injectivity in Corollary \ref{cor:parallel-xaxis} (with $s_1\neq s_2$) yields $\gamma_f(s_1)\neq\gamma_f(s_2)$.
Concerning the latter case (ii), by Lemma \ref{lem:border-endpoints} we deduce that there is some $r\neq0$ such that
\[
\gamma_f(a_{j_2}+s) = \gamma_f(a_{j_1}+s) -
\begin{pmatrix}
r \\ 0
\end{pmatrix}
\qquad (0\leq s\leq 2\K_p(1)),
\]
and hence $\gamma_f(s_1)\neq\gamma_f(s_2)$.
This is a contradiction.

\smallskip

\textbf{Case IV} (\textsl{Orbitlike $p$-elastica})\textbf{.} 
We shall show that there is no closed orbitlike $p$-elastica by contradiction, so assume that there exists an orbitlike $p$-elastica $\gamma:[0,L]\to\R^2$ such that $\gamma(0) = \gamma(L)$ and $\gamma'(0) = \gamma'(L)$.
By \eqref{eq:EP4}, we have $\gamma(s)=\gamma_o(s+s_0, q)$ with some $q\in(0,1)$ and $s_0\in\R$. 
The change of variables $s=s(x)=\Fdn(x,q)-s_0$ yields
\begin{equation*}
\hat{\gamma}(x,q) :=\gamma(s(x),q)= \frac{1}{q^2}
      \begin{pmatrix}
      2 \E_{2,\frac{p}{p-1}}(x,q)  + (q^2-2) \Fdn(x,q) \\
      - \frac{p}{p-1}(1-q^2\sin x)^{1-\frac{1}{p}}
      \end{pmatrix},
\end{equation*}
and then, by the formulae in Theorem~\ref{thm:Formulae_1} combined with the change of variables $s=\Fdn(x,q)-s_0$, the tangential angle function $\hat{\theta}(x):=\theta_o(s(x)+s_0)$ and the signed curvature $\hat{k}(x):=k_o(s(x)+s_0)$ are given by 
\[
\hat{\theta}(x) = 2x, \quad \hat{k}(x)=2\big( 1-q^2 \sin x \big)^{\frac{1}{p}}.
\]
Set $x_0:=\Fdn^{-1}(s_0, q)$ and $x_1:=\Fdn^{-1}(s_0+L, q)$.
Since $\gamma$ is a closed curve, we have $\hat{\gamma}(x_0)= \hat{\gamma}(x_1)$ and $\hat{\theta}(x_0)- \hat{\theta}(x_1)\in 2\pi\Z$.
From $\hat{\theta}(x_0)- \hat{\theta}(x_1)\in 2\pi\Z$ it follows that $x_1=x_0 +m\pi $ for some $m\in \N$, and this together with $\hat{\gamma}(x_0)= \hat{\gamma}(x_1)$ implies that
\begin{align*}
%0=&\Big(\tfrac{2}{q^2} \E_{2,\frac{p}{p-1}}(x_1,q)  + (1-\tfrac{2}{q^2}) \Fdn(x_1,q) \Big)  \\
%&\qquad\quad-\Big(\tfrac{2}{q^2} \E_{2,\frac{p}{p-1}}(x_0,q)  + (1-\tfrac{2}{q^2} ) \Fdn(x_0,q) \Big) \\
%=&2m\Big(\tfrac{2}{q^2} \E_{2,\frac{p}{p-1}}(q) + (1-\tfrac{2}{q^2} )\Kdn(q) \Big),
&\Big(\tfrac{2}{q^2} \E_{2,\frac{p}{p-1}}(x_1,q)  + (1-\tfrac{2}{q^2}) \Fdn(x_1,q) \Big)-\Big(\tfrac{2}{q^2} \E_{2,\frac{p}{p-1}}(x_0,q)  + (1-\tfrac{2}{q^2} ) \Fdn(x_0,q) \Big) \\
&=2m\Big(\tfrac{2}{q^2} \E_{2,\frac{p}{p-1}}(q) + (1-\tfrac{2}{q^2} )\Kdn(q) \Big) = 2mX_{2,p}(q) =0, 
\end{align*}
where we also used \eqref{eq:period_Ecn}
and $X_{2,p}$ is defined by \eqref{eq:def-X_2,p}.
%However, by Lemma~\ref{lem:nabe-lem4}, the right-hand side of the above is positive for any $q\in(0,1)$.
%This is a contradiction.
However, this is a contradiction since $X_{2,m}(q)<0$ holds for any $q\in(0,1)$ by Lemma~\ref{lem:nabe-lem4}. 
% \smallskip
% \textbf{Case V} (\textsl{Circular $p$-elastica})\textbf{.} 
% It is clear that $[0,L]\ni s \mapsto \gamma_c(s+s_0)$ is a closed curve if and only if $L=2N\pi$ for some $N\in \N$.
\end{proof}

%%%%%%%%%%%%%%%%%%%%%%%%%%%%%%%%%%%%%%
%%%%%%%%%%%%%%%%%%%%%%%%%%%%%%%%%%%%%%
%%%%%%%%%%%%%%%%%%%%%%%%%%%%%%%%%%%%%%
%%%%%%%%%%%%%%%%%%%%%%%%%%%%%%%%%%%%%%
%%%%%%%%%%%%%%%%%%%%%%%%%%%%%%%%%%%%%%
%%%%%%%%%%%%%%%%%%%%%%%%%%%%%%%%%%%%%%

\appendix

\section{First variation of the $p$-bending energy}\label{sect:First-variation-Bp}

In this section we derive the known formula for the G\^{a}teaux derivative of the $p$-bending energy $\mathcal{B}_p$ for the reader's convenience.
A similar computation can also be found e.g.\ in \cite[Proposition A.1]{BVH}.
Let $I:=(0,1)$.
For an immersed curve $\gamma:I\to\R^n$ ($t\in I$), let $ds:=|\gamma'|dt$ be the line element in the sense of a weighted measure on $I$.
Let $\partial_s$ denote the arclength derivative along $\gamma$, i.e., $\partial_s\psi=\frac{1}{|\gamma'|}\psi'$.
In particular, the curvature vector is then represented by $$\kappa:=\partial_s^2\gamma=\frac{\gamma''}{|\gamma'|^2} - \frac{( \gamma', \gamma'' )\gamma'}{| \gamma' |^4}.$$

%%%%%%%%%%%%%%%%%%%%%%%%%%%%%%%%%%%%%%
\begin{lemma}\label{lem:G-derivative}
Let $p\in(1,\infty)$ and $n\in\N$.
For an immersed curve $\gamma\in W_\mathrm{imm}^{2,p}(I;\R^n)$, let $\B_p$ be the $p$-bending energy defined by
$$\mathcal{B}_p[\gamma]:=\int_I|\kappa|^p ds.$$
Then the G\^ateaux derivative $d\mathcal{B}_p$ of $\B_p$ at $\gamma$ is given by, for $h \in W^{2,p}(I;\R^n)$,
\begin{align*}
\langle d\B_p[\gamma] ,  h \rangle 
    &= \int_I  \Big( (1-2p) |\kappa|^p (\partial_s\gamma, \partial_sh) +p|\kappa|^{p-2}(\kappa, \partial_s^2h)\Big)  ds.
\end{align*}
Similarly, the length functional $\mathcal{L}[\gamma]=\int_\gamma ds$ has the G\^ateaux derivative at $\gamma$ in the form of $\langle d\mathcal{L}[\gamma] ,  h \rangle = \int_I(\partial_s\gamma,\partial_sh)ds$.
\end{lemma}
%%%%%%%%%%%%%%%%%%%%%%%%%%%%%%%%%%%%%%
\begin{proof}
We only argue for $\B_p$.
For $h \in W^{2,p}(I;\R^n)$ and small $\ve\in\R$, let
$\vc_{\ve} := \vc + \ve h \in W_\mathrm{imm}^{2,p}(I;\R^n)$.
Let $s_\varepsilon$ denote its arclength parameter, i.e., $\partial_{s_\varepsilon}\psi:=\frac{1}{|\gamma_\varepsilon'|}\psi'$ and $ds_{\varepsilon}:=|\gamma_\varepsilon'|dt$.
Let $\kappa_\varepsilon:=\partial_{s_\varepsilon}^2\gamma_\varepsilon$.
Then by a direct computation we have
$$
\frac{d}{d\ve} \B_p[\vc_{\ve}] 
= \int_0^1 |\vk_{\ve} |^p \partial_\varepsilon(ds_\varepsilon) 
+ \int_0^1 p \left| \kappa_{\ve} \right|^{p-2} 
\big( \kappa_{\ve} , \partial_{\varepsilon} \kappa_{\ve}  \big) ds,$$
where $\partial_\varepsilon$ denotes the partial derivative with respect to $\varepsilon$, namely,
\begin{align*}
\partial_\varepsilon(ds_\varepsilon)
&= \frac{ (\gamma'_{\ve} , h' )}{| \gamma'_{\ve} |} \,dt = (\partial_{s_\varepsilon}\gamma_\varepsilon,\partial_{s_\varepsilon}h) ds_\varepsilon,\\
\partial_{\ve} \kappa_{\ve}
&= \frac{h''}{|\gamma'_{\ve}|^2} - \frac{(\gamma'_{\ve}, \gamma''_{\ve})h'}{| \gamma'_{\ve} |^4} -2 (\gamma'_{\ve}, h') \left(\frac{\gamma''_{\ve}}{| \gamma'_{\ve} |^4}  - \frac{(\gamma'_{\ve}, \gamma''_{\ve}) \gamma_\varepsilon' }{| \gamma'_{\ve} |^6}\right) \\
& \qquad  - \left( \frac{( \gamma'_{\ve}, h'')+(\gamma''_{\ve}, h') }{| \gamma'_{\ve} |^4} -2 \frac{(\gamma'_{\ve}, \gamma''_{\ve}) (\gamma'_{\ve}, h') }{| \gamma'_{\ve} |^6}  \right) \gamma'_{\ve}\\
&= \partial_{s_\varepsilon}^2h -2 (\partial_{s_\varepsilon}\gamma_\varepsilon,\partial_{s_\varepsilon}h)\kappa_\varepsilon - 
 \big(\partial_{s_\varepsilon}(\partial_{s_\varepsilon}\gamma_\varepsilon,\partial_{s_\varepsilon}h) \big) \partial_{s_\varepsilon}\gamma.
\end{align*}
Using $(\kappa_\varepsilon,\partial_{s_\varepsilon}\gamma)=0$ and letting $\ve= 0$
imply the assertion.
\end{proof}

Note that for any $\gamma\in W^{2,p}_\mathrm{imm}(I;\R^n)$ the arclength function $\sigma(t):=\int_0^t|\gamma'|$ is of class $W^{2,p}$, has strictly positive derivative, and maps $\bar{I}$ to $[0,\mathcal{L}[\gamma]]$.
Hence the arclength reparametrization $\tilde{\gamma}:=\gamma\circ\sigma^{-1}$ is an element of $W^{2,p}_\mathrm{imm}(0,\mathcal{L}[\gamma];\R^n)$ with $|\tilde{\gamma}'|\equiv1$.
Therefore, applying the change of variables $s=\sigma(t)$ to Lemma \ref{lem:G-derivative}, we deduce that $\gamma$ is a $p$-elastica if and only if $\tilde{\gamma}$ satisfies \eqref{eq:0525-1} with $\eta=h\circ\sigma^{-1}\in W^{2,p}(0,L;\R^n)$ for all $h\in C^\infty_{\rm c}(I;\R^n)$, or (by approximation) equivalently for all $\eta\in C^\infty_{\rm c}(0,L;\R^n)$.

%%%%%%%%%%%%%%%%%%%%%%%%%%%%%%%%%%%%%%
\section{Proof of optimal regularity of $\cn_p$ and $\sech_p$} \label{sect:proof-regularity-cn-sech}

\begin{proof}[Proof of Proposition \ref{thm:regularity_cn}]
Fix $q\in [0,1)$ arbitrarily.
It is easy to check the analyticity of $\cn_p(\cdot,q)$ on $\R\setminus Z_{p,q}$ %\eqref{eq:1205-1} 
since $\amcn(\cdot,q)$ is analytic (cf.\ the proof of Proposition~\ref{thm:regularity_dn}) and %$\cos{\amcn(x,q)}\neq0$.
since the function $x\mapsto |x|^{\frac{2}{p}-1}x$ is analytic on $\R\setminus \{0\}$.
We split the rest of the proof into several steps.

\textbf{Step 1} (\textsl{Case (i)})\textbf{.}
In the case of $q\in(0,1)$, by Proposition~\ref{prop:cndn}, $k(s):=2q\cn_p(s, q)$ is a solution of \eqref{eq:EL} with $\lambda=2^{p-1}(p-1)(2q^p-q^{p-2})$.
Then $w:=|k|^{p-2}k$ is a $C^2$-solution to \eqref{eq:0610-1}.
Since $\frac{1}{p-1}=2m-1$ for some integer $m\geq1$, equation \eqref{eq:0610-1} is polynomial, namely 
$$w''+(p-1)w^{4m-1}-\lambda w^{2m-1}=0,$$
so $w$ is analytic, and hence $\cn_p(\cdot, q)=\frac{1}{2q}k=\frac{1}{2q}|w|^{\frac{2-p}{p-1}}w=\frac{1}{2q}w^{2m-1}$ is also analytic.
In the case of $q=0$, we deduce from \eqref{eq:1203-88} that $u(s):=\cn_p(s, 0)$ is a (weak) solution of $(|u|^{p-2}u)''+(2-\frac{2}{p})u=0$.
Then, as in the previous argument, since $w:=|u|^{p-2}u$ is a $C^2$-solution to $w''+(2-\frac{2}{p})w^{2m-1}=0$ with some integer $m\geq1$, we obtain the analyticity of $u=|w|^{\frac{2-p}{p-1}}w=w^{2m-1}$.

\textbf{Step 2} (\textsl{Derivative formulae and lower-order regularity})\textbf{.} 
Now, in order to address (ii) and (iii), we first derive a higher-order derivative formula for $\cn_p$.
Let $m_p:=\lceil\frac{1}{p-1}\rceil$.
Write $a(x)$ instead of $\amcn(x, q)$ for simplicity.
Let $g_0(t):=1$ and inductively define a family $\{g_m\}_{m=0}^{m_p}\subset C^\infty(\R)$ by
\begin{align*}
g_{m+1}(t):=(\tfrac{2}{p}-m)^{-1}\sqrt{1-q^2 t^2} \Big(-\big( \tfrac{2}{p}(m+1) -2m \big) t \, g_m(t) 
+(1-t^2)  g_m'(t) \Big).
\end{align*}
Note that, in cases (ii) and (iii), $\frac{2}{p}-m\neq0$ holds for any integer $m\in[0,m_p]$. 
%Note in particular that $g_m(\pm1)\neq0$ for any $m$.
One can inductively prove that $g_m(\pm1)\neq0$ for any integer $m\in[0,m_p]$.
Then, noting that 
$$a'(x)=| \cos a(x)|^{\frac{2}{p}-1}\sqrt{1-q^2\sin^2 a(x)},$$
one can directly compute the following formula for any integer $m \in[0, m_p]$ and $x\in \R\setminus Z_{p,q}$ by induction on $m$:
\begin{align}
% \frac{d^m}{dx^m}\cn_p(x,q) &=\frac{d^m}{dx^m} \Big(\Psi \big(\cos a(x) \big)  \Big)  \notag\\
\frac{d^m}{dx^m}\cn_p(x,q) &=\frac{d^m}{dx^m} \Big(|\cos a(x)|^{\frac{2}{p}-1}\cos a(x) \Big)  \notag\\
%&= \Psi^{(m)}\big(\cos  a(x)  \big)  \big| \cos a(x)  \big|^{(-1+\frac{2}{p})m}  g_m\big(\sin a(x)  \big) \notag  \\
&=
\begin{cases}
c_m \big|\cos a(x)  \big|^{\tfrac{2}{p}(m+1) -2m} g_m\big(\sin a(x)  \big), \quad & m\text{ : odd}, \\
c_m \big| \cos a(x)  \big|^{\tfrac{2}{p}(m+1) -2m-1}\cos a(x)  g_m\big(\sin a(x)  \big), & m\text{ : even},
\end{cases} 
 \label{eq:1111-1}
\end{align}
where %in the last part we used the elementary formula:
%\[ \Psi^{(m)}(x)=
%\begin{cases}
%c_m |x|^{\frac{2}{p}-m} \quad & m\text{ : odd}, \\
%c_m |x|^{\frac{2}{p}-m-1}x & m\text{ : even}, 
%\end{cases}
%\quad 
$c_m := \prod_{i=1}^{m}(\tfrac{2}{p}-i+1)$.
Note that if $1\leq m\leq m_p:=\lceil \frac{1}{p-1} \rceil$, then $c_m\neq0$.
In what follows we discuss the regularity by using \eqref{eq:1111-1}.

The above formula directly implies lower-order regularity.
Indeed, if $m\in [1, m_p-1]$, then $ m < \frac{1}{p-1}$ and hence $\frac{2}{p}(m+1) -2m > 0$.
This together with \eqref{eq:1111-1} implies that the derivative $\frac{d^m}{dx^m}\cn_p(x, q)$ (computed on $\R\setminus Z_{p,q}$) is continuously extended to $\R$, i.e.,
\begin{align} \label{eq:1112-2}
 \cn_p(\cdot,q) \in C^{m_p-1}(\R).    
\end{align}
In what follows we consider the regularity of order $m_p$ in cases (ii) and (iii).

\textbf{Step 3} (\textsl{Case (ii)})\textbf{.} 
Let $\frac{1}{p-1}$ be an even integer.
Then $\frac{2}{p}(m_p+1) -2m_p-1=-1$, and hence by \eqref{eq:1111-1} we deduce that for $x \in \R\setminus Z_{p,q}$,
%Since $\cn_p(x,q)= \Psi(\cos(a(x)))$, using \eqref{eq:1111-1} with $g(t)\equiv1$, we observe that $\cn_p(\cdot, q) \in C^{m_p-1}(\R)$ and 
\begin{align}\label{eq:1223-4}
\frac{d^{m_p} }{d x^{m_p} } \cn_p (x, q) 
=c_{m_p}\big| \cos{a(x)} \big|^{-1}\cos{a(x)} g_{m_p}\big(\sin{a(x)} \big).
\end{align}
%Moreover, from \eqref{eq:1111-11}, $g_{m_p}(\pm1)\neq0$ also follows.
Since $a(x)$ is continuous and $\sin (a(x)) \in [-1,1]$, the continuity of $g_{m_p}$ implies $\cn_p(\cdot, q) \in W^{m_p, \infty}(\R)$.
Next we show 
\begin{equation} \label{eq:kaidan}
\frac{d^{m_p} }{d x^{m_p} } \cn_p(\cdot, q) \notin C(\R).
\end{equation}
Fix $x_0 \in Z_{p,q}$ arbitrarily.
%Then $a(x_0)$ is either ${\pi}/{2}$ or $-{\pi}/{2}$; 
Then $a(x_0)=\frac{\pi}{2}+n\pi$ for some $n\in\N$;
we only argue for $a(x_0)={\pi}/{2}$.
Since $a$ is strictly increasing in $\R$, we observe from \eqref{eq:1223-4} that
\[
\lim_{x\uparrow x_0}\frac{d^{m_p} }{d x^{m_p} } \cn_p (x, q) =c_{m_p} g_{m_p}(1), 
\quad 
\lim_{x\downarrow x_0}\frac{d^{m_p} }{d x^{m_p} } \cn_p (x, q) =-c_{m_p} g_{m_p}(1).
\]
Combining this with %\eqref{eq:1111-11}
$g_{m_p}(\pm1)\neq0$, we notice that 
$\frac{d^{m_p}}{dx^{m_p}}\cn_p(\cdot,q)$ is not continuous at $x=x_0$
and thus obtain \eqref{eq:kaidan}.
Moreover, this also implies that $\frac{d^{m_p}}{dx^{m_p}}\cn_p(\cdot,q) \notin W^{1,1}(U)$, where $U$ is a neighborhood of $x_0 \in Z_{p,q}$, and hence $\cn_p\not\in W^{m_p+1,1}_\mathrm{loc}(\R)$.

\textbf{Step 4} (\textsl{Case (iii)})\textbf{.} Let $\frac{1}{p-1}$ be not an integer. 
Fix $x_0 \in Z_{p,q}$ arbitrarily.
Then similar to Case (ii), we may assume $a(x_0)={\pi}/{2}$.
Since $\frac{1}{p-1} < m_p$, 
we notice that $\frac{2}{p}(m_p+1)-2m_p<0$.
Therefore, using \eqref{eq:1111-1}, we obtain for $x\in \R\setminus Z_{p,q}$,
\begin{align} \label{eq:1112-4}
\left|\frac{d^{m_p}}{dx^{m_p}} \cn_p (x, q) \right|=
c_{m_p} \big|\cos{a(x)} \big|^{\tfrac{2}{p}({m_p}+1) -2{m_p}} \big|g_{m_p}\big(\sin{a(x)} \big) \big|.
\end{align}
%and it also follows from \eqref{eq:1111-11} that $g_{m_p}(\pm1)\neq0$.
Recalling that $\frac{2}{p}(m_p+1)-2m_p<0$ and $g_{m_p}(\pm1)\neq0$, we see that  $|\frac{d^{m_p}}{dx^{m_p}}\cn_p(x,q)|$ diverges as $x\to x_0$. %has a singularity at $x=x_0$.
For a sufficiently small $\delta>0$ and arbitrary $r \geq1$, we see from \eqref{eq:1112-4} that
\begin{align*}
\mathbf{I}&=\int_{x_0-\delta}^{x_0+\delta}
\left|\frac{d^{m_p}}{dx^{m_p}} \cn_p (x, q) \right|^r \,dx \\
&=
\int_{x_0-\delta}^{x_0+\delta}|c_{m_p}|^r \big|\cos{a(x)} \big|^{\big(\tfrac{2}{p}({m_p}+1) -2{m_p}\big)r} \big|g_{m_p}\big(\sin{a(x)} \big) \big|^r \,dx.
\end{align*}
The change of variables $\xi=a(x)=\amcn(x,q)$ yields, for some $\delta'>0$,
\begin{align*}
\mathbf{I}
&=|c_{m_p}|^r
\int_{\frac{\pi}{2}-\delta'}^{\frac{\pi}{2}+\delta'} \big|\cos\xi \big|^{\big(\tfrac{2}{p}({m_p}+1) -2{m_p}\big)r} \big|g_{m_p}(\sin\xi) \big|^r\frac{|\cos\xi|^{1-\frac{2}{p}}}{\sqrt{1-q^2\sin^2 \xi}} \,d\xi.
\end{align*}
Since $g_{m_p}(\sin \xi)$ is continuous around $\xi=\pi/2$ and satisfies $g_{m_p}(\sin \frac{\pi}{2})\neq0$, the integral $\mathbf{I}$ is finite if and only if
\[
\Big(\frac{2}{p}({m_p}+1) -2{m_p}\Big)r + 1-\frac{2}{p} > -1.
\]
Since $m_p+1-m_p p <0$, and since  $m_p<\frac{1}{p-1}$, the above condition is equivalent to $1 \leq r <r_p$.
%\begin{align*}
%1 \leq r < \frac{1-p}{m_p + 1- m_p p} = r_p.
%\end{align*}
The proof is complete.
\end{proof}

\begin{proof}[Proof of Proposition~\ref{thm:regularity_sech}]
First we address the analyticity.
Since $\amcn(\cdot,1)$ is analytic, since $|\amcn(\cdot,1)| < \frac{\pi}{2}$ on $(-\K_p(1), \K_p(1))$, 
%where we regard $(-\K_p(1), \K_p(1))$ as $\R$ if $p\in(1,2]$, 
and since the function $x\mapsto \cos^{\frac{2}{p}}{x}$ is analytic on $(-\frac{\pi}{2},\frac{\pi}{2})$, we see that $\sech_p$ is analytic on $(-\K_p(1), \K_p(1))$.
If $p>2$, then $\sech_p\equiv0$ on $\R\setminus[-\K_p(1), \K_p(1)]$ so trivially analytic.
Henceforth we discuss optimality for $p\in(2, \infty)$.

\textbf{Step 1} (\textsl{Derivative formulae and lower-order regularity})\textbf{.}
As above we write $a(x)$ instead of $\amcn(x,1)$ throughout this proof.
Note that $a'(x)=(\cos{a(x)})^\frac{2}{p}$ on $(-\K_p(1), \K_p(1))$.
Similar to 
\eqref{eq:1111-1}, we inductively obtain for $m\in[1,M_p]$,
\begin{align} \label{eq:1114-2}
\frac{d^m}{dx^m}\sech_px
&= \big(\cos a(x) \big)^{\frac{2}{p}(m+1)-m}  h_m\big(\sin a(x)\big) , \quad x \in (-\K_p(1), \K_p(1)),
\end{align}
where $\{h_m\}_{m=0}^{M_p}\subset C^{\infty}(\R)$ is defined as follows: $h_0(t):=1$, and inductively define
\begin{align*}
h_{m+1}(t):=-(\tfrac{2}{p}(m+1)-m)th_m(t)+(1-t^2)h'_m(t).
\end{align*}
In particular, $h_m(\pm1)\neq0$ for any integer $m\in[0,M_p]$.
When $1\leq m \leq M_p-1$, we have $m \leq M_p-1 < \frac{p}{p-2} -1 = \frac{2}{p-2}$, and hence
%\[ \frac{2}{p}(m+1)-m= \frac{2}{p} + \frac{2-p}{p}m >0.\]
$\frac{2}{p}(m+1)-m>0$.
%By the same way to obtain \eqref{eq:1112-2} 
Using this positivity and \eqref{eq:1114-2}, by the same way to obtain \eqref{eq:1112-2}, 
we see that
\[
\sech_p \in C^{M_p-1}(\R).
\]

\textbf{Step 2} (\textsl{Case (i)})\textbf{.} Suppose that $\frac{2}{p-2}$ is an integer. 
Then we notice that $\frac{2}{p}(M_p+1)-M_p = 0$, 
and hence in view of \eqref{eq:1114-2} we have
\begin{align*} %\label{eq:1115-1}
 \lim_{x\uparrow \K_p(1)}\frac{d^{M_p}}{dx^{M_p}}\sech_p x
= h_{M_p}(1) \neq0, \quad 
\lim_{x\downarrow -\K_p(1)}\frac{d^{M_p}}{dx^{M_p}}\sech_p x
= h_{M_p}(-1) \neq0.
\end{align*}
Therefore, the $M_p$-th derivative of $\sech_p$ is bounded on $\R$, but not continuous at $x=\pm\K_p(1)$.
In particular, we obtain
%\[\sech_p \in W^{M_p, \infty}(\R), \quad \sech_p  \notin W^{M_p+1, 1}(\R). \]
$\sech_p \in W^{M_p, \infty}(\R)$ and  $\sech_p  \notin W^{M_p+1, 1}(\R)$.

\textbf{Step 3} (\textsl{Case (ii)})\textbf{.} Suppose that $\frac{2}{p-2}$ is not an integer. 
Fix $r\geq 1$ arbitrarily and consider 
\begin{align}\label{eq:1115-2}
\begin{split}
\int_{\R} 
\left| \frac{d^{M_p}}{dx^{M_p}}\sech_p x \right|^r \,dx
=
\int_{-\K_p(1)}^{\K_p(1)} 
\left| \frac{d^{M_p}}{dx^{M_p}}\sech_p x \right|^r \,dx.
\end{split}
\end{align}
Since $\frac{2}{p-2} \notin \N$, we have $\frac{2}{p-2} < M_p$.
Then it follows that 
%\[\frac{2}{p}(M_p+1)-M_p=\frac{2}{p} + \frac{2-p}{p}M_p <0.\]
$\frac{2}{p}(M_p+1)-M_p<0$.
Therefore, we observe from \eqref{eq:1114-2} that
$|\frac{d^{M_p}}{dx^{M_p}}\sech_p{x}|$ %has a singularity at $x= \pm \K_{p} (1) $.
diverges as $x\to\pm \K_{p} (1)$.
% i.e., $|\frac{d^{M_p}}{dx^{M_p}} f_{p,\lambda}|$ has a singularity at $x= \pm T_{p, \lambda}$.
It suffices to consider the integrability around $x= \K_p(1)$ in \eqref{eq:1115-2}.
By \eqref{eq:1114-2} and the change of variables $\eta=a(x)$, this integrability is reduced to whether the following integral is finite for sufficiently small $\delta>0$:
\begin{align*}
\int_{\frac{\pi}{2}-\delta}^{\frac{\pi}{2}} (\cos\eta )^{\big(\tfrac{2}{p}({M_p}+1) -M_p\big)r} \big|h_{M_p}(\sin\eta) \big|^r (\cos\eta)^{-\frac{2}{p}} \,d\eta.
\end{align*}
Since $h_{M_p}(\sin \eta)$ is continuous and satisfies $h_{M_p}(\sin \frac{\pi}{2})\neq0$, the integral is finite if and only if
\[ \Big(\frac{2}{p}(M_p+1) -M_p\Big)r -\frac{2}{p} > -1,\]
which is equivalent to $1 \leq r <  R_p$.
The proof is complete.
\end{proof}

%%%%%%%%%%%%%%%%%%%%%%%%%%%%%%%%%%%%%%
%%%%%%%%%%%%%%%%%%%%%%%%%%%%%%%%%%%%%%
%%%%%%%%%%%%%%%%%%%%%%%%%%%%%%%%%%%%%%
%%%%%%%%%%%%%%%%%%%%%%%%%%%%%%%%%%%%%%
%%%%%%%%%%%%%%%%%%%%%%%%%%%%%%%%%%%%%%
%%%%%%%%%%%%%%%%%%%%%%%%%%%%%%%%%%%%%%

\bibliographystyle{abbrv}
\bibliography{ref_Miura-Yoshizawa-ver6}

\end{document}